\documentclass[a4paper, 11pt]{amsart}

\usepackage{amsmath,amssymb,enumitem,verbatim,stmaryrd,xcolor,microtype}
\usepackage[T1]{fontenc}
\usepackage[utf8]{inputenc}
\usepackage[english]{babel}
\usepackage[top=3.5cm,bottom=3.5cm,left=3.2cm,right=3.2cm]{geometry}
\usepackage[bookmarksdepth=3,linktoc=page,colorlinks,linkcolor={red!80!black},citecolor={red!80!black},urlcolor={blue!80!black}]{hyperref}\newcommand{\arxiv}[1]{\href{http://arxiv.org/pdf/#1}{arXiv:#1}}
\usepackage{tikz}\usetikzlibrary{matrix,arrows,decorations.markings}
\usepackage[all]{xy}%\CompileMatrices\CompilePrefix{xymatrix/diagram}
\usepackage[mathcal]{euscript}                               % use the euscript caligraphic font with \mathscr{...}
\usepackage{mathptmx}                                        % times font
\usepackage{etoolbox}\makeatletter\patchcmd{\@startsection}{\@afterindenttrue}{\@afterindentfalse}{}{}\makeatother    %omit indentation of the first paragraph of a section
\patchcmd{\section}{\scshape}{\bfseries}{}{}\makeatletter\renewcommand{\@secnumfont}{\bfseries}\makeatother           %boldface section and subsection titles (no caption), including numbers
\usepackage[backgroundcolor=yellow,linecolor=yellow,textsize=footnotesize]{todonotes} \makeatletter \providecommand \@dotsep{5} \def\listtodoname{List of Todos} \def\listoftodos{\@starttoc{tdo}\listtodoname} \makeatother %\todo{} for margin notes, supress in pdf with option [disable]

\theoremstyle{plain}
\newtheorem{thm}{Theorem}[section]
\newtheorem{cor}[thm]{Corollary}
\newtheorem{lemma}[thm]{Lemma}
\newtheorem{prop}[thm]{Proposition}
\newtheorem{thmA}{Theorem}  %alphabetic theorem counter: Theorem A, Theorem B, ...

\theoremstyle{definition}
\newtheorem{df}[thm]{Definition}
\newtheorem{rem}[thm]{Remark}
\newtheorem{ex}[thm]{Example}

\usepackage{etoolbox}
\makeatletter
\patchcmd{\@startsection}{\@afterindenttrue}{\@afterindentfalse}{}{}             %omit indentation of the first paragraph of a section
\patchcmd{\part}{\bfseries}{\bfseries\LARGE}{}{}
\patchcmd{\section}{\scshape}{\bfseries}{}{}\renewcommand{\@secnumfont}{\bfseries} %boldface no smallcaps section and subsection titles with numbers
\patchcmd{\@settitle}{\uppercasenonmath\@title}{\large}{}{}
\patchcmd{\@setauthors}{\MakeUppercase}{}{}{}
\addto{\captionsenglish}{} %boldface no smallcaps Abstract
\addto{\captionsenglish}{} %bolface no smallcaps Figure
\addto{\captionsenglish}{} %bolface no smallcaps Table
\makeatother

\usepackage{fancyhdr}

\DeclareRobustCommand{\gobblefour}[4]{}    % Command \SkipTocEntry for surpressing a section title in TOC
\DeclareSymbolFont{sfoperators}{OT1}{bch}{m}{n} \DeclareSymbolFontAlphabet{\mathsf}{sfoperators} \makeatletter\def\operator@font{\mathgroup\symsfoperators}\makeatother % different font for math operators
\DeclareSymbolFont{cmletters}{OML}{cmm}{m}{it}              
\DeclareSymbolFont{cmsymbols}{OMS}{cmsy}{m}{n}
\DeclareSymbolFont{cmlargesymbols}{OMX}{cmex}{m}{n}
\DeclareMathSymbol{\myjmath}{\mathord}{cmletters}{"7C}     \let\jmath\myjmath %Defining the missing commands: \jmath, \amalg and \coprod
\DeclareMathSymbol{\myamalg}{\mathbin}{cmsymbols}{"71}     
\DeclareMathSymbol{\mycoprod}{\mathop}{cmlargesymbols}{"60}\let\coprod\mycoprod
\DeclareMathSymbol{\myalpha}{\mathord}{cmletters}{"0B}     \let\alpha\myalpha %Greek letters from Computer Modern
\DeclareMathSymbol{\mybeta}{\mathord}{cmletters}{"0C}      \let\beta\mybeta
\DeclareMathSymbol{\mygamma}{\mathord}{cmletters}{"0D}     \let\gamma\mygamma
\DeclareMathSymbol{\mydelta}{\mathord}{cmletters}{"0E}     \let\delta\mydelta
\DeclareMathSymbol{\myepsilon}{\mathord}{cmletters}{"0F}   \let\epsilon\myepsilon
\DeclareMathSymbol{\myzeta}{\mathord}{cmletters}{"10}      \let\zeta\myzeta
\DeclareMathSymbol{\myeta}{\mathord}{cmletters}{"11}       \let\eta\myeta
\DeclareMathSymbol{\mytheta}{\mathord}{cmletters}{"12}     \let\theta\mytheta
\DeclareMathSymbol{\myiota}{\mathord}{cmletters}{"13}      \let\iota\myiota
\DeclareMathSymbol{\mykappa}{\mathord}{cmletters}{"14}     \let\kappa\mykappa
\DeclareMathSymbol{\mylambda}{\mathord}{cmletters}{"15}    \let\lambda\mylambda
\DeclareMathSymbol{\mymu}{\mathord}{cmletters}{"16}        \let\mu\mymu
\DeclareMathSymbol{\mynu}{\mathord}{cmletters}{"17}        \let\nu\mynu
\DeclareMathSymbol{\myxi}{\mathord}{cmletters}{"18}        \let\xi\myxi
\DeclareMathSymbol{\mypi}{\mathord}{cmletters}{"19}        \let\pi\mypi
\DeclareMathSymbol{\myrho}{\mathord}{cmletters}{"1A}       \let\rho\myrho
\DeclareMathSymbol{\mysigma}{\mathord}{cmletters}{"1B}     \let\sigma\mysigma
\DeclareMathSymbol{\mytau}{\mathord}{cmletters}{"1C}       \let\tau\mytau
\DeclareMathSymbol{\myupsilon}{\mathord}{cmletters}{"1D}   \let\upsilon\myupsilon
\DeclareMathSymbol{\myphi}{\mathord}{cmletters}{"1E}       \let\phi\myphi
\DeclareMathSymbol{\mychi}{\mathord}{cmletters}{"1F}       \let\chi\mychi
\DeclareMathSymbol{\mypsi}{\mathord}{cmletters}{"20}       \let\psi\mypsi
\DeclareMathSymbol{\myomega}{\mathord}{cmletters}{"21}     \let\omega\myomega
\DeclareMathSymbol{\myvarepsilon}{\mathord}{cmletters}{"22}\let\varepsilon\myvarepsilon
\DeclareMathSymbol{\myvartheta}{\mathord}{cmletters}{"23}  \let\vartheta\myvartheta
\DeclareMathSymbol{\myvarpi}{\mathord}{cmletters}{"24}     \let\varpi\myvarpi
\DeclareMathSymbol{\myvarrho}{\mathord}{cmletters}{"25}    \let\varrho\myvarrho
\DeclareMathSymbol{\myvarsigma}{\mathord}{cmletters}{"26}  \let\varsigma\myvarsigma
\DeclareMathSymbol{\myvarphi}{\mathord}{cmletters}{"27}    \let\varphi\myvarphi

\DeclareMathOperator{\Spec}{Spec}

\DeclareMathOperator{\Sch}{Sch}
\DeclareMathOperator{\OBSch}{OBSch}
\DeclareMathOperator{\Hom}{Hom}
\DeclareMathOperator{\Proj}{Proj}
\DeclareMathOperator{\colim}{colim}

\DeclareMathOperator{\im}{im}
\DeclareMathOperator{\eq}{eq}
\DeclareMathOperator{\coeq}{coeq}

\DeclareMathOperator{\OBAff}{{OBAff}}

\DeclareMathOperator{\Alg}{{Alg}}
\DeclareMathOperator{\An}{{An}}
\DeclareMathOperator{\Trop}{{Trop}}
\DeclareMathOperator{\Bend}{{Bend}}
\DeclareMathOperator{\bend}{{bend}}

\DeclareMathOperator{\Val}{{Val}}
\DeclareMathOperator{\Blpr}{{Blpr}}
\DeclareMathOperator{\OBlpr}{{OBlpr}}
\DeclareMathOperator{\Sets}{Sets}
\DeclareMathOperator{\Rings}{{Rings}}
\DeclareMathOperator{\SRings}{{SRings}}
\DeclareMathOperator{\Mon}{{Mon}}
\DeclareMathOperator{\Halos}{{Halos}}
\DeclareMathOperator{\OBlprSp}{{OBlprSp}}

\newcommand\A{{\mathbb A}}
\newcommand\B{{\mathbb B}}
\newcommand\C{{\mathbb C}}
\newcommand\F{{\mathbb F}}
\newcommand\G{{\mathbb G}}

\newcommand\N{{\mathbb N}}
\renewcommand\P{{\mathbb P}}

\newcommand\R{{\mathbb R}}
\renewcommand\S{{\mathbb S}}
\newcommand\T{{\mathbb T}}
\newcommand\Z{{\mathbb Z}}

\newcommand\cC{{\mathcal C}}
\newcommand\cD{{\mathcal D}}

\newcommand\cF{{\mathcal F}}
\newcommand\cG{{\mathcal G}}

\newcommand\cM{{\mathcal M}}

\newcommand\cO{{\mathcal O}}
\newcommand\cP{{\mathcal P}}

\newcommand\cR{{\mathcal R}}
\newcommand\cS{{\mathcal S}}

\newcommand\cU{{\mathcal U}}
\newcommand\cV{{\mathcal V}}

\newcommand\cZ{{\mathcal Z}}

\newcommand\fm{{\mathfrak m}}
\newcommand\fp{{\mathfrak p}}
\newcommand\fq{{\mathfrak q}}
\newcommand\Fun{{\F_1}}
\newcommand\Funsq{{\F_{1^2}}}

\renewcommand\int{\textup{int}}

\newcommand\id{\textup{id}}

\newcommand\inv{{\textup{inv}}}
\newcommand\proper{{\textup{prop}}}
\newcommand\pos{{\textup{pos}}}

\newcommand\hull{\textup{hull}}
\newcommand\alg{\textup{alg}}
\newcommand\an{\textup{an}}

\newcommand\ob{\textup{ob}}

\newcommand\conic{\textup{conic}}
\newcommand\idem{\textup{idem}}
\newcommand\blue{\textup{blue}}

\newcommand\core{\textup{core}}
\newcommand\padd{\textup{padd}}
\newcommand\mon{\textup{mon}}

\newcommand\ev{\textup{ev}}
\newcommand\trop{\textup{trop}}
\renewcommand\log{\textup{log}}
\newcommand\toric{\textup{toric}}
\newcommand\gp{\textup{gp}}
\newcommand\init{\textup{in}}
\newcommand\mult{\textup{mult}}
\newcommand\univ{\textup{univ}}

\newcommand\Sotimes{\otimes^{\!+}}

\renewcommand\={\equiv}
\renewcommand\geq{\geqslant}
\renewcommand\leq{\leqslant}
\newcommand{\gen}[1]{\langle #1 \rangle}
\newcommand{\bpquot}[2]{#1\!\sslash\!#2}
\newcommand{\bpgenquot}[2]{#1\!\sslash\!\gen{#2}}
\newdir{ >}{{}*!/-5pt/@^{(}}\newcommand{\arincl}[1]{\ar@{ >->}@<-0,0ex>#1} %inclusion arrow for xy-matrix with better spacing
\newcommand{\norm}[1]{| #1 |}

\title{A unifying approach to tropicalization}

\author{Oliver Lorscheid}
\address{\rm University of Groningen, the Netherlands, and IMPA, Rio de Janeiro, Brazil}
\email{{oliver@impa.br}}
\begin{document}

\ \vspace{-0,1cm}

\begin{abstract}
 In this paper, we introduce ordered blueprints and ordered blue schemes, which serve as a common language for the different approaches to tropicalizations and which enhances tropical varieties with a schematic structure. As an abstract concept, we consider a tropicalization as a moduli problem about extensions of a given valuation $v:k\to T$ between ordered blueprints $k$ and $T$. If $T$ is idempotent, then we show that a generalization of the Giansiracusa bend relation leads to a representing object for the tropicalization, and that it has yet another interpretation in terms of a base change along $v$. We call such a representing object a \emph{scheme theoretic tropicalization}.
 
 This theory recovers and improves other approaches to tropicalizations as we explain with care in the second part of this text. 
 
 The Berkovich analytification and the Kajiwara-Payne tropicalization appear as rational point sets of a scheme theoretic tropicalization. The same holds true for its generalization by Foster and Ranganathan to higher rank valuations.
 
 The scheme theoretic Giansiracusa tropicalization can be recovered from the scheme theoretic tropicalizations in our sense. We obtain an improvement due to the resulting blueprint structure, which is sufficient to remember the Maclagan-Rinc\'on weights. 
 
 The Macpherson analytification has an interpretation in terms of a scheme theoretic tropicalization, and we give an alternative approach to Macpherson's construction of tropicalizations.
 
 The Thuillier analytification and the Ulirsch tropicalization are rational point sets of a scheme theoretic tropicalization. Our approach yields a generalization to any, possibly nontrivial, valuation $v:k\to T$ with idempotent $T$ and enhances the tropicalization with a schematic structure.
\end{abstract}

\maketitle
%{\ \vspace{-415pt}\\ \flushright\tiny\bf\ \today\\ Draft for the 3rd arXiv version\\}\vspace{380pt}

\begin{small} \tableofcontents \end{small}

%%%%%%%%%%%%%%%%%%%%%%%%%%%%%%%%%%%%%%%%%%%%%%%%%%%%%%%%%%%%%%%%%%%%%%%%%%%%%%%%%%%%%%%%%%%%%%%%%%%%%%%%%%%%%%%%%%%%%%%%%%%%%%%%%%%%%%%%%%%%%%%%%%%%%%%%%%%%%%%%%%%%%%%%%%%
%%%%%%%%%%%%%%%%%%%%%%%%%%%%%%%%%%%%%%%%%%%%%%%%%%%%%%%%%%%%%%%%%%%%%%%%%%%%%%%%%%%%%%%%%%%%%%%%%%%%%%%%%%%%%%%%%%%%%%%%%%%%%%%%%%%%%%%%%%%%%%%%%%%%%%%%%%%%%%%%%%%%%%%%%%%

\section*{Introduction}
\label{intro}

\subsection*{Outlook to the main results}
The purpose of this paper is the development of a language that allows us to consider the different techniques of tropicalizing a classical scheme within the same framework and that provides a scheme theoretic structure for tropicalizations. 

This sought-after language is that of ordered blueprints and ordered blue schemes. Our theory recovers and extends the following earlier approaches to analytification and tropicalization in terms of functorial constructions:
\begin{itemize}
 \item Berkovich analytification (Thm.\ \ref{thm: Berkovich analytification as rational point set}) and its variant by Thuillier (Thm.\ \ref{thm: Thuillier analytification as rational point set}), considered as topological spaces;
 \item Kajiwara-Payne tropicalization (Thm.\ \ref{thm: Kajiwara-Payne tropicalization as rational point set});
 \item its generalization to higher rank valuations by Foster and Ranganathan (Thm.\ \ref{thm: Foster-Ranganathan analytification as rational point set});
 \item Jeffrey and Noah Giansircusa's ``scheme theoretic tropicalization'' (Thm.\ \ref{thm: Giansiracusa tropicalization as bend});
 \item Maclagan and Rinc\'on's recovery of the weights of the ``classical'' tropicalization from the Giansiracusa tropicalization (Thm.\ \ref{thm: Maclagan-Rincon weights});
 \item Macpherson's analytification, including an extension of his theory to tropicalization (Thm.\ \ref{thm: Macpherson analytification and the bend functor});
 \item Ulirsch's tropicalization of logarithmic schemes (Thm.\ \ref{thm: Ulirsch tropicalization as rational point set}), in good situations.
\end{itemize}
In fact, the structure of an ordered blue scheme carries more information than any of the aforementioned approaches, which fits natural into the context of tropicalizations. We hope that this enriched structure will be useful for future developments in tropical geometry.

\subsection*{History}
Let us begin with a historical outline of tropical geometry, in order to explain the relevance of our results.

In spite of the early works of Bergman (\cite{Bergman71}) and Bieri and Groves (\cite{Bieri-Groves84}), tropical geometry became an active research area only in the early 2000s when it became clear that the combinatorial nature of tropical varieties could be used to study their classical counterparts; for instance, see Mikhalkin's celebrated computation of Gromov-Witten invariants (\cite{Mikhalkin05}). 

Let $k$ be a field and $X$ a closed subvariety of the torus $(k^\times)^n$. From its early days on, the tropicalization $\Trop(X)$ along a logarithmic valuation $v:k^\times\to \R$ was equally understood as an amoeba, as the corner locus of its defining polynomials and as the coordinatewise evaluation of seminorms extending $v$ (\cite{Bieri-Groves84}, \cite{Einsiedler-Kapranov-Lind06}). It was known that $\Trop(X)$ can be endowed with the structure of a finite polyhedral complex in $\R^n$, whose top dimensional polyhedra carry weights that satisfy a certain balancing condition (\cite{Bieri-Groves84}, \cite{Speyer05}). It was also clear that the tropicalization of $X$ could be compactified via an embedding of $(k^\times)^n$ into a complete toric variety (\cite{Mikhalkin00}, \cite{Nishinou-Siebert06}, \cite{Speyer05}). 

However, it took some years till this knowledge found a clear formulation in the independent works of Kajiwara (\cite{Kajiwara08}) and Payne (\cite{Payne09}) who defined the tropicalization of a closed subvariety $X$ of a toric variety along a nonarchimedean valuation as a quotient of the Berkovich space of $X$, which can be understood as a stack quotient (\cite{Ulirsch14}).

From this point on, the understanding of tropicalization was broadened in different directions. One important source of inspiration were skeleta of Berkovich spaces, as introduced by Berkovich (\cite{Berkovich99}). In the situation of a variety over a discretely valued field, a semistable model over the discrete valuation ring defines a skeleton for the Berkovich space. While a strict correspondence between skeleta from semistable models and tropica\-lizations holds only in special situations, a generalized framework of skeleta for semistable pairs illuminated this relation; cf.\ Tyomkin (\cite{Tyomkin12}), Baker, Payne and Rabinoff (\cite{Baker-Payne-Rabinoff11}, \cite{Baker-Payne-Rabinoff13}), and Gubler, Rabinoff and Werner (\cite{Gubler-Rabinoff-Werner14}, \cite{Gubler-Rabinoff-Werner15}). 

An important variant of skeleta for semistable models is Thuillier's theory of skeleta for toroidal embeddings over trivially valued fields (\cite{Thuillier07}). Abramovich, Caporaso and Payne interpreted these skeleta as tropicalizations (\cite{Abramovich-Caporaso-Payne12}), and Ulirsch (\cite{Ulirsch13}) clarified this process in terms of a tropicalization associated to fine and saturated log schemes, which passes through an associated Kato fan and the local tropicalization of Popescu-Pampu and Ste\-panov (\cite{Popescu-Pampu-Stepanov13}). Ulirsch's tropicalization of fine and saturated Zariski log schemes coincides with the approach of Gross and Siebert (\cite{Gross-Siebert13}) in their study of logarithmic Gromov-Witten invariants.

A recent variant of the Kajiwara-Payne tropicalization replaces the valuation $v:k^\times\to \R$ by a valuation $v:k^\times\to \R^n$ of higher rank. This was first considered by Banerjee (\cite{Banerjee15}) in the case of higher local fields, and the idea was taken up and generalized by Foster and Ranganathan (\cite{Foster-Ranganathan15}), who showed that higher rank tropicalizations reflect certain properties of classical varieties over $k$.

With the progress of generalized scheme theory, often coined as $\Fun$-geometry, a theory of semiring schemes and, in particular, schemes over the tropical numbers became available; see the work of Durov (\cite{Durov07}), To\"en and Vaqui\'e (\cite{Toen-Vaquie09}), and the author (\cite{blueprints1}). Jeff and Noah Giansiracusa used this theory in the case of closed subschemes of toric varie\-ties to enhance the tropicalization with a schematic structure (\cite{Giansiracusa13}). At the same time, Macpherson endowed such a tropicalization with the structure of an analytic space (\cite{Macpherson13}). Strikingly, Maclagan and Rinc\'on showed that the schematic structure of the tropicalization together with the embedding into an ambient torus encodes the structure of the tropical variety as a balanced weighted polyhedral complex (\cite{Maclagan-Rincon14}).

In the following, we will explain how to put these different approaches to tropicalizations on a common footing via ordered blueprints.

\subsection*{From coordinates to blueprints}
Grosso modo, a tropicalization of a $k$-scheme $X$ is the image of certain chosen coordinates of $X$ under a valuation $v$ of the field $k$. The coordinates for $X$ can be given by different means: an embedding of $X$ into affine space or into a toric variety; a simple normal crossing divisor on $X$; a simple toroidal embedding; a fine and saturated log structure for $X$.

For simplicity, let $X=\Spec R$ be an affine $k$-scheme. The choice of coordinates singles out a multiplicative subset $A$ of $R$. In case of an closed immersion $\iota:X\to \Spec k[A_0]$ into a toric variety, $A$ equals the set of elements of the form $\Gamma\iota(c\cdot a)\in R$ where $c\in k$ and $a\in A_0$ and $\Gamma\iota:k[A_0]\to R$ is the map between the respective global sections. In case of the complement $U\subset X$ of a simple normal crossing divisor, or, more general, a simple toroidal embedding $U\subset X$, the multiplicative set $A$ equals the intersection $R\cap\cO_X(U)^\times$. In case of a fine and saturated log structure $\alpha:\cM_X\to\cO_X$, the multiplicative subset is $A=\cM_X(X)$.

Let $B^+$ be the subring of $R$ that is generated by $A$. Then the inclusion $A\subset B^+$ is a blueprint in the sense of \cite{blueprints1}. In this paper, we explain how to tropicalize $X$ with respect to the choice of blueprint $B=(A\subset B^+)$, and we show that this recovers the previously mentioned concepts of tropicalization.

\subsection*{Analytification as a base change}
Before we enter the theory of ordered blueprints, we want to explain the underlying idea that is inspired by Paugam's approach \cite{Paugam09} to analytic geometry. The following is a simplified account of this theory.

An ordered semiring is a (commutative) semiring $R$ (with $0$ and $1$) together with a partial order $\leq$ that is additive and multiplicative. A \emph{subadditive homomorphism of ordered semirings} is an order preserving multiplicative map $f:R_1\to R_2$ with $f(0)=0$, $f(1)=1$ and $f(a+b)\leq f(a)+f(b)$. 

This allows us to perform the following gedankenexperiments. We consider rings as trivially ordered semirings. Note that a subadditive homomorphism between trivially ordered semirings is always additive, which means that subadditive homomorphisms of rings are homomorphisms.

If we endow the semiring $\R_{\geq0}$ with its natural total order, then a seminorm $v:k\to\R_{\geq0}$ on a ring $k$ is nothing else than a subadditive homomorphism of ordered semirings. If we exchange the usual addition of $\R_{\geq0}$ by the maximum operation, which yields the ordered semiring $\T$ of tropical numbers, then a subadditive homomorphism $v:k\to \T$ is nothing else than a nonarchimedean seminorm on $k$.

Given a field $k$ with a nonarchimedean absolute value $v:k\to\T$ and an affine $k$-scheme $X=\Spec R$, the Berkovich analytification $X^\an$ equals the set of all seminorms $w:R\to \T$ that extend $v$, i.e.\ the set of all subadditive homomorphisms $w$ that make the diagram
\[
 \xymatrix@R=1pc@C=6pc{k \ar[r]^v \ar[d] & \T \ar[d]_\id \\ R \ar@{-->}[r]^w &\T}
\]
commute. This suggests the interpretation of the Berkovich space $X^\an$ as the set $(X\otimes_k\T)(\T)$ of $\T$-rational points of the base change $X\otimes_k\T=\Spec(R\otimes_k\T)$ of $X$ along $v$.

The problem is that it is not clear if tensor products exist in general and how to construct them. We circumvent this problem by considering the larger category of ordered blueprints, which contains tensor products naturally; cf.\ Remarks \ref{rem: halos} and \ref{rem: tropicalization as base change} for details.

\subsection*{Ordered blueprints}
In the following exposition, we present a different, but equivalent, definition of ordered blueprints from the main text of this paper. For the precise connection between these two viewpoints, cf.\ Remark \ref{rem: ordered blueprints as monoids in ordered semirings}.

An \emph{ordered blueprint} is an ordered semiring $B^+$ together with a multiplicatively closed subset $B^\bullet\subset B^+$ of generators of $B^+$ that contains $0$ and $1$. A \emph{morphism of ordered blueprints} is an order preserving homomorphism of semirings that sends generators to generators. This defines the category $\OBlpr$ of ordered blueprints, which turns out to be closed under small limits and colimits and, in particular, has a tensor product. We write $B$ for the ordered blueprint $B^\bullet\subset B^+$.

Some examples are the following. A semiring $R$ can be considered as the ordered blueprint $B=(R\subset R)$, together with the trivial order on $R$. We call ordered blueprints $B$ whose semiring is trivially ordered \emph{algebraic}, and we can associate with every ordered blueprint $B$ its \emph{algebraic core $B^\core$} which results from replacing the order of $B^+$ by the trivial order. 

We denote the algebraic semiring $\R_{\geq0}$ together with its natural total order by $\R_{\geq0}^\pos$. Similarly, we denote the algebraic semiring $\T$ of tropical numbers together with its natural total order by $\T^\pos$. More generally, we can define for every ordered blueprint $B$ its associated \emph{totally positive blueprint $B^\pos$}, which is $B$ together with the order generated by the relations $a\leq b$ whenever there is an $c\in B^+$ such that $a+c\leq b$ in $B^+$. Note that it comes with a morphism $B\to B^\pos$.

The name stems from the fact that $0\leq a$ for every $a\in B^\pos$. Note that in general, the order of $B^\pos$ might identify different elements of $B^+$. For instance, $R^\pos$ is trivial if $R$ is a ring. For an idempotent semiring $R$, however, the totally positive blueprint $R^\pos$ carries the natural partial order of $R$ and we recover $R$ as $(R^\pos)^\core$.

\subsection*{Valuations}
In agreement with the naive approach explained above, totally positive blue\-prints will play the role of the targets of valuations. Our interpretation of the domains of valuations will pass through the following construction.

Let $B=(B^\bullet\subset B^+)$ be an ordered blueprint. We define its associated \emph{monomial blueprint} as the ordered blueprint $B^\bullet\subset B^{\mon,+}$ where $B^{\mon,+}$ is the monoid semiring $\N[B^\bullet]$ of $B^\bullet$ modulo the identification of the respective zeros of $B^\bullet$ and $\N[B^\bullet]$. The partial order of $B^{\mon,+}$ is generated by the \emph{left monomial relations} $a\leq\sum b_j$ with $a,b_j\in B^\bullet$ whenever this holds in $B^+$. Note that the identity on $B^\bullet$ induces a morphism $B^\mon\to B$.

With these definitions at hand, we see that a map $v:B\to \R_{\geq0}$ from a ring $B$ to the non-negative reals is a seminorm if and only if the composition
\[
 B^\mon \ \longrightarrow \ B \ \stackrel v \longrightarrow \ \R_{\geq0} \ \longrightarrow \ \R_{\geq0}^\pos
\]
is a morphism of ordered blueprints. A map $v:B\to \T$ is a nonarchimedean seminorm if and only if $B^\mon\to\T^\pos$ is a morphism. 

This motivates our general definition of a \emph{valuation} as a multiplicative map $v:B\to T$ between ordered blueprints $B$ and $T$ such that $B^\mon\to T^\pos$ is a morphism.

\subsection*{Scheme theory}
It is possible to extend ordered blueprints to a geometric category $\OBSch$ of \emph{ordered blue schemes} in terms of topological spaces with a sheaf in $\OBlpr$. This comes with a contravariant functor  $\Spec:\OBlpr\to\OBSch$ that associates with an ordered blueprint the space of its prime ideals.

In particular, we can consider the sets $X(T)$ of $T$-rational points of an ordered blue $T$-scheme $X$. If $T$ carries a topology, then $X(T)$ becomes a topological space with respect to the fine topology, which was introduced by the author and Salgado in \cite{Lorscheid-Salgado14}.

\subsection*{Tropicalization and the bend relation}
To avoid technicalities concerning scheme theory, we restrict ourselves to affine schemes in the following presentation of our results. This suffices to explain the essential content of our theory since tropicalization is a process that commutes with restrictions to affine patches and a generalization to geometry is achieved by standard arguments in most situations.

Let $k$ be an ordered blueprint and $B$ an \emph{ordered blue $k$-algebra}, i.e.\ a morphism $k\to B$. Let $v:k\to T$ be a valuation. Consider the functor $\Val_v(B,-)$ that associates with an ordered blue $T$-algebra $S$ the set of valuations $w:B\to S$ that extend $v$, which means that the diagram
\[
 \xymatrix@R=1pc@C=6pc{B \ar[r]^w & S \\ k \ar[u] \ar[r]^v & T \ar[u] }
\]
commutes. Let $X=\Spec B$. A \emph{tropicalization of $X$ along $v$} is an ordered blue $T$-scheme that represents $\Val_v(B,-)$.

In complete generality, a tropicalization of $X$ along $v$ does not exist. However, in the following two situations, we can prove its existence in terms of an explicit description. Let $X^\mon=\Spec B^\mon$.

\begin{thmA}\label{thmA}
 If $T$ is totally positive, then $X^\mon\otimes_{k^\mon}T$ is a tropicalization of $X$ along $v$.
\end{thmA}

This is Theorem \ref{thm: tropicalization for totally positive tropical base}. This theorem realizes the idea that the tropicalization is the base change along a valuation. In particular, if $k$ and $B$ are monomial, then $X\otimes_k T$ is a tropicalization of $X$ along $v$.

In order to formulate the second existence theorem for tropicalizations, we have to introduce the bend, which is a generalization of the Giansiracusa tropicalization (\cite{Giansiracusa13}) to the context of ordered blueprints. The \emph{bend of an ordered blueprint $B$ along $v$} is the ordered blue $T$-algebra $\Bend_v(B)$ whose ordered semiring $\Bend_v(B)^+$ and whose underlying monoid $\Bend_v(B)^\bullet$ are defined as follows. The semiring $\Bend_v(B)^+$ is the quotient of the semigroup semiring $T^+[B^\bullet]$ by the relations of the form 
\[\textstyle
 (v(c)t)\cdot a \ = \ t\cdot(c.a) \qquad \text{and} \qquad t\cdot a+\sum t\cdot b_j \ = \ \sum t\cdot b_j
\]
with $c\in k^\bullet$, $t\in T^\bullet$, $a, b_j\in B^\bullet$ and $a\leq\sum b_j$ in $B^+$. The monoid $\Bend_v(B)^\bullet$ consists of the classes of elements of the form $t\cdot a$ in $\Bend_v(B)^+$ and the order of $\Bend_v(B)^+$ is generated by the order of $T$. 

We say that $T$ is \emph{idempotent} if $T^+$ is an idempotent semiring. For $X=\Spec B$, we define $\Bend_v(X)=\Spec\big(\Bend_v(B)\big)$. The following is Theorem \ref{thm: tropicalization for idempotent base}.

\begin{thmA}\label{thmB}
 If $T$ is idempotent, then there is a canonical isomorphism 
 \[
  \Bend_v(X) \quad \stackrel\sim\longrightarrow \quad (X^\mon\otimes_{k^\mon}T^\pos)^\core\otimes_{T^\core}T,
 \]
 and $\Bend_v(X)$ is a tropicalization of $X$ along $v$.
\end{thmA}

As a consequence, a tropicalization of $X$ along $v$ is algebraic and equal to the spectrum of $(B^\mon\otimes_{k^\mon}T^\pos)^\core$ if $T$ is idempotent and algebraic. If $T$ is idempotent and totally positive, then $\Bend_v(B)=B^\mon\otimes_{k^\mon}T$.

In the following, we will explain how the different concepts of analytification and tropicalization of classical schemes fit into the framework of tropicalizations of ordered blueprints and ordered blue schemes.

\subsection*{Berkovich analytification and Kajiwara-Payne tropicalization}
Let $k$ be a field and $v:k\to \T$ a valuation. Let $Y=\Spec R$ a $k$-scheme and $\iota:Y\to \Spec k[A_0]$ a closed embedding into a toric $k$-variety. 

The restriction of a seminorm $w:k[A_0]\to \T$ in $Y^\an$ to $A_0$ is a multiplicative map $A_0\to \T$. If we define $\Hom(A_0,\T)$ with the real topology coming from $\T$, then this restriction defines a continuous map $\trop_{v,\iota}^{KP}:Y^\an\to \Hom(A_0,\T)$. The \emph{Kajiwara-Payne tropicalization of $Y$} is the image $\Trop_{v,\iota}^{KP}(Y)=\trop_{v,\iota}^{KP}(Y^\an)$ under this map.

The associated blueprint $B$ is defined as $B^+=R$ and $B^\bullet=\{\Gamma\iota(ca)|c\in k^\bullet, a \in A_0\}$ where $\Gamma\iota:k[A_0]\to R$ is the map of global sections induced by $\iota$. 

The inclusion $B\to R$ defines a morphism $\beta:Y\to Z$ of ordered blue schemes where $Z=\Spec B$. The following summarizes Theorems \ref{thm: Berkovich analytification as rational point set} and \ref{thm: Kajiwara-Payne tropicalization as rational point set}.

\begin{thmA}\label{thmC}
  The Berkovich space $Y^\an$ is naturally homeomorphic to $\Bend_v(Y)(\T)$, the Kajiwara-Payne tropicalization $\Trop_{v,\iota}^{KP}(Y)$ is naturally homeomorphic to $\Bend_v(Z)(\T)$ and the diagram 
 \[
  \xymatrix@R=1pc@C=6pc{Y^\an \ar[r]^{\trop_{v,\iota}^{KP}} \ar[d]^\simeq & \Trop_{v,\iota}^{KP}(Y) \ar[d]^\simeq \\ \Bend_v(Y)(\T) \ar[r]^{\Bend_v(\beta)(\T)} & \Bend_v(Z)(\T)}
 \]
 of continuous maps commutes.
\end{thmA}

\subsection*{Foster-Ranganathan tropicalization}
Let $\T^{(n)}=\R_{>0}^n\cup\{0\}$ be the idempotent semiring with componentwise multiplications and whose addition is defined as taking the maximum with respect to the lexicographical order. With respect to the order topology, it is a topological Hausdorff semifield. 

Let $k$ be a field, endowed with a higher rank valuation $v:k\to \T^{(n)}$. Let $Y=\Spec R$ be an affine $k$-scheme and $\iota:Y\to\Spec[A_0]$ a closed immersion into a toric $k$-variety.

Replacing $\T$ by $\T^{(n)}$ in the definitions of the Berkovich analytification and the Kajiwara-Payne tropicalization yields the \emph{Foster-Ranganathan analytification $\An_v^{FR}(Y)$ of $Y$ along $v$} and the \emph{Foster-Ranganathan tropicalization $\Trop_{v,\iota}^{FR}(Y)$ of $Y$ along $v$ with respect to $\iota$}, respectively.

Let $B$ be the blueprint associated with $\iota$ as defined above and $Z=\Spec B$. Let $Y\to Z$ be the induced morphism of blue $k$-schemes. The following is Theorem \ref{thm: Foster-Ranganathan analytification as rational point set}.

\begin{thmA}\label{thmD}
 The Foster-Ranganathan analytification $\An^{FR}_v(Y)$ is naturally homeomorphic to $\Bend_v(Y)(\T^{(n)})$, the Foster-Ranganathan tropicalization $\Trop^{FR}_{v,\iota}(Y)$ is naturally homeomorphic to $\Bend_v(Z)(\T^{(n)})$ and the diagram
 \[
   \xymatrix@R=1pc@C=6pc{\An_v^{FR}(Y) \ar[r]^{\trop_{v,\iota}^{FR}} \ar[d]^\simeq & \Trop_{v,\iota}^{FR}(Y) \ar[d]^\simeq \\ \Bend_v(Y)(\T^{(n)}) \ar[r]^{\Bend_v(\beta)(\T^{(n)})} & \Bend_v(Z)(\T^{(n)})}
 \]
 of continuous maps commutes.
\end{thmA}

\subsection*{Giansiracusa tropicalization and Maclagan-Rinc\'on weights}
We remarked already that the bend of an ordered blueprint is a generalization of the Giansiracusa tropicalization from \cite{Giansiracusa13}. The precise relation is as follows.

Let $k$ be a ring and $v:k\to T$ be a valuation into a totally ordered idempotent semiring $T$. Let $Y=\Spec R$ be a $k$-scheme. Let $A_0$ be a monoid and $\eta:A_0\to R$ a multiplicative map such that $k[A_0]\to R$ is surjective. The \emph{Giansiracusa tropicalization $\Trop_{v,\eta}^{GG}(Y)$ of $Y$ with respect to $v$ and $\eta$} is the spectrum of the semiring 
\[\textstyle
 \Trop_{v,\eta}^{GG}(R) \quad = \quad T[A_0] \ \bigl/ \ \bigl\{ \ a+\sum b_j\=\sum b_j \ \bigl| \ \eta(a)+\sum \eta(b_j)=0 \text{ in }R \ \bigr\}.
\]

The blueprint associated with $\eta$ is $B=(A\subset R)$ where $A=\{c.\eta(a)\in R|c\in k, a\in A_0\}$. The following is Theorem \ref{thm: Giansiracusa tropicalization as bend}.

\begin{thmA}\label{thmE}
 There is a canonical isomorphism $\Trop_{v,\eta}^{GG}(R)\simeq \Bend_v(B)^+$ of semirings.
\end{thmA}

The Giansiracusa tropicalization $\Trop_{v,\eta}^{GG}(Y)$ comes with a closed embedding into the toric $\T$-scheme $\Spec \T[A_0]$. Maclagan and Rinc\'on (\cite[Thm.\ 1.2]{Maclagan-Rincon14}) show that the structure of the tropical variety $\Trop(Y)=\Trop_{v,\eta}^{GG}(Y)(\T)$ as a weighted polyhedral complex can be recovered from this embedding, assuming the following context: $v:k\to \T$ is a valuation with dense image such that the value group $v(k^\times)$ lifts to $k^\times$ and assume that $Y=\Spec R$ is an equidimensional closed $k$-subvariety of $\G_{m,k}^n$, which corresponds to a multiplicative map $\eta:A_0\to R$ where $A_0=\{X_1^{e_1}\dotsb X_n^{e_n}|(e_1,\dotsc,e_n)\in\Z^n\}$. 

We show that in this situation, the structure of a weighted polyhedral complex can still be recovered from the weaker structure of the associated blue $\T$-scheme $\Bend_v(Z)$. More precisely, we exhibit an explicit formula for the \emph{Maclagan-Rinc\'on weight} $\mu(w)$ of a $\T$-rational point $w$ of $\Bend_v(Z)$ and show the following in Theorem \ref{thm: Maclagan-Rincon weights}.

\begin{thmA}\label{thmF}
 Let $\sigma$ be a top dimensional polyhedron of $\Trop(Y)=\Bend_v(Z)(\T)$. Then $\mult(\sigma)=\mu(w)$ for every $w$ in the relative interior of $\sigma$.
\end{thmA}

\subsection*{Macpherson analytification}
Let $k$ be a ring and $B$ a $k$-algebra. The \emph{Macpherson analytification of $B$ over $k$} is the idempotent semiring $\An(B,k)$ of finitely generated $k$-submodules $M_1$ and $M_2$ of $B$ with respect to the addition $M_1+M_2$ and the multiplication $M_1\cdot M_2$ given by elementwise operations. The semiring $\An(B,k)$ represents the functor $\Val(B,k;-)$ that associates with an idempotent semiring $T$ the set of all valuations $v:B\to T$ that are integral on $k$, i.e.\ $v(c)+1=1$ for all $c\in k$.

This concept can be generalized to any ordered blue $k$-algebra $B$ over an ordered blueprint $k$. A $k$-span of $B$ is a subset $M$ of $B$ that is stable under multiplication by $k^\bullet$ and contains all $b\in B$ for which there are elements $a_i\in M$ and a relation $b\leq \sum a_i$ in $B^+$. A $k$-span is finitely generated if it contains a finite subset such that there is no smaller $k$-span containing it. We define $\An(B,k)$ as the idempotent semiring of all finitely generated $k$-spans of $B$. 

The definition of $\Val(B,k;-)$ extends to this setting as a functor on idempotent semirings, which are the same as $\B$-algebras where $\B$ is the Boolean semifield. 

In order to describe $\An(B,k)$ as a bend, we define $B^\mon_{k\leq 1}$ as the ordered blueprint $B^\mon$ together with the order that contains all relations of $B^\mon$ together with the relations $a.1\leq 1$ for $a\in k$. Define $\Fun=(\{0,1\}\subset\N)$, together with the trivial order. There is a unique valuation $v_0:\Fun\to \B$, given by $v_0(0)=0$ and $v_0(a)=1$ for $a>0$. The following is Theorem \ref{thm: Macpherson analytification and the bend functor}.

\begin{thmA}\label{thmG}
 There is a canonical isomorphism $\An(B,k)\simeq \Bend_{v_0}(B^\mon_{k\leq1})^+$ of semirings and $\An(B,k)$ represents $\Val(B,k;-)$. 
\end{thmA}

Let $v:k\to T$ be a valuation of a ring $k$ into an idempotent semiring and $\cO_k=\{a\in k|v(a)+1=1\}$, which is a subring of $k$. As a consequence of Theorem \ref{thmG}, we obtain $\Bend_v(B)^+\simeq \An(B,\cO_k)\otimes_{\An(k,\cO_k)}T$ in Corollary \ref{cor: the bend as Macpherson analytification}. This provides an alternative to Macpherson's original construction of tropicalizations via nonarchimedean analytic geometry, cf.\ \cite[section 7.3]{Macpherson13}.

\subsection*{Thuillier analytification and Ulirsch tropicalization}
Let $k$ be a field and $v:k\to \cO_\T$ the trivial valuation where $\cO_\T$ is the subsemiring $\{a\in\T|a+1=1\}$ of $\T$. We consider $\cO_\T$ together with its topology as a subset of $\R$. Let $X=\Spec R$ be a $k$-scheme. The Thuillier analytification $X^\beth$ is the set of all extensions $w:R\to\cO_\T$ of $v$ to $R$, together with the topology induced by $\cO_\T$.

Let $\alpha:\cM_X\to\cO_X$ be a fine and saturated log structure for $X$ without monodromy. Then there is a universal morphism $(X,\cM_X/\cM_X^\times)\to F_X$ of monoidal spaces into a Kato fan $F_X$, and this morphism induces a continuous map $\trop_\alpha^U:X^\beth \to\overline\Sigma_X$ into the extended cone complex $\overline\Sigma_X$ of $\cO_\T$-rational points of the Kato fan $F_X$. The Ulirsch tropicalization of a closed $k$-subscheme $Y=\Spec S$ of $X$ is the image $\Trop_{\alpha,\iota}^U(Y)=\trop_\alpha^U(Y^\beth)$ in $\overline\Sigma_X$ where $\iota$ refers to the closed immersion $\iota:Y\to X$.

For the sake of simplifying this exposition, we assume that the Kato fan is affine. We define the associated blue $k$-scheme $Z$ as the spectrum of the blueprint $B$ where $B^\bullet$ is the image of $\cM_X(X)$ under $\Gamma\iota:R\to S$ and $B^+$ is the subsemiring of $S$ generated by $B^\bullet$. The blue scheme $Z$ comes together with a morphism $\beta:Y\to Z$. Then the following summarizes Theorems \ref{thm: Thuillier analytification as rational point set}, \ref{thm: Ulirsch tropicalization as rational point set} and \ref{thm: recovering the Kato fan}.

\begin{thmA}\label{thmH}
 The Thuillier space $X^\beth$ is naturally homeomorphic to $\Bend_v(X)(\cO_\T)$, the Ulirsch tropicalization $\Trop_{\alpha,\iota}^U(Y)$ is naturally homeomorphic to the topological space $\Bend_v(Z)(\cO_\T)$ and the diagram
 \[
  \xymatrix@R=1pc@C=6pc{Y^\beth \ar[d]^\simeq\ar[r]^{\trop_{\alpha,\iota}^U(Y)} & \Trop_{\alpha,\iota}^U(Y)\ar[d]^\simeq \\ \Bend_v(Y)(\cO_\T) \ar[r]^{\Bend_v(\beta)(\cO_\T)} & \Bend_v(Z)(\cO_\T)}
 \]
 of continuous maps commutes. If $Y=X$ and $\alpha:\cM_X\to\cO_X$ is a monomorphism of sheaves, then we can recover the Kato fan $F_X$ and the embedding $\Trop_{\alpha,\iota}^U(X)\to\overline\Sigma_X$ from $\Bend_v(Z)$.
\end{thmA}

Note that the restriction to the trivial valuation $v:k\to\cO_T$ is caused by the following technical obstruction: the Ulirsch tropicalization relies on the choices of local sections to $\cM_X\to\cM_X/\cM_X^\times$, and this contribution can be avoided since $\cO_\T^\times=\{1\}$. Passing from the fine and saturated log structure to the associated blue scheme avoids these choices and overcomes any restrictions on the valuation.

\subsection*{Conclusion}
The scheme theoretic tropicalization in terms of blue schemes provides a framework that embraces all other concepts of tropicalization considered in this paper, up to some technical restrictions that we address below. This theory extends the different generalizations of the Kajiwara-Payne tropicalization commonly into all directions, with exception of the restrictions mentioned below. In particular, this means that the scheme theoretic structure of the Giansiracusa tropicalization extends to the context of Thuillier analytification and Ulirsch tropicalization of fine and saturated log schemes with respect to any valuation $v:k\to T$ into idempotent $T$.

Another more subtle improvement is the following. The tropicalization of a blue scheme comes with the structure of a blue scheme. This additional structure determines the tropical variety as a topological space, and in the case of a closed subscheme of a torus, it encodes the weights that appear if the tropical variety is gets identified with a polyhedral weighted complex. This has the consequence that we can detach the tropicalization from its ambient space like a toric variety or an extended cone complex.

\subsection*{Technical restrictions}
In this text, we do not pursue a theory of \'etale morphisms for blue schemes. Therefore we restrict ourselves to Zariski log schemes in our treatment of Ulirsch tropicalization. We further assume that the continuous map $\chi:X\to F_X$ from the log scheme $X$ to its Kato fan satisfies that the inverse image of an affine open subset is affine.

We expect that these restrictions are not essential, but we leave the treatment of a more general theory and, in particular, \'etale morphisms for ordered blue schemes to future investigations.

\subsection*{Content overview} This text is divided into two parts. The first part introduces ordered blueprints and ordered blue schemes. The second part applies this theory to tropicalization.

The first part contains the following sections. After settling some conventions for this paper in section \ref{section: conventions}, we introduce ordered blueprints and various subcategories and functorial constructions in section \ref{section: ordered blueprints}. In section \ref{section: valuations}, we explain the relation of our notion of valuations to seminorms and Krull valuations. In section \ref{section: Ordered blue schemes}, we review and extend the theory of blue schemes to the realm of ordered blueprints, which provides the scheme theoretic background for the constructions in the second part of the paper. In section \ref{section: endofunctors and base extension to semiring schemes}, we extend several endofunctors on ordered blueprints to ordered blue schemes and we introduce the base extension to semiring schemes. In section \ref{section: rational points}, we explain how a topology on an ordered blueprint $T$ yields a topology on the set $X(T)$ of $T$-rational points. 

Note that the step from blueprints to ordered blueprints is subtle, in the sense that we merely relax an axiom, yet powerful, in the sense that it comes with various new aspects. Much of sections \ref{section: ordered blueprints} and \ref{section: valuations} is dedicated to a discussion of these novelties. Much of sections \ref{section: Ordered blue schemes} and \ref{section: endofunctors and base extension to semiring schemes} follows our previous developments for blue schemes. Topologies for rational point sets (section \ref{section: rational points}) are a novel aspect to the general theory.

The second parts starts with section \ref{section: Scheme theoretic tropicalization}, which introduces the general concept of tropicalizing an ordered blue scheme along a valuation and prove the central results Theorems \ref{thmA} and \ref{thmB}. In the subsequent sections, we explain the relation to other concepts of analytifications and tropicalizations and prove Theorems \ref{thmC}--\ref{thmH}. Since the section headers are self-explanatory, we would like to refer the reader to the table of contents for finding the corresponding results.

\subsection*{Acknowledgements} 
I would like to thank Sam Payne for organizing the meeting Algebraic Foundations for Tropical Geometry in May 2014. I would like to thank all participants for our discussions during the workshop. My particular thanks go to Matt Baker and Andrew Macpherson for their patient explanations on analytic geometry and skeleta; to Jeff and Noah Giansiracusa for our conversations on scheme theoretic foundations; to Diane Maclagan for sharing her ideas on tropical schemes and her comments on a previous version of  this text; and to Martin Ulirsch for our discussions on the connection between blue schemes and log schemes, and for his careful corrections of a previous version. I would like to thank Ethan Cotterill for bringing several publications to my attention and for his help with various questions on tropical geometry. I would like to thank Walter Gubler, Joseph Rabinoff and Annette Werner for their explanations on skeleta, which led me to the conclusion that the theory of this paper is not yet sufficiently developed to explain skeleta as rational point sets of underlying schemes. I would like to thank Tyler Foster and Dhruv Ranganathan for their explanations on higher rank analytifications and tropicalizations. I would like to thank an anonymous referee for his or her thoughtful report and valuable feedback on the paper.

%%%%%%%%%%%%%%%%%%%%%%%%%%%%%%%%%%%%%%%%%%%%%%%%%%%%%%%%%%%%%%%%%%%%%%%%%%%%%%%%%%%%%%%%%%%%%%%%%%%%%%%%%%%%%%%%%%%%%%%%%%%%%%%%%%%%%%%%%%%%%%%%%%%%%%%%%%%%%%%%%%%%%%%%%%%
%%%%%%%%%%%%%%%%%%%%%%%%%%%%%%%%%%%%%%%%%%%%%%%%%%%%%%%%%%%%%%%%%%%%%%%%%%%%%%%%%%%%%%%%%%%%%%%%%%%%%%%%%%%%%%%%%%%%%%%%%%%%%%%%%%%%%%%%%%%%%%%%%%%%%%%%%%%%%%%%%%%%%%%%%%%

\part{Ordered blueprints}

In the first part, we set up the theory of ordered blueprints and ordered blue schemes. The category of ordered blueprints recovers well-known objects as ordered semirings and monoids as well as blueprints, halos, hyperrings and sesquiads. Several constructions of endofunctors allow us to talk about seminorms and valuations in terms of morphisms. Section \ref{subsection: overview of subcategories} contains an illustration of the relevant subcategories of the category of ordered blueprints.

The spectrum of an ordered blueprint is based on the notion of a prime $k$-ideal, which yields a topological space together with a structure sheaf. This lets us define an ordered blue scheme as a so-called ordered blueprinted space that is covered by spectra of ordered blueprints. After explaining how a topology for an ordered blueprint $T$ induces a topology for the set $X(T)$ of $T$-rational points of an ordered blue scheme $X$, we briefly sketch how the approach to ordered blue schemes is connected to To\"en and Vaqui\'e's relative schemes, and comment on the gap originating from different Grothendieck topologies on the category of ordered blueprints.

\section{Conventions}
\label{section: conventions}

In this text we use the following conventions. A \emph{monoid} is a multiplicatively written commutative semigroup $A$ with unit element $1$ and a morphism of monoids is a multiplicative map that maps $1$ to $1$. A \emph{monoid with zero} is a monoid $A$ with an additional element $0$ that satisfies $0\cdot a=0$ for all $a\in A$. A morphism of monoids with zero is a morphism of monoids that maps $0$ to $0$. We denote the category of monoids with zero by $\Mon$.

A \emph{semiring} is always commutative with $0$ and $1$. A ring is a semiring with an additive inverse $-1$ of $1$. An idempotent semiring is a semiring with $1+1=1$.

The reader that is familiar with blueprints will find that the definition in section \ref{subsection: algebraic blueprints} is equivalent to the definitions in other texts on blueprints, like \cite{Lorscheid17} and \cite{blueprintedview}, with the exception of \cite{blueprints1} where a blueprint in the sense of this text is called a \emph{proper blueprint with zero}.

As a last point, we like to draw the reader's attention to the following inconsistency with standard notation, as used in the introduction. Tensor products and free objects of (ordered) semirings, considered in the category of (ordered) blueprints, are not (ordered) semirings. Our convention is to use the according standard symbols for the constructions \emph{in the category of ordered blueprints} and to refer to the corresponding construction inside the category of ordered semirings with a superscript ``$+$''. This applies to the following notations.

Given a blueprint $B$ and a monoid with zero $A$, we write $B[A]$ for the free blueprint over $B$, whose underlying set is $\{0\}\cup\{b\cdot a|b\in B, a\in A\}$. We write $B[A]^+$ for the generated semiring. For instance, while we denote by $\N[A]^+=\{\sum a_iT^i|a_i\in\N\}$ the semiring of polynomials, we denote by $\N[A]$ the blueprint of monomials $aT^i$ with $a\in\N$. Given semiring homomorphisms $D\to B$ and $D\to C$, we denote by $B\Sotimes_DC$ the tensor product in the category of semirings, which differs in general from the tensor product $B\otimes_DC$ in the category of blueprints. Note that the precise relationship is given by $B\Sotimes_DC=(B\otimes_DC)^+$ since the ``semiring completion'' $(-)^+:\Blpr\to\SRings$ is left-adjoint to the embedding $\SRings\to\Blpr$.

Other instances for this notation are the affine line and the multiplicative group scheme. In the category of (ordered) blueprints, the functor $B\mapsto B$ is represented by $\A^1_{B}=\Spec B[X]$ and the functor $B\mapsto B^\times$ is represented by $\G_{m,B}=\Spec B[X^{\pm1}]$. To distinguish these objects from the classical affine line and the classical multiplicative group scheme for a semiring $B$, we use $\A^{1,+}_{B}=\Spec B[X]^+$ and $\G^+_{m,B}=\Spec B[X^{\pm1}]^+$ for the latter objects.

%%%%%%%%%%%%%%%%%%%%%%%%%%%%%%%%%%%%%%%%%%%%%%%%%%%%%%%%%%%%%%%%%%%%%%%%%%%%%%%%%%%%%%%%%%%%%%%%%%%%%%%%%%%%%%%%%%%%%%%%%%%%%%%%%%%%%%%%%%%%%%%%%%%%%%%%%%%%%%%%%%%%%%%%%%%
%%%%%%%%%%%%%%%%%%%%%%%%%%%%%%%%%%%%%%%%%%%%%%%%%%%%%%%%%%%%%%%%%%%%%%%%%%%%%%%%%%%%%%%%%%%%%%%%%%%%%%%%%%%%%%%%%%%%%%%%%%%%%%%%%%%%%%%%%%%%%%%%%%%%%%%%%%%%%%%%%%%%%%%%%%%

\section{Basic definitions}
\label{section: ordered blueprints}

In this section, we introduce the category of ordered blueprints and various subcategories and endofunctors that are of relevance for this paper.

\begin{df}
 An \emph{ordered blueprint} is a monoid $A$ with zero together with a \emph{subaddition on $A$}, which is a relation $\cR$ on the set $\N[A]^+=\{\sum a_i|a_i\in A\}$ of finite formal sums of elements of $A$ that satisfies the following list of axioms (where we write $\sum a_i\leq \sum b_j$ for $(\sum a_i,\sum b_j)\in\cR$ with $a_i,b_j\in A$, and where $0$ is the zero in $A$ and $\text{(empty sum)}$ is the empty sum in $\N[A]^+$). 
 \begin{enumerate}[label={(B\arabic*)}]
  \item\label{B1} $a\leq a$ for all $a\in A$;                                                                           \hfill\textit{(reflective)}
  \item\label{B2} $\sum a_i\leq\sum b_j$ and $\sum b_j\leq\sum c_k$ implies $\sum a_i\leq\sum c_k$;                     \hfill\textit{(transitive)}
  \item\label{B3} $\sum a_i\leq\sum b_j$ and $\sum c_k\leq\sum d_l$ implies $\sum a_i+\sum c_k\leq\sum b_j+\sum d_l$;   \hfill\textit{(additive)}
  \item\label{B4} $\sum a_i\leq\sum b_j$ and $\sum c_k\leq\sum d_l$ implies $\sum a_ic_k\leq\sum b_jd_l$;               \hfill\textit{(multiplicative)}
  \item\label{B5} $0\leq\text{(empty sum)}$ and $\text{(empty sum)}\leq 0$;                                             \hfill\textit{(zero)}
  \item\label{B6} $a\leq b$ and $b\leq a$ implies $a=b$ as elements in $A$.                                             \hfill\textit{(proper)}
 \end{enumerate}
 We write $B=\bpquot A\cR$ for an ordered blueprint with $A$ and $\cR$ as above. A \emph{morphism of ordered blueprints $B_1=\bpquot{A_1}{\cR_1}$ and $B_2=\bpquot{A_2}{\cR_2}$} is a monoid morphism $f:A_1\to A_2$ such that $\sum a_i\leq_1\sum b_j$ implies $\sum f(a_i)\leq_2\sum f(b_j)$. We denote the category of ordered blueprints by $\OBlpr$.
\end{df}

Note that axioms \ref{B1} and \ref{B2} state that $\cR$ is a \emph{pre-order} on $\N[A]^+$, and axiom \ref{B6} states that $\cR$ restricts to a partial order on $A$, considered as a subset of $\N[A]^+$.

Often, we refer to the underlying set of $A$ by the symbol $B$, i.e.\ we write $a\in B$ for $a\in A$ and $f:B_1\to B_2$ for the underlying map $A_1\to A_2$ of monoids. We say that $\sum a_i\leq \sum b_j$ holds in $B=\bpquot A\cR$ if $(\sum a_i,\sum b_j)$ is an element of $\cR$. We write $\sum a_i\geq\sum b_j$ if $\sum b_j\leq\sum a_i$, and $\sum a_i\=\sum b_j$ if both $\sum a_i\leq\sum b_j$ and $\sum a_i\geq\sum b_j$.  For instance, axiom \ref{B5} can be rewritten as $0\=\text{(empty sum)}$.

Any set $S$ of relations (of the form $\sum a_i\leq\sum b_j$) has a closure $\gen S$ under axioms \ref{B1}--\ref{B5}, which is the smallest relation $\cR$ on $\N[A]^+$ that contains $S$ and satisfies axioms \ref{B1}--\ref{B5}. However, axiom \ref{B6} plays a restrictive role: not every set $S$ of relations is contained in a relation $\cR$ on $\N[A]^+$ that satisfies \ref{B6}. Therefore $\gen S$ will always refer to the \emph{closure under \ref{B1}--\ref{B5}}, while $B=\bpgenquot AS$ refers to the proper quotient, as introduced in the following section.

Every monoid $A$ with zero has a smallest subaddition $\gen\emptyset$, and we consider $A$ as the ordered blueprint $\bpgenquot{A}{\emptyset}$. Since a map $A_1\to A_2$ is a morphism in $\Mon$ if and only it is a morphism in $\OBlpr$, this defines a full embedding $\Mon\to\OBlpr$. We say that an ordered blueprint $B$ is a \emph{monoid} if it is in the essential image of this embedding.

%%%%%%%%%%%%%%%%%%%%%%%%%%%%%%%%%%%%%%%%%%%%%%%%%%%%%%%%%%%%%%%%%%%%%%%%%%%%%%%%%%%%%%%%%%%%%%%%%%%%%%%%%%%%%%%%%%%%%%%%%%%%%%%%%%%%%%%%%%%%%%%%%%%%%%%%%%%%%%%%%%%%%%%%%%%

\subsection{Proper quotients}
\label{subsection: proper quotients for blueprints}

Given a pair $B=(A,\cR)$ that satisfies axioms \ref{B1}--\ref{B5}, we can associate with it the following ordered blueprint $B_\proper=\bpquot{(A/\sim)}{\widetilde\cR}$. We define the equivalence relation $\sim$ on $A$ by $a\sim b$ if $a\= b$. We define $\sum\bar a_i\widetilde\leq \sum \bar b_j$ if $\sum a_i\leq\sum b_j$ where $\bar a_i,\bar b_j\in (A/\sim)$ are the respective classes of $a_i, b_j\in A$. By axiom \ref{B5}, $A/\sim$ is a monoid with zero, and axiom \ref{B2} ensures that the definition of $\widetilde\cR$ is independent of the choice of representatives. It is easily verified that $B_\proper$ is indeed an ordered blueprint, which we call the \emph{proper quotient of $B$}; we say that $B=(A,\cR)$ is an \emph{(improper) representation of $B_\proper$}. We denote by $\bpquot A\cR$ the proper quotient of $(A,\cR)$.

Let $B=(A,\cR)$ be as above and $f:B\to C$ a multiplicative map into an ordered blueprint $C$ such that $\sum a_i\leq\sum b_j$ in $B$ implies $\sum f(a_i)\leq\sum f(b_j)$ in $C$. Then $f$ factors uniquely into the quotient map $B\to B_\proper$ followed by a morphism $f_\proper: B_\proper\to C$. This shows that $(-)_\proper$ is functorial in pairs $(A,\cR)$ as above. If we want to stress that $(A,\cR)$ is a proper representation of $B$, then we write that $B=(A,\cR)$ is an ordered blueprint.

Let $B=\bpquot A\cR$ be an ordered blueprint and $\sim$ the restriction of $\cR$ to $A$. Then we call $B^\bullet=A/\sim$ the \emph{underlying monoid $A/\sim$} of $\bpquot A\cR$. Note that a morphism $B\to C$ determines a monoid morphism $f^\bullet:B^\bullet\to C^\bullet$ between the underlying monoids of $B$ and $C$. This yields a functor $(-)^\bullet:\OBlpr\to\Mon$, which is right adjoint and left inverse to the embedding $\Mon\to\OBlpr$.

We say that a morphism $f:B\to C$ is \emph{injective} or \emph{surjective} if the map $f^\bullet:B^\bullet\to C^\bullet$ between the underlying monoids is injective or surjective, respectively. A morphism $f:B\to C$ of blueprints is \emph{full} if every relation $\sum f(a_i)\leq\sum f(b_j)$ in $C$ with $a_i,b_j\in B$ implies $\sum a_i\leq\sum b_j$ in $B$.

A \emph{subblueprint of $B$} is a blueprint $B'$ together with an injective morphism $B'\to B$. Note that a morphism is a monomorphism if and only if the map between the underlying monoids is injective. A subblueprint is \emph{full} if the inclusion $B'\to B$ is full. Note that a full subblueprint $B'\subset B$ is determined by the submonoid $(B')^\bullet$ of $B^\bullet$.

Let $B=\bpquot{A}{\cR}$ be a blueprint and $S$ be a set of relations on $\N[A]^+$. We denote by $\bpgenquot BS$ the blueprint $\bpgenquot{A}{\cR\cup S}$.

%%%%%%%%%%%%%%%%%%%%%%%%%%%%%%%%%%%%%%%%%%%%%%%%%%%%%%%%%%%%%%%%%%%%%%%%%%%%%%%%%%%%%%%%%%%%%%%%%%%%%%%%%%%%%%%%%%%%%%%%%%%%%%%%%%%%%%%%%%%%%%%%%%%%%%%%%%%%%%%%%%%%%%%%%%%

\subsection{Limits and colimits}
\label{subsection: limits and colimits for blueprints}

The product $B=\bpquot A\cR$ of a family of blueprints $B_l=\bpquot{A_l}{\cR_l}$ is represented by the Cartesian product $A=\prod A_l$ of the underlying monoids with coordinatewise multiplication, together with the subaddition
\[\textstyle
 \cR \quad = \quad \{ \ \sum a_{i} \leq\sum b_{j} \ | \ \sum a_{i,l} \leq \sum b_{j,l} \text{ for all }l \ \}.
\]

The equalizer $\eq(f,g)$ of two morphisms $f,g: B_1\to B_2$ is represented by the full ordered subblueprint $B=\bpquot A\cR$ of $B_1$ with
\[
 A \quad = \quad \{ \ a\in B_1 \ | \ f(a)=g(a)\ \}.
\]

The coequalizer of two morphisms $f_1,f_2:B_1\to B_2$ is $\bpquot A\cR$ where $A$ is the underlying monoid of $B_2$ and $\cR$ is generated by the subaddition of $B_2$ and all relations of the form $f_1(a)\=f_2(a)$ with $a\in B_1$.

The coproduct of two ordered blueprints $B$ and $C$ is the smash product $B\wedge C$, which is obtained from Cartesian product $B\times C$ by identifying $B\times\{0\}\cup\{0\}\times C$ with $0$. In particular, there is a tensor product $B\otimes_DC$ for every diagram $B\leftarrow D \to C$, which is the quotient of $B\times C$ by all relations of the form $(db,c)=(b,dc)$ where $b\in B$, $c\in C$ and $d\in D$. The coproduct of an infinite family is the filtered colimit of the coproducts of its finite subfamilies. The filtered colimit can be constructed as usual; for instance, see \cite{CLS12}.

An initial object of $\OBlpr$ is the monoid $\Fun=\{0,1\}$ and a terminal object is the \emph{trivial blueprint} $0=\bpgenquot{\{0\}}{\emptyset}$. We summarize:

\begin{lemma}
 The category of ordered blueprints is complete and cocomplete with initial and terminal objects. \qed
\end{lemma}

%%%%%%%%%%%%%%%%%%%%%%%%%%%%%%%%%%%%%%%%%%%%%%%%%%%%%%%%%%%%%%%%%%%%%%%%%%%%%%%%%%%%%%%%%%%%%%%%%%%%%%%%%%%%%%%%%%%%%%%%%%%%%%%%%%%%%%%%%%%%%%%%%%%%%%%%%%%%%%%%%%%%%%%%%%%

\subsection{Algebraic blueprints}
\label{subsection: algebraic blueprints}

Whenever we see the need to make a clear distinction between blueprints as considered in \cite{Lorscheid17} and other types of ordered blueprints as they appear in this text, we shall add the attribute ``algebraic'' to ``blueprint.''

An \emph{(algebraic) blueprint} is a pair of a monoid $A$ together with a \emph{preaddition}, which is a subaddition $\cR$ that satisfies
\begin{enumerate}[label={(B\arabic*)}]\setcounter{enumi}{6}
 \item\label{B7} $\sum a_i\leq\sum b_j$ if and only if $\sum b_j\leq\sum a_i$,                                          \hfill\textit{(symmetric)}
\end{enumerate}
i.e.\ $\cR$ is an equivalence relation on $\N[A]^+$ that satisfies the additional axioms \ref{B3}--\ref{B6}. In other words, an ordered blueprint $B=\bpquot A\cR$ is an algebraic blueprint if and only if $\cR$ is symmetric.

A \emph{morphism of algebraic blueprints} is the same as a morphism of ordered blueprints. We denote the full subcategory of algebraic blueprints in $\OBlpr$ by $\Blpr=\OBlpr^\alg$ and the embedding as full subcategory by
\[
 \iota^\alg: \quad \Blpr \quad \longrightarrow \quad \OBlpr.
\]

The embedding $\iota^\alg$ has a left adjoint
\[
 (-)^\hull: \quad \OBlpr \quad \longrightarrow \quad \Blpr,
\]
which sends an ordered blueprint $B=\bpquot A\cR$ to its \emph{algebraic hull $B^\hull=\bpquot A{\cR^\hull}$} with 
\[\textstyle
 \cR^\hull \quad = \quad \{ \ \sum a_i\=\sum b_j \ | \ \sum a_i \leq \sum b_j\text{ in }B\ \}.
\]
Note that $(A,\cR^\hull)$ is in general not a proper representation of $B^\hull$, even if $(A,\cR)$ is a proper representation of $B$. The algebraic hull comes with a canonical morphism $B\to B^\hull$ that maps $a\in B$ to its class in the proper quotient of $(A,\cR^\hull)$. This morphism is universal among all morphisms from $B$ to algebraic blueprints, which explains the functoriality of $(-)^\hull$.

The embedding $\iota^\alg$ has also a right adjoint
\[
 (-)^\core: \quad \OBlpr \quad \longrightarrow \quad \Blpr,
\]
which sends an ordered blueprint $B=\bpquot A\cR$ to its \emph{algebraic core $B^\core=\bpquot A{\cR^\core}$} with
\[\textstyle
 \cR^\core \quad = \quad \{ \ \sum a_i\=\sum b_j \ | \ \sum a_i \= \sum b_j\text{ in }B \ \}.
\]
Note that $(A,\cR^\core)$ is always a proper representation of $B^\core$ if $(A,\cR)$ is a proper representation of $B$. The algebraic core comes with a canonical morphism $B^\core\to B$ that is the identity on the underlying monoid $A$. This morphism is universal among all morphisms from an algebraic blueprint to $B$, which explains the functoriality of $(-)^\core$.

\begin{lemma}
 An ordered blueprint $B$ is an algebraic blueprint if and only if one of the following equivalent conditions are satisfied:
 \begin{enumerate}
  \item The canonical morphism $B\to B^\hull$ is an isomorphism.
  \item The canonical morphism $B^\core\to B$ is an isomorphism.
  \item The canonical morphism $B^\core\to B^\hull$ is an isomorphism.
 \end{enumerate}
\end{lemma}

\begin{proof}
 This follows easily from the definitions.
\end{proof}

%%%%%%%%%%%%%%%%%%%%%%%%%%%%%%%%%%%%%%%%%%%%%%%%%%%%%%%%%%%%%%%%%%%%%%%%%%%%%%%%%%%%%%%%%%%%%%%%%%%%%%%%%%%%%%%%%%%%%%%%%%%%%%%%%%%%%%%%%%%%%%%%%%%%%%%%%%%%%%%%%%%%%%%%%%%

\subsection{Blueprints with inverses}
\label{subsection: blueprints with inverses}

A \emph{blueprint with $-1$} (or \emph{with inverses}) is an ordered blueprint $B$ that contains an element $-1$ that satisfies $1+(-1)\=0$. This implies that every element $a\in B$ has an additive inverse, which is an element $-a$ with $a+(-a)\=0$. The additive inverse is necessarily unique.

The blueprint $\Funsq=\bpgenquot{\{0,\pm1\}}{1+(-1)\=0}$ has a unique morphism into any other blueprint with $-1$. Therefore the full subcategory of blueprints with $-1$ corresponds to $\Blpr_\Funsq$, and the base change functor $-\otimes_\Fun\Funsq:\OBlpr\to\Blpr_\Funsq$ is left adjoint and left inverse to the inclusion $\iota^\inv:\Blpr_\Funsq\to\OBlpr$ as a subcategory. We also write $B^\inv$ for $B\otimes_\Fun\Funsq$.

\begin{lemma}\label{lemma: ordered blueprints with -1 are algebraic}
 Every ordered blueprint with $-1$ is an algebraic blueprint. 
\end{lemma}

\begin{proof}
 By multiplication with $-1$, a relation $\sum a_i\leq \sum b_j$ implies that $\sum -a_i\leq \sum -b_j$ and thus
 \[\textstyle
  \sum b_j \quad \= \quad \sum b_j +\sum -a_i +\sum a_i \quad \leq \quad \sum b_j +\sum -b_j +\sum a_i  \quad \= \quad \sum a_i,
 \]
 which shows that, indeed, $\sum a_i\=\sum b_j$.
\end{proof}

Summing up all facts, we justified the notation $\Blpr^\inv$ or $\Blpr_\Funsq$ for the full subcategory of blueprints with $-1$ in $\OBlpr$.

%%%%%%%%%%%%%%%%%%%%%%%%%%%%%%%%%%%%%%%%%%%%%%%%%%%%%%%%%%%%%%%%%%%%%%%%%%%%%%%%%%%%%%%%%%%%%%%%%%%%%%%%%%%%%%%%%%%%%%%%%%%%%%%%%%%%%%%%%%%%%%%%%%%%%%%%%%%%%%%%%%%%%%%%%%%

\subsection{The universal ordered semiring}
\label{subsection: the universal ordered semiring}

In this text, an \emph{ordered semiring} is a semiring $R$ together with a partial order $\leq$ that satisfies for all $x,y,z,t\in R$
\begin{enumerate}[label={(S\arabic*)}]
 \item $x\leq y$ and $z\leq t$ implies $x+z\leq y+t$;                 \hfill\textit{(additive)}
 \item $x\leq y$ and $z\leq t$ implies $xz\leq yt$.                   \hfill\textit{(multiplicative)}
\end{enumerate}
A morphism of ordered semirings is an order-preserving homomorphism of ordered semirings that maps $0$ to $0$ and $1$ to $1$.

Let $B=\bpquot A\cR$ be an ordered blueprint. The \emph{universal ordered semiring $B^+$ associated with $B$} is the semiring $\N[A]^+/\cR^\core$ together with the partial order defined by
\[\textstyle
 \left[\sum a_i\right] \leq \left[\sum b_j\right] \quad \text{in }B^+\text{ if and only if}\quad \sum a_i \leq \sum b_j \quad\text{in }B.
\]
Note that $B^+$ is well-defined as an ordered semiring: by additivity and multiplicativity of $\cR$, $B^+$ inherits the structure of a semiring as a quotient of the semiring $\N[A]^+$; by transitivity, the partial order on $B^+$ is well-defined (as a relation on $B^+$); by reflexivity and transitivity of $\cR$ and the definition of $B^+$, this relation is indeed a partial order on $B^+$; again by additivity and multiplicativity, the partial order of $B^+$ is additive and multiplicative.

A morphism $f:B\to C$ of ordered blueprints induces a morphism of ordered semirings $f^+:B^+\to C^+$ that is defined by $f([\sum a_i])=[\sum f(a_i)]$. This establishes the functor $(-)^+$ from $\OBlpr$ to the category of ordered semirings.

Conversely, we can consider every ordered semiring $(R,\leq)$ as an ordered blueprint $B=\bpquot A\cR$: we let $A$ be the underlying multiplicative monoid of $R$ and define
\[\textstyle
 \cR \quad = \quad \{ \ \sum a_i \leq \sum b_j \ | \ \sum a_i \leq \sum b_j \text{ in }R \ \}.
\]
Then a map between ordered semirings is the same as a morphism between the associated ordered blueprints. This defines an embedding $\iota^+$ of the category of ordered semirings as a full subcategory of the category of ordered blueprints, which is a left adjoint to $(-)^+$. Moreover, $(-)^+\circ\iota^+$ is isomorphic to the identity functor on the category of ordered semirings. 

From now on, we identify the category of ordered semirings with the essential image of $\iota^+$, which allows us to talk about morphisms from ordered blueprints into ordered semirings. We see that an ordered blueprint $B=\bpquot A\cR$ comes with the morphism $B\to B^+$ that sends $a\in A$ to the class $[a]\in B^+$. This morphism is universal for morphisms from $B$ into ordered semirings.

We identify the category of semirings with the subcategory of trivially ordered semirings. This coincides with the realization of semirings as algebraic blueprints, followed by the embedding of $\Blpr$ into $\OBlpr$. We derive yet another characterization of algebraic blueprints in $\OBlpr$.

\begin{lemma}
 An ordered blueprint $B$ is an algebraic blueprint if and only if $B^+$ is trivially ordered. \qed
\end{lemma}

\begin{rem} \label{rem: ordered blueprints as monoids in ordered semirings}
 The subaddition $\cR$ of $B$ can be recovered from the embedding $B^\bullet\subset B^+$ as 
 \[\textstyle
  \cR \quad = \quad \Bigl\langle \ \sum a_i \leq \sum b_j \ \Bigl| \ \sum a_i\leq\sum b_j \text{ in } B^+ \ \Bigr\rangle.
 \]
 An order preserving homomorphism $f:B_1^+\to B_2^+$ of ordered semirings comes from a morphism of blueprints $B_1\to B_2$ if and only if $f$ maps the underlying monoid of $B_1$ to the underlying monoid of $B_2$.
 
 This yields an equivalence between the category of ordered blueprints, as defined in this section, with the category of inclusions $B^\bullet\subset B^+$ as considered in the introduction.
\end{rem}

%%%%%%%%%%%%%%%%%%%%%%%%%%%%%%%%%%%%%%%%%%%%%%%%%%%%%%%%%%%%%%%%%%%%%%%%%%%%%%%%%%%%%%%%%%%%%%%%%%%%%%%%%%%%%%%%%%%%%%%%%%%%%%%%%%%%%%%%%%%%%%%%%%%%%%%%%%%%%%%%%%%%%%%%%%%

\subsection{Monomial blueprints}
\label{subsection: monomial blueprints}

The first step towards realizing norms and valuations as morphisms is to concentrate on inequalities of the form $a\leq \sum b_j$. We can associate with every ordered blueprint an ordered blueprint based on inequalities of this sort in a functorial way.

A \emph{(left) monomial relation} is a relation of the form $a\leq\sum b_j$. A \emph{(left) monomial (ordered) blueprint} is an ordered blueprint $B$ whose subaddition $\cR$ is generated by monomial relations, i.e.
\[\textstyle
 \cR \quad = \quad \left\langle \ a\leq\sum b_j \ \left| \ (a,\sum b_j)\in\cR \ \right.\right\rangle.
\]
We denote the full subcategory of monomial blueprints in $\OBlpr$ by $\OBlpr^\mon$.

Let $B=\bpquot A\cR$ be an ordered blueprint. The \emph{associated monomial blueprint} is defined as $B^\mon=\bpquot A{\cR^\mon}$ with
\[\textstyle
 \cR^\mon \quad = \quad \left\langle \ a\leq\sum b_j \ \left| \ a\leq \sum b_j\text{ in }B \right.\right\rangle.
\]
The obvious inclusion $B^\mon\to B$ is universal for all morphisms from a monomial blueprint to $B$, and it is an isomorphism if and only if $B$ itself is monomial. This defines a right adjoint and left inverse $(-)^\mon:\OBlpr\to\OBlpr^\mon$ to the inclusion functor $\iota^\mon:\OBlpr^\mon\to\OBlpr$.

\begin{rem}\label{rem: halos}
 Paugam's category $\Halos$ of halos and halo morphisms in \cite{Paugam09} appears naturally as a full subcategory of $\OBlpr^\mon$. A \emph{halo} is an ordered semiring and a \emph{(multiplicative) halo morphism} is an order preserving multiplicative map $f:B_1\to B_2$ of ordered semirings such that $f(0)=0$, $f(1)=1$ and $f(a+b)\leq f(a)+f(b)$. If we consider $B_1$ and $B_2$ as ordered blueprints, then it is easily seen that a map $f:B_1\to B_2$ is a halo morphism if and only if the composition $f':B_1^\mon\to B_1\to B_2$ is a morphism of ordered blueprints. By the universal property of a monomial blueprint, $f'$ factors uniquely through the morphism $f^\mon:B_1^\mon\to B_2^\mon$ of monomial blueprints. This defines a fully faithful embedding $(-)^\mon:\Halos\to\OBlpr^\mon$.
\end{rem}
 
\begin{rem}\label{rem: hyperrings}
 Another closely related concept is the notion of a hyperring, cf.\ \cite{Stratigopoulos69}, \cite{Krasner57}, \cite{Viro10}, \cite{Viro11} and \cite{Connes-Consani11b}. A \emph{(commutative) hyperring} is a multiplicative monoid $R$ together with a function $f:R\times R\to \cP(R)$ into the power set $\cP(R)$ of $R$ that associates defines the sum $a+b$ of two elements $a,b\in R$ as a non-empty subset of $R$. This function satisfies certain axioms in analogy to the classical ring axioms. Given a hyperring $R$, we define the blueprint $B=\bpquot{R^\bullet}{\cR_R}$ where $\cR_R$ is generated by the monomial relations $c\leq a+b$ whenever $c\in a+b$. This defines a full embedding of the category of hyperrings into the category of monomial blueprints. 
\end{rem}

%%%%%%%%%%%%%%%%%%%%%%%%%%%%%%%%%%%%%%%%%%%%%%%%%%%%%%%%%%%%%%%%%%%%%%%%%%%%%%%%%%%%%%%%%%%%%%%%%%%%%%%%%%%%%%%%%%%%%%%%%%%%%%%%%%%%%%%%%%%%%%%%%%%%%%%%%%%%%%%%%%%%%%%%%%%

\subsection{Partially additive blueprints}
\label{subsection: partially additive blueprints}

A \emph{partially additive blueprint} is an algebraic blueprint $B$ whose preaddition is generated by relations of the form $a\=\sum b_j$. We denote the full subcategory of partially additive blueprints in $\OBlpr$ by $\Blpr^\padd$. To an ordered blueprint $B=\bpquot A\cR$, we can associate a partially additive blueprint $B^\padd=\bpquot{B}{\cR^\padd}$ where $\cR^\padd$ is generated by all relations $a\=\sum b_j$ that are contained in $\cR$. The inclusion $B^\padd\to B$ is universal for all morphism from a partially ordered blueprint to $B$. This defines a right adjoint and left inverse $(-)^\padd:\OBlpr\to \Blpr^\padd$ to the inclusion functor $\iota^\padd:\Blpr^\padd\to\OBlpr$. 

By the very definition of partially additive and monomial blueprints, we obtain that $(-)^\mon:\Blpr^\padd\to\Blpr^\mon$ is a fully faithful embedding of categories with left adjoint and left inverse $(-)^\hull:\Blpr^\mon\to\Blpr^\padd$. This means that we model the category of partially additive blueprints as monomial blueprints, which will be of importance for our observations in section \ref{subsection: tropicalization for a totally positive base}. Note that examples of partially additive blueprints include semirings, monoids and blueprints with $-1$. 

\begin{rem}\label{rem: sesquiads}
 Partially additive blueprints are closely connected to the idea of a sesquiad, cf.\ \cite{Deitmar11a}. Namely, a \emph{sesquiad} is a monoid $A$ together with partial functions $f_n:A^n\dashrightarrow A$ (for $n\geq1$) that satisfies certain axioms. The functions $f_n$ express the sum $a=\sum b_j\in A$ if $(b_1,\dotsc,b_n)$ is in the domain of $f_n$. Equivalently, a sesquiad is a partially additive blueprint $B=\bpquot A\cR$ that is \emph{cancellative}, i.e.\ the canonical morphism $B^+\to B^{+,\inv}$ is injective.
\end{rem}

%%%%%%%%%%%%%%%%%%%%%%%%%%%%%%%%%%%%%%%%%%%%%%%%%%%%%%%%%%%%%%%%%%%%%%%%%%%%%%%%%%%%%%%%%%%%%%%%%%%%%%%%%%%%%%%%%%%%%%%%%%%%%%%%%%%%%%%%%%%%%%%%%%%%%%%%%%%%%%%%%%%%%%%%%%%

\subsection{Totally positive blueprints}
\label{subsection: totally positive blueprints}

A \emph{totally positive} blueprint is an ordered blueprint $B$ that satisfies $0\leq 1$. We denote the full subcategory of totally positive blueprints in $\OBlpr$ by $\OBlpr^\pos$.

Let $B=\bpquot A\cR$ be an ordered blueprint. The \emph{associated totally positive blueprint} is defined as $B^\pos=\bpgenquot B{0\leq 1}$, which is the same as $B\otimes_\Fun\bigl(\bpgenquot\Fun{0\leq1}\bigr)$. The obvious morphism $B\to B^\pos$ is universal for all morphisms from $B$ to a totally positive blueprint, and it is an isomorphism if and only if $B$ itself is totally positive. This defines the functor $(-)^\pos:\OBlpr\to\OBlpr^\pos$, which is left adjoint and left inverse to the inclusion $\OBlpr^\pos\to\OBlpr$ as a subcategory.

\begin{lemma}\label{lemma: totally positive blueprints}
 Let $B$ be an ordered blueprint. Then the following are equivalent.
 \begin{enumerate}
 \item\label{pos1} $B$ is totally positive;
 \item\label{pos2} $0\leq a$ for all $a\in B$;
 \item\label{pos3} $\sum a_i+\sum c_k\leq\sum b_j$ implies $\sum a_i\leq\sum b_j$.
 \end{enumerate}
\end{lemma}
       
\begin{proof}
 Let $B$ satisfy \eqref{pos1}. Multiplying the relation $0\leq1$ by $a\in B$ yields \eqref{pos2}.
 
 Let $B$ satisfy \eqref{pos2}. A relation $\sum a_i + \sum c_k\leq \sum b_j$ implies $\sum a_i \= \sum a_i + \sum 0 \leq \sum a_i+\sum c_k \leq \sum b_j$, which is \eqref{pos3}.
 
 Let $B$ satisfy \eqref{pos3}. Then $0+1\leq 1$ implies $0\leq 1$. Thus \eqref{pos1}.
\end{proof}

\begin{cor}\label{cor: properties of totally positive blueprints}
 Let $B$ be an ordered blueprint. 
 \begin{enumerate}
  \item\label{sub1} If $a\leq 0$ in $B$, then $a\=0$ in $B^\pos$.
  \item\label{sub2} If $1+\sum c_k\leq 0$ for some $c_k$ in $B$, then $B^\pos$ is the trivial blueprint. Thus if $B$ is with $-1$, then $B^\pos=0$.
  \item\label{sub3} The canonical morphism $B\to B^\pos$ is the identity between the respective underlying monoids if and only if 
        \[\textstyle
         a+\sum c_k\leq b \quad \text{and} \quad b+\sum d_l\leq a \quad \text{imply} \quad a =b \quad \text{in} \quad B.
        \]
        If $B$ is a semiring, then this is the case if and only if $a+c+d=a$ implies $a+c=a$.
 \end{enumerate}
\end{cor}

\begin{proof}
 By Lemma \ref{lemma: totally positive blueprints} \eqref{pos2}, we have $0\leq a$ for all $a$ in $B^\pos$. If $a\leq 0$ in $B$, then $a\=0$ in $B^\pos$, which shows \eqref{sub1}.

 By Lemma \ref{lemma: totally positive blueprints} \eqref{pos3}, a relation $1+\sum c_k\leq 0$ in $B$ implies $1\leq0$ in $B^\pos$. Thus $0\=1$ by \eqref{sub1}, which is equivalent with $B^\pos=0$. This shows \eqref{sub2}.

 We prove \eqref{sub3}. By definition, the canonical morphism $B\to B^\pos$ is a surjective monoid morphism. Two relations $a+\sum c_k\leq b$ and $b+\sum d_l\leq a$ in $B$ imply that $a\=b$ in $B^\pos$, which shows that already $a\=b$ in $B$ if $B\to B^\pos$ is an isomorphism between the respective underlying monoids.

 Since all additional relations in $B^\pos$ are generated by leaving out summands on the left hand side of relations in $B$, we get only new relations $a\leq b$ in $B^\pos$ for relations of the form $a+\sum c_k\leq b$ in $B$. From this, the other direction of the claim follows.

 If $B$ is a semiring, we can substitute $\sum c_k$ by its sum $c$ and $\sum d_l$ by its sum $d$. Moreover, the inequalities $a+c\leq b$ and $b+d\leq a$ are equalities, which yields
 \[
  a+c+d \quad = \quad b+d \quad = \quad a.
 \]
 On the other hand, $a+c+d=a$ yields $b+d=a$ if we set $b=a+c$, i.e.\ we re-obtain the two equations that we started with. With the same substitution $b=a+c$, the equation $a=b$ is equivalent to $a=a+c$. This proves the latter claim of \eqref{sub3}.
\end{proof}

%%%%%%%%%%%%%%%%%%%%%%%%%%%%%%%%%%%%%%%%%%%%%%%%%%%%%%%%%%%%%%%%%%%%%%%%%%%%%%%%%%%%%%%%%%%%%%%%%%%%%%%%%%%%%%%%%%%%%%%%%%%%%%%%%%%%%%%%%%%%%%%%%%%%%%%%%%%%%%%%%%%%%%%%%%%

\subsection{Strictly conic blueprints}
\label{subsection: strictly conic blueprints}

A property that is relevant to our proof of Theorem \ref{thmA} (tropicalization with a totally positive base) is that $(B^\pos)^\core$ is isomorphic to $B$. In this section, we shall see that this property singles out strictly conic ordered blueprints whose name is derived from the notion of a strict semiring: a semiring $R$ is strictly conic if for every $x\in R$, the subset $x+R$ forms a ``strict cone,'' in the sense that $x\notin y+R$ for every $y\in x+R$ different from $y=x$.

A \emph{strictly conic ordered blueprint} is an ordered blueprint $B$ that satisfies
\begin{enumerate}[label={(B\arabic*)}]\setcounter{enumi}{7}
 \item\label{B8} $\sum a_i+\sum c_k\leq\sum b_j$ and $\sum b_j+\sum d_l\leq\sum a_i$ imply $\sum a_i\=\sum b_j$.                                          \hfill\textit{(strictly conic)}
\end{enumerate}
We denote the full subcategory of strictly conic ordered blueprints in $\OBlpr$ by $\OBlpr^\conic$. Let $B=\bpquot A\cR$ be an ordered blueprint. The \emph{strictly conic ordered blueprint associated with $B$} is the ordered blueprint $B^\conic=\bpquot A{\cR^\conic}$ where $\cR^\conic$ is generated by $\cR$ and
\[\textstyle
 \left\{ \ \sum a_i\=\sum b_j \ \left| \ \sum a_i+\sum c_k\leq\sum b_j\text{ and }\sum b_j+\sum d_l\leq\sum a_i \ \right.\right\}.
\]
The associated strictly conic ordered blueprint comes together with the obvious morphism $B\to B^\conic$, which is universal for all morphisms from $B$ to a strictly conic ordered blueprint. This defines the left adjoint and left inverse $(-)^\conic:\OBlpr\to\OBlpr^\conic$ to the inclusion $\OBlpr^\conic\to\OBlpr$ as a subcategory.

%In the algebraic case, $B$ is strictly conic if for every $x\in B^+$, the subset $x+B^+$ forms a ``strict cone,'' in the sense that $x\notin y+B^+$ for every $y\in x+B^+$ that differs from $x$.

In order to investigate the relation between an ordered blueprint $B$ and $(B^\pos)^\core$, consider the commutative diagram
\[
 \xymatrix@C=5pc@R=1pc{  & B \ar[dr]^{\alpha_B^\pos} \\ B^\core \ar[ur]^{\alpha_B^\core} \ar[dr]_{\beta_B=(\alpha_B^\pos)^\core\quad}  & & B^\pos. \\ & (B^\pos)^\core \ar[ur]_{\alpha_{B^\pos}^\core}}
\]

\begin{prop} \label{prop: beta isom iff B strictly conic}
 The map $\beta_B$ is an isomorphism of blueprints if and only if $B$ is strictly conic.
\end{prop}

\begin{proof}
 Assume that $\beta_B$ is an isomorphism. Since $\sum a_i+\sum c_k\leq\sum b_j$ and $\sum b_j+\sum d_l\leq\sum a_i$ in $B$ imply $\sum a_i\=\sum b_j$ in $B^\pos$, and therefore in $(B^\pos)^\core$, this must also hold in $B^\core$ as $\beta_B$ is an isomorphism. By the definition of the algebraic core, $\sum a_i\=\sum b_j$ in $B$, which shows that $B$ is strictly conic.

 To prove the reverse direction, assume that $B$ is strictly conic. By Corollary \ref{cor: properties of totally positive blueprints} \eqref{sub3}, $\alpha_B^\pos$ is an isomorphism between the underlying monoids. The maps $\alpha_B^\core$ and $\alpha_{B^\pos}^\core$ are so, too, by the definition of the algebraic core. This shows that $\beta_B$ is an isomorphism between the underlying monoids.

 Given an equality $\sum a_i\=\sum b_j$ in $(B^\pos)^\core$, this must already hold in $B^\pos$. By the definition of $B^\pos$, there must be relations of the form $\sum a_i+\sum c_k\leq\sum b_j$ and $\sum b_j+\sum d_l\leq\sum a_i$ in $B$. As $B$ is strictly conic, we have $\sum a_i\=\sum b_j$ in $B$ and therefore in $B^\core$. This shows that $\beta_B$ is an isomorphism.
\end{proof}

Let $\Blpr^{\conic}$ be the full subcategory of strictly conic algebraic blueprints in $\OBlpr$.

\begin{cor} \label{cor: strictly conic algebraic blueprint}
 Let $B$ be an algebraic blueprint. Then  $(B^\pos)^\core$ is isomorphic to $B$ if and only if $B$ is strictly conic. Consequently, $(-)^\pos$ embeds $\Blpr^\conic$ fully faithfully into $\OBlpr^\pos$, with left-inverse $(-)^\core$. \qed
\end{cor}

\begin{cor} \label{cor: strictly conic semiring}
 Let $B$ be a semiring. Then the following are equivalent.
 \begin{enumerate}
  \item\label{item1} $B$ is strictly conic.
  \item\label{item2} $B=B^\core\to (B^\pos)^\core$ is an isomorphism.
  \item\label{item3} $a+c+d=a$ implies $a+c=a$.
  \item\label{item4} $B\to B^\pos$ is an isomorphism between the underlying monoids.
 \end{enumerate}
\end{cor}

\begin{proof}
 The equivalence of \eqref{item1} and \eqref{item2} is Corollary \ref{cor: strictly conic algebraic blueprint}. The equivalence of \eqref{item3} and \eqref{item4} is Corollary \ref{cor: properties of totally positive blueprints} \eqref{item3}. That \eqref{item3} is equivalent to \eqref{item1} is shown analogous to the proof of Corollary \ref{cor: properties of totally positive blueprints} \eqref{item3}.
\end{proof}

Recall that a \emph{strict semiring} is a semiring $B$ such that an equality $a+b=0$ implies $a=0$. An \emph{idempotent blueprint} is a blueprint $B$ with $1+1\=1$, which implies $a+a\=a$ for every $a\in B$. We denote the full subcategory of $\OBlpr$ of idempotent blueprints by $\OBlpr^\idem$. Its initial object is the Boolean semiring $\B=\bpgenquot{\{0,1\}}{1+1\=1}$ and the functor $-\otimes_\Fun\B$ is a left adjoint and left inverse to the inclusion functor $\OBlpr^\idem\to\OBlpr$. A \emph{nonnegative blueprint} is a blueprint $B$ such that the only element $a\in B$ with $a\leq 0$ is $a=0$.

\begin{lemma}\label{lemma: idempotent or totally positive implies strictly conic}
 The following holds true.
 \begin{enumerate}
  \item\label{idemposmonconic1} A strictly conic semiring is strict.
  \item\label{idemposmonconic2} An idempotent algebraic blueprint is strictly conic.
  \item\label{idemposmonconic3} A totally positive blueprint is strictly conic.
  \item\label{idemposmonconic4} A nonnegative monomial blueprint is strictly conic.
 \end{enumerate}
\end{lemma}

\begin{proof}
 Let $B$ be a strictly conic semiring. Since $a+b=0$ implies $0+a+b=0$ and thus $a=0+a=0$, $B$ is a strict semiring. Thus \eqref{idemposmonconic1}.

 Let $B$ be an idempotent algebraic blueprint and assume that $\sum a_i+\sum c_k\=\sum b_j$ and $\sum b_j+\sum d_l\=\sum a_i$. Then
 \begin{align*}
  \textstyle \sum a_i \ \= \ \sum a_i+\sum a_i \ \= \ \sum a_i+\sum b_j+\sum d_l \ &\textstyle \= \ \sum a_i+\sum b_j+\sum b_j+\sum d_l \\ &\textstyle \= \ \sum a_i+\sum a_i+\sum c_k+\sum b_j+\sum d_l \\ &\textstyle \= \ \sum a_i+\sum b_j+\sum c_k+\sum d_l,
 \end{align*}
 which, by reasons of symmetry, equals $\sum b_j$. Therefore $B$ is strictly conic. Thus \eqref{idemposmonconic2}.

 If $B$ is totally positive, then the relations $\sum a_i+\sum c_k\=\sum b_j$ and $\sum b_j+\sum d_l\=\sum a_i$ imply $\sum a_i\leq\sum b_j$ and $\sum b_j\leq\sum a_i$. Thus $\sum a_i\=\sum b_j$ as desired. This shows \eqref{idemposmonconic3}.
 
 Let $B$ be non-negative and monomial and consider $\sum a_i+\sum c_k\leq\sum b_j$ and $\sum b_j+\sum d_l\leq\sum a_i$ where we assume that $a_i,b_j,c_k,d_l$ are non-zero. These relations are generated by left monomial relations of the form $a'\leq \sum b'_j$, which contain at least one nonzero term $b_j'$ if $a'$ is nonzero since $B$ is nonnegative. Therefore $\#\{a_i,c_k\}\leq\#\{b_j\}$ and $\#\{b_j,d_l\}\leq\#\{a_i\}$, which is only possible if  $\{c_k\}=\{d_l\}=\emptyset$. Consequently, $\sum a_i\=\sum b_j$, which shows that $B$ is strictly conic as claimed in \eqref{idemposmonconic4}.
\end{proof}

\begin{cor}\label{cor: idempotent implies core=core after pos and B subset B-pos}
 If $B$ is an idempotent ordered blueprint, then the canonical morphism $B^\core\to (B^\pos)^\core$ is an isomorphism and the canonical morphism $B\to B^\pos$ is a bijection.
\end{cor}

\begin{proof}
 By Lemma \ref{lemma: idempotent or totally positive implies strictly conic}, $B$ is strictly conic and by Proposition \ref{prop: beta isom iff B strictly conic}, $B^\core\to (B^\pos)^\core$ is an isomorphism. Consequently, we obtain a bijection $B^\core=(B^\pos)^\core\to B^\pos$, which factors into the canonical morphisms $B^\core\to B$ and $B\to B^\pos$. Since $B^\core\to B$ is a bijection, we conclude that $B\to B^\pos$ is also a bijection.
\end{proof}

\begin{ex}[A strict semiring that is not strictly conic]
 The semiring $R=\bpgenquot{\N[S,T]^+}{1+S+T\=1}$ is obviously a strict semiring. However, $1+S+T=1$ while $1+S\neq 1$, which shows that $R$ is not strictly conic.
\end{ex}

\begin{ex}[Idempotent semirings]
 We can endow an idempotent semiring $B$ with the partial order $a\leq b$ if and only if there is a $c\in B$ such that $a+c\=b$. By Corollary \ref{cor: idempotent implies core=core after pos and B subset B-pos}, $B\to B^\pos$ is a bijection and by Lemma \ref{lemma: totally positive blueprints} \eqref{pos3}, the partial order associated with $B$ is equal to the restriction of the subaddition of $B^\pos$ to its underlying set, which is equal to $B$. This observation is crucial for our reinterpretation of valuations in idempotent semirings in terms of morphisms into the associated totally positive blueprint.
\end{ex}

\begin{ex}[Non-negative reals and tropical numbers]\label{ex: nonnegative reals and the tropical numbers}
 The non-negative real numbers $\R_{\geq0}$ form a strictly conic semiring with respect to to their usual multiplication and addition. A more general class of strictly conic semirings are the non-negative reals together with the modified addition
 \[
  a+^tb \quad = \quad \left\{ \begin{array}{ll} (a^t+b^t)^{1/t} & \text{if }t\in [1,\infty) \\
                                              \max\{a,b\}       & \text{if }t=\infty.
                            \end{array} \right.
 \]
 We denote this semiring by $\R_{\geq0}^t$, and $\R_{\geq0}^\infty$ by $\T$, the \emph{tropical numbers}. Note that taking logarithms identifies $\T$ with the $\max$/$+$-semiring $\R\cup\{-\infty\}$, which is more commonly considered as the semiring of tropical numbers. 

 By Corollary \ref{cor: strictly conic semiring}, $\R_{\geq0}\to \R_{\geq0}^\pos$ and $\T\to \T^\pos$ are isomorphisms between the underlying monoids, which is in both cases $\R_{\geq0}^\bullet$. The subaddition of $\R_{\geq0}^\pos$ is
 \[\textstyle
  \left\{ \ \sum a_i \leq \sum b_j \ \left| \ \sum a_i + c\= \sum b_j\text{ for some }c\text{ in }\R_{\geq0} \ \right.\right\},
 \]
 which coincides with the natural order of $\R_{\geq0}$ if we identify $\sum a_i$ with its sum in $\R_{\geq0}$. The subaddition of $\T^\pos$ is
 \[\textstyle
  \left\{ \ \sum a_i \leq \sum b_j \ \left| \ \max\{a_i\} \leq \max\{b_j\}\text{ in }\R_{\geq0} \ \right.\right\},
 \]
 which also induces the natural order of $\R_{\geq0}$ if restricted to relations of the form $a\leq b$ with $a,b\in \R_{\geq0}$. For the consideration of non-archimedean norms, it is useful to observe the following comparison between these two totally positive blueprints.
\end{ex}

\begin{lemma}\label{lemma: the morphism from T to R}
 The identity between the underlying monoids induces a morphism 
 \[
  (\T^\pos)^\mon \ \longrightarrow \ (\R_{\geq0}^\pos)^\mon
 \]
 of monomial blueprints.
\end{lemma}

\begin{proof}
 By the above description of the respective preadditions of $\R_{\geq0}^\pos$ and $\T^\pos$, we conclude that the preaddition of $(\T^\pos)^\mon$ is generated by the relations of the form $a\leq\sum b_j$ whenever $a\leq\max \{b_j\}$ as elements of $\R_{\geq0}$. These relations are also contained in $(\R_{\geq0}^\pos)^\mon$. This shows that the identity map $(\T^\pos)^\mon\to(\R_{\geq0}^\pos)^\mon$ between the underlying monoids is a morphism of blueprints.
\end{proof}

%%%%%%%%%%%%%%%%%%%%%%%%%%%%%%%%%%%%%%%%%%%%%%%%%%%%%%%%%%%%%%%%%%%%%%%%%%%%%%%%%%%%%%%%%%%%%%%%%%%%%%%%%%%%%%%%%%%%%%%%%%%%%%%%%%%%%%%%%%%%%%%%%%%%%%%%%%%%%%%%%%%%%%%%%%%

\subsection{Overview of subcategories}
\label{subsection: overview of subcategories}

We denote the category of semirings by $\SRings$ and the category of rings by $\Rings$. They both form full subcategories of $\OBlpr$ by associating with a (semi)ring $R$ the blueprint $B=\bpgenquot{R^\bullet}{\cR}$ where $\cR=\{\sum a_i\=\sum b_j|\sum a_i=\sum b_j\text{ in }R\}$. Note that a semiring is a ring if and only if it is with inverses. 

Note also the following facts: a monomial blueprint that is algebraic is a monoid; a totally positive algebraic blueprint is trivial; a strictly conic blueprint with inverses is trivial.

Using the previous results on the relations of the different subcategories of $\OBlpr$, we can illustrate the subcategories of $\OBlpr$ that are relevant to this text as in Figure \ref{fig: subcategories of ordered blueprints}. An inclusion of areas indicates an inclusion of subcategories, and areas with empty intersection indicates that the only common object in the corresponding subcategories is $\{0\}$.
\begin{figure}[htb]
 \begin{center}
 \begin{tikzpicture}[scale=\textwidth/11.0cm]
  \draw[fill=gray!5,rounded corners] (0.0,0.0) rectangle (11.0,6.0);		\draw (1.5,4.7) node {$\OBlpr$};		%ordered blueprints
  \draw[fill=gray!15,rounded corners] (1.5,0.5) rectangle (9.0,3.5);		\draw (7.7,1.1) node {$\OBlpr^\conic$};		%strictly conic blueprints
  \draw[fill=gray!25,rounded corners] (3.0,2.5) rectangle (10.5,5.5);		\draw (9.0,4.7) node {$\Blpr$};			%algebraic blueprints
  \draw[fill=gray!35,rounded corners] (3.1,2.6) rectangle (7.5,5.4);		\draw (4.7,4.0) node {$\Blpr^\padd$};		%partially additive blueprints
  \draw[fill=gray!35,rounded corners] (3.7,0.6) rectangle (6.3,2.2);		\draw (5.1,1.1) node {$\OBlpr^\pos$};		%totally positive blueprints
  \draw[fill=gray!45,rounded corners] (6.0,2.7) rectangle  (7.4,5.2);		\draw (6.7,4.0) node {$\SRings$};		%semirings
  \draw[fill=gray!55,rounded corners] (3.2,4.4) rectangle  (7.3,5.3);		\draw (4.8,4.8) node {$\Blpr^\inv$};		%blueprints with -1
  \draw[fill=gray!55,rounded corners] (0.5,1.5) rectangle  (4.5,3.4);		\draw (2.6,2.1) node {$\OBlpr^\mon$};		%monomial blueprints
  \draw[fill=gray!55,rounded corners] (5.5,1.5) rectangle  (8.0,3.4);		\draw (7.1,2.1) node {$\OBlpr^\idem$};		%idempotent ordered blueprints          
  \draw[fill=gray!65,rounded corners] (6.1,4.5) rectangle  (7.2,5.1);		\draw (6.65,4.75) node {$\Rings$};		%rings
  \draw[fill=gray!65,rounded corners] (3.2,2.7) rectangle  (4.4,3.3);		\draw (3.8,3.0) node {$\Mon$};			%monoids
  \draw[dashed,rounded corners] (1.5,0.5) rectangle (9.0,3.5);		%strictly conic blueprints
  \draw[dashed,rounded corners] (3.0,2.5) rectangle (10.5,5.5);		%algebraic blueprints
  \draw[dashed,rounded corners] (3.1,2.6) rectangle (7.5,5.4);		%partially additive blueprints
  \draw[dashed,rounded corners] (3.7,0.6) rectangle (6.3,2.2);		%totally positive blueprints
  \draw[dashed,rounded corners] (6.0,2.7) rectangle (7.4,5.2);		%semirings
  \draw[dashed,rounded corners] (3.2,4.4) rectangle (7.3,5.3);		%blueprints with -1
  \draw[dashed,rounded corners] (0.5,1.5) rectangle (4.5,3.4);		%monomial blueprints
 \end{tikzpicture}
 \caption{Some relevant subcategories of $\OBlpr$}
 \label{fig: subcategories of ordered blueprints}
 \end{center} 
\end{figure}

%%%%%%%%%%%%%%%%%%%%%%%%%%%%%%%%%%%%%%%%%%%%%%%%%%%%%%%%%%%%%%%%%%%%%%%%%%%%%%%%%%%%%%%%%%%%%%%%%%%%%%%%%%%%%%%%%%%%%%%%%%%%%%%%%%%%%%%%%%%%%%%%%%%%%%%%%%%%%%%%%%%%%%%%%%%
%%%%%%%%%%%%%%%%%%%%%%%%%%%%%%%%%%%%%%%%%%%%%%%%%%%%%%%%%%%%%%%%%%%%%%%%%%%%%%%%%%%%%%%%%%%%%%%%%%%%%%%%%%%%%%%%%%%%%%%%%%%%%%%%%%%%%%%%%%%%%%%%%%%%%%%%%%%%%%%%%%%%%%%%%%%

\section{Valuations}
\label{section: valuations}

With the formalism developed in the previous section, we are ready to give the general definition of a valuation, which restricts to the different concepts of (semi)norms and valuations in particular cases. 

\begin{df}
 Let $B$ and $S$ be two ordered blueprints. A \emph{valuation of $B$ in $S$} is a morphism $v^\bullet:B^\bullet\to S^\bullet$ between the underlying monoids that admits a morphism $\tilde v:B^\mon\to S^\pos$ such that the diagram
 \[
  \xymatrix@R=1pc@C=6pc{B^\bullet \ar[d]\ar[r]^{v^\bullet} & S^\bullet \ar[d] \\ B^\mon \ar[r]^{\tilde v} & S^\pos}
 \]
 commutes. We write $v:B\to S$ for a valuation $v$ of $B$ in $S$.
\end{df}

Since the canonical morphism $S^\bullet\to S^\mon$ is a bijection, $\tilde v$ is uniquely determined by $v$. In other words, a valuation is a multiplicative map $v:B\to S$ such that the composition map
\[
 \tilde v: \ B^\mon \ \longrightarrow \ B \ \stackrel v \longrightarrow \ S \ \longrightarrow \ S^\pos
\]
is a morphism of ordered blueprints. 

Note that $v$ is uniquely determined by $\tilde v$ if $S\to S^\pos$ is a bijection. By Lemma \ref{lemma: totally positive blueprints} \eqref{item3}, this holds for strictly conic $S$, which is the case that we are interested most in this paper.

Note further that every morphism $v:B\to S$ is a valuation since the diagram
\[
 \xymatrix@R=1,5pc{B^\bullet \ar[d] \ar[rrr]^{v^\bullet} &&& S^\bullet \ar[d] \\ B^\mon \ar[r]\ar@/^1pc/[rrr]^{\tilde v} & B \ar[r]^v & S\ar[r] & S^\pos}
\]
commutes. If $B$ is monomial and $S$ totally positive, then every valuation $v:B\to S$ is a morphism.

%%%%%%%%%%%%%%%%%%%%%%%%%%%%%%%%%%%%%%%%%%%%%%%%%%%%%%%%%%%%%%%%%%%%%%%%%%%%%%%%%%%%%%%%%%%%%%%%%%%%%%%%%%%%%%%%%%%%%%%%%%%%%%%%%%%%%%%%%%%%%%%%%%%%%%%%%%%%%%%%%%%%%%%%%%%

\subsection{Seminorms}
\label{subsection: seminorms}

Let $R$ be a ring. A \emph{seminorm} on $R$ is a monoid morphism $v:R\to \R_{\geq0}$ that satisfies the \emph{triangle inequality} $v(a+b)\leq v(a)+v(b)$ for all $a,b\in R$. A \emph{non-archimedean seminorm} is a monoid morphism that satisfies the \emph{strong triangle inequality} $v(a+b)\leq \max\{v(a),v(b)\}$.

\begin{lemma} \label{lemma: seminorms as morphisms}
 A map $v:R\to\R_{\geq0}$ is a seminorm if and only if the composition
 \[
  \tilde v: \ R^\mon \ \longrightarrow \ R \ \stackrel v \longrightarrow \ \R_{\geq0} \ \longrightarrow \ \R_{\geq0}^\pos
 \]
 is a morphism of ordered blueprints. A map $v:R\to\R_{\geq0}$ is a non-archimedean seminorm if and only if the composition
 \[
  \tilde v: \ R^\mon\ \longrightarrow \ R \ \stackrel v \longrightarrow \ \T \ \longrightarrow \ \T^\pos
 \]
 is a morphism of ordered blueprints where we identify $\R_{\geq0}$ as a set with the tropical semiring $\T$. 
\end{lemma}

\begin{proof}
 Since the canonical maps $R^\mon\to R$, $\R_{\geq0}\to\R_{\geq0}^\pos$ and $\T\to \T^\pos$ are bijections, it is clear that the map $v:R\to\R_{\geq0}$ is a morphism of monoids if and only if the composition $R^\mon\to\R_{\geq0}^\pos$ or $R^\mon\to\T^\pos$, respectively, is a monoid morphism.
 
 A monoid morphism $v:R\to \R_{\geq0}$ is a seminorm if and only if it satisfies $v(b)\leq \sum v(a_i)$ for arbitrary sums $b\=\sum a_i$. The subaddition of $R^\mon$ is generated by the relations $b\leq\sum a_i$ for which $b\=\sum a_i$ in $R$. Such a relation is mapped to the relation $\tilde v(b)\leq \sum \tilde v(a_i)$, which is in the subaddition of $\R_{\geq0}^\pos$ by the triangle inequality for $v$ and Example \ref{ex: nonnegative reals and the tropical numbers}. This means that $\tilde v$ is a morphism.
 
 Assume, conversely, that $\tilde v:R^\mon\to \R_{\geq0}^\pos$ is a morphism of ordered blueprints and consider $b\=\sum a_i$ in $R$. Then we have $b\leq\sum a_i$ in $R^\mon$ and $\tilde v(b)\leq \sum \tilde v(a_i)$ in $\R_{\geq0}^\pos$, which means that $v(b)\leq\sum v(a_i)$ with respect to the natural order of $\R_{\geq0}$. This means that $\tilde v$ is a seminorm.
 
 Since the addition of the tropical semiring $\T$ is $a+b=\max\{a,b\}$, the latter claim of the lemma follows by the same argument as the former one.
\end{proof}

\begin{rem}
 The non-archimedean seminorms can be characterized as the following seminorms. By the universal property of a monomial blueprint, a morphism $R^\mon\to \R_{\geq0}^\pos$ factors uniquely into $R^\mon\to (\R_{\geq0}^\pos)^\mon\to \R_{\geq0}^\pos$. Using the morphism $(\T^\pos)^\mon \to (\R_{\geq0}^\pos)^\mon$ from Lemma \ref{lemma: the morphism from T to R}, we see that the seminorm $v:R\to \R_{\geq0}$ is non-archimedean if and only if $R^\mon\to \R_{\geq0}^\pos$ factors into 
 \[
  R^\mon \ \longrightarrow \ (\T^\pos)^\mon \ \longrightarrow \ (\R_{\geq0}^\pos)^\mon \ \longrightarrow \ \R_{\geq0}^\pos.
 \]
\end{rem}

%%%%%%%%%%%%%%%%%%%%%%%%%%%%%%%%%%%%%%%%%%%%%%%%%%%%%%%%%%%%%%%%%%%%%%%%%%%%%%%%%%%%%%%%%%%%%%%%%%%%%%%%%%%%%%%%%%%%%%%%%%%%%%%%%%%%%%%%%%%%%%%%%%%%%%%%%%%%%%%%%%%%%%%%%%%

\subsection{Krull valuations}
\label{subsection: Krull valuations}

Let $\Gamma$ be a multiplicatively written partially ordered commutative semigroup with unit $1$. We denote by $\Gamma_0$ the ordered blueprint $(\Gamma\cup\{0\},\cR)$ where $\cR$ is generated by the partial order of $\Gamma$ and the relation $0\leq 1$.

\begin{prop} \label{prop: ordered semigroups}
 The tensor product $\Gamma_\B=\Gamma_0\otimes_\Fun\B$ is a totally positive blueprint with idempotent algebraic core $\Gamma_\B^\core=(\Gamma_\B)^\core$. The canonical morphism $\Gamma_0\to\Gamma_\B$ is bijective and the canonical morphism $(\Gamma_\B^\core)^\pos\to\Gamma_\B$ is an isomorphism. If $\Gamma$ is totally ordered, then $\Gamma_\B^\core$ is a semiring with $a+b=\max\{a,b\}$.
\end{prop}

\begin{proof}
 By Corollary \ref{cor: properties of totally positive blueprints} \eqref{sub1}, $\Gamma_\B$ is totally positive, and by definition, the algebraic core of $\Gamma_\B=\bpgenquot{\Gamma_0}{1+1\=1}$ is idempotent. 
 
 Since $\Gamma_\B=\bpgenquot{\Gamma_0}{1+1\=1}$, it is clear that $\Gamma_0\to\Gamma_\B$ is surjective. Since the core of $\Gamma_0$ has the trivial preaddition $\gen{\emptyset}$, $\Gamma_0\to\Gamma_\B$ is injective.
 
 Let $a\leq b$ be a relation in $\Gamma_0$. Then we have $b\=0+b\leq a+b\leq b+b\= b$ in $\Gamma_\B$ and $a+b\= b$ in its algebraic core $B$. Consequently, $a\leq b$ in $B^\pos$. Since also the relation $1+1\=1$ is in $B$ and $B^\pos$ and the subaddition of $\Gamma_\B$ is generated by relations of the form $a\leq b$ and $1+1\=1$, the canonical morphism $B^\pos\to\Gamma_\B$ is an isomorphism.
 
 If $\Gamma$ is totally ordered, then for all $a$ and $b$, the sum $a+b\=\max\{a,b\}$ is defined by the above argument. Therefore $B$ is a semiring.
\end{proof}

Let $k$ be a field and $\Gamma$ a (multiplicatively written) totally ordered group. A \emph{Krull valuation of $k$ with value group $\Gamma$} is a surjective monoid map $v:k\to\Gamma_0$ with $v(a+b)\leq \max\{v(a),v(b)\}$.

\begin{cor} \label{cor: Krull valuations as morphisms}
 A surjective map $v:k\to \Gamma$ is a Krull valuation if and only if the composition
 \[
  \tilde v: \ k^\mon\ \longrightarrow \ k \ \stackrel v \longrightarrow \ \Gamma_0 \ \longrightarrow \ \Gamma_\B
 \]
 is a morphism of ordered blueprints.
\end{cor}

\begin{proof}
 Since $k^\mon\to k$ and $\Gamma_0\to\Gamma_\B$ are bijections, cf.\ Proposition \ref{prop: ordered semigroups}, $v$ is a monoid morphism if and only if $\tilde v$ is so. By Proposition \ref{prop: ordered semigroups}, the addition of the semiring $\Gamma_\B^\core$ is defined as $a+b=\max\{a,b\}$ (with respect to the order of $\Gamma$). Therefore the same arguments as in the proof of Lemma \ref{lemma: seminorms as morphisms} show that $v$ satisfies the strong triangle inequality if and only if $\tilde v$ maps relations $b\leq\sum a_i$ in $k^\mon$ to relations $\tilde v(b)\leq \sum \tilde v(a_i)$ in $\Gamma_\B$.
\end{proof}

\begin{rem}
 Note that usually, the group $\Gamma$ is written additively and considered with the reverse order, i.e.\ we have $v(a+b)\geq \min\{v(a),v(b)\}$ and $v(0)=\infty$. We deviate from this convention since in the context of this paper, it is more natural to work with exponential valuations.

 According to Proposition \ref{prop: ordered semigroups}, the concept of Krull valuation can be generalized by considering seminorms of $R$ in idempotent semirings $S$, which correspond to morphisms $R^\mon\to S^\pos$ of ordered blueprints. We will see in section \ref{subsection: tropicalization for an idempotent base} that the class of idempotent semirings plays a particular role for tropicalizations. This viewpoint can also be found in Macpherson's paper \cite{Macpherson13}.
\end{rem}

%%%%%%%%%%%%%%%%%%%%%%%%%%%%%%%%%%%%%%%%%%%%%%%%%%%%%%%%%%%%%%%%%%%%%%%%%%%%%%%%%%%%%%%%%%%%%%%%%%%%%%%%%%%%%%%%%%%%%%%%%%%%%%%%%%%%%%%%%%%%%%%%%%%%%%%%%%%%%%%%%%%%%%%%%%%

\subsection{Characters}
\label{subsection: characters}

If $S^\pos=\{0\}$, then a valuation $v:B\to S$ is nothing else than a monoid morphism $v^\bullet:B^\bullet\to S^\bullet$. This is the case for characters of an abelian group $G$ in a field $k$, which is a group homomorphism $G\to k^\times$.

More precisely, if we define the monoid with zero $B=G\cup\{0\}$, then $B^\bullet\simeq B^\mon$. Since a field $S=k$ is with $-1$, we have $S^\pos=0$. Since the image of $0$ is determined, we see that the association 
\[
  \begin{array}{ccc}
   \{\,\text{valuations }v:B\to k\,\} & \longrightarrow & \{\, \text{characters }\chi:G\to k\,\}\\
       v:B\to k                       & \longmapsto     &  v^\bullet\vert_G:G\to k
  \end{array}
\]
is a bijection. We can characterize unitary characters in $\C$ as valuations in the monoid $S=\S^1\cup\{0\}$, i.e.\ the association 
\[
  \begin{array}{ccc}
   \{\,\text{valuations }v:B\to S\,\} & \longrightarrow & \{\, \text{unitary characters }\chi:G\to \C\,\}\\
       v:B\to S                       & \longmapsto     &  G\hookrightarrow B\stackrel{v^\bullet}\longrightarrow S\hookrightarrow\C
  \end{array}
\]
is a bijection.

%%%%%%%%%%%%%%%%%%%%%%%%%%%%%%%%%%%%%%%%%%%%%%%%%%%%%%%%%%%%%%%%%%%%%%%%%%%%%%%%%%%%%%%%%%%%%%%%%%%%%%%%%%%%%%%%%%%%%%%%%%%%%%%%%%%%%%%%%%%%%%%%%%%%%%%%%%%%%%%%%%%%%%%%%%%
%%%%%%%%%%%%%%%%%%%%%%%%%%%%%%%%%%%%%%%%%%%%%%%%%%%%%%%%%%%%%%%%%%%%%%%%%%%%%%%%%%%%%%%%%%%%%%%%%%%%%%%%%%%%%%%%%%%%%%%%%%%%%%%%%%%%%%%%%%%%%%%%%%%%%%%%%%%%%%%%%%%%%%%%%%%

\section{Scheme theory}
\label{section: Ordered blue schemes}

In this section, we introduce the geometric framework for scheme theoretic tropicalizations, as considered in this text. The central object of this theory is an ordered blue scheme, which is a generalization of a blue scheme, as introduced in \cite{blueprints1}. Roughly speaking, an ordered blue scheme is a certain topological space together with a sheaf in $\OBlpr$. 

In the following exposition of ordered blue schemes, we omit proofs whenever they can be done in complete analogy to the corresponding facts for blueprints. For proofs in the latter case, cf.\ \cite{blueprints1}. For an alternative exposition of ordered blue schemes, cf.\ \cite{Baker-Lorscheid18}. For more details on ordered blueprints, cf.\ \cite{Lorscheid18}. For some basic examples of blue schemes, cf.\ section 4 (``Basic definitions'') of \cite{blueprintedview}.

%%%%%%%%%%%%%%%%%%%%%%%%%%%%%%%%%%%%%%%%%%%%%%%%%%%%%%%%%%%%%%%%%%%%%%%%%%%%%%%%%%%%%%%%%%%%%%%%%%%%%%%%%%%%%%%%%%%%%%%%%%%%%%%%%%%%%%%%%%%%%%%%%%%%%%%%%%%%%%%%%%%%%%%%%%%

\subsection{Localizations}
\label{subsection: localizations}

Let $B$ be an ordered blueprint with underlying monoid $A$ and subaddition $\cR$. A \emph{multiplicative subset of $B$} is a multiplicatively closed subset $S$ of $A$ that contains $1$. The \emph{localization of $B$ at $S$} is the ordered blueprint $S^{-1}B=\bpquot{S^{-1}A}{\cR_S}$ where $S^{-1}A=\{\frac as|a\in A,s\in S\}$ is the localization of the monoid $A$ at $S$, i.e.\ $\frac as=\frac{a'}{s'}$ if and only if there is a $t\in S$ such that $tsa'=ts'a$, and where 
\[\textstyle
 \cR_S \ = \ \left\langle\left. \, \sum \frac{a_i}1 \= \sum \frac{b_j}1 \, \right| \, \sum a_i\=\sum b_j \text{ in }B \, \right\rangle.
\]
The localization of $S^{-1}B$ comes together with a canonical morphism $B\to S^{-1}B$ that sends $a$ to $\frac a1$. 

We say that a morphism $B\to C$ is a \emph{localization} if there is a multiplicative subset $S$ of $B$ such that $C\simeq S^{-1}B$ and $B\to C$ corresponds to the canonical morphism $B\to S^{-1}B$ under this isomorphism. We say that $B\to C$ is a \emph{finite localization} if the multiplicative subset $S$ of $B$ can be chosen to be finitely generated, i.e.\ one can choose finitely many generators $s_1,\dotsc,s_n$ in $S$ such that every other element $t\in S$ is a product of powers of $s_1,\dotsc,s_n$. If $S$ is generated by $s_1,\dotsc,s_n$, then $S^{-1}B=B[h^{-1}]$ where $h=s_1\dotsb s_n$ and $B[h^{-1}]=\tilde S^{-1}B$ for $\tilde S=\{h^i\}_{i\geq0}$.

%%%%%%%%%%%%%%%%%%%%%%%%%%%%%%%%%%%%%%%%%%%%%%%%%%%%%%%%%%%%%%%%%%%%%%%%%%%%%%%%%%%%%%%%%%%%%%%%%%%%%%%%%%%%%%%%%%%%%%%%%%%%%%%%%%%%%%%%%%%%%%%%%%%%%%%%%%%%%%%%%%%%%%%%%%%

\subsection{Ordered blueprinted spaces}
\label{subsection: Ordered blueprinted spaces}

An \emph{ordered blueprinted space}, or for short an \emph{$\OBlpr$-space}, is a topological space $X$ together with a sheaf $\cO_X$ in $\OBlpr$. In practice, we suppress the \emph{structure sheaf $\cO_X$} from the notation and denote an $\OBlpr$-space by the same symbol $X$ as its \emph{underlying topological space}. For every point $x$ of $X$, the \emph{stalk in $x$} is the colimit $\cO_{X,x}=\colim\cO_X(U)$ over the system of all open neighbourhoods $U$ of $x$. 

A \emph{morphism of $\OBlpr$-spaces} is a continuous map $\varphi:X\to Y$ between the underlying topological spaces together with a morphism $\varphi^\#:\varphi^{-1}\cO_Y\to\cO_X$ of sheaves on $X$ that is \emph{local} in the following sense: for every $x\in X$ and $y=\varphi(x)$, the induced morphism $\cO_{Y,y}\to\cO_{X,x}$ of stalks sends non-units to non-units. This defines the category $\OBlprSp$ of ordered blueprinted spaces.

Note that we avoid the notion of ``locally'' ordered blueprinted spaces (i.e.\ stalks have a unique maximal ideal) since this depends on the notion of (maximal) ideal of an ordered blueprint, for which there are different reasonable choices. It might comfort the reader to see that our choice of maximal ideal in section \ref{subsection: ideals} has the implication that every ordered blueprinted space is local.

%%%%%%%%%%%%%%%%%%%%%%%%%%%%%%%%%%%%%%%%%%%%%%%%%%%%%%%%%%%%%%%%%%%%%%%%%%%%%%%%%%%%%%%%%%%%%%%%%%%%%%%%%%%%%%%%%%%%%%%%%%%%%%%%%%%%%%%%%%%%%%%%%%%%%%%%%%%%%%%%%%%%%%%%%%%

\subsection{Ideals}
\label{subsection: ideals}

In this section, we introduce $m$-ideals for ordered blueprints. The ``$m$'' stems from ``monoid'', and we use this prefix in order to distinguish this notion from other notions of ideals that are of relevance for ordered blueprints. The terminology is consistent with that used in \cite{Lorscheid18} and \cite{Baker-Lorscheid18}. More details on ideals for ordered blueprints can be found in \cite{Lorscheid18}.

Let $B$ be an ordered blueprint. An \emph{$m$-ideal of $B$} is a subset $I$ of $B$ such that $0\in I$ and $IB=I$. An $m$-ideal $I$ is \emph{prime} if its complement $S=B-I$ is a multiplicative subset. The \emph{localization of $B$ at $\fp$} is $B_\fp=S^{-1}B$. 

Note that every ordered blueprint $B$ has a unique maximal $m$-ideal, which is the complement $\fm=B-B^\times$ of the unit group of $B$. This maximal $m$-ideal is prime since $B^\times$ is a multiplicative subset of $B$. In other words, every ordered blueprint is \emph{local}. The maximal $m$-ideal of the localization $B_\fp$ at a prime $m$-ideal $\fp$ of $B$ is $\fp B_\fp$. 

Given an $m$-ideal $I$ of $B$, we can form the quotient $B/I=\bpgenquot {B}{a\sim 0|a\in I}$. The surjection $\pi_I:B\to B/I$ is universal among all morphisms of ordered blueprints that map $I$ to $0$. We have that $I\subset\pi_I^{-1}(0)$, but it is not true in general that this inclusion is an equality. In particular, it can happen for a proper $m$-ideal $I$ of $B$ that $B/I$ is the trivial ordered blueprint, cf.\ \cite[Cor.\ 5.9.9]{Lorscheid18}.

We cite the following fact from \cite[Prop.\ 5.9.6]{Lorscheid18}.

\begin{prop}\label{prop: prime ideals in localizations}
 Let $B$ be an ordered blueprint, $S$ a multiplicative subset and $\iota_S:B\to S^{-1}B$ the localization map. Then taking inverse images along $\iota_S$ defines a bijection
 \[
   \big\{\, \text{prime $m$-ideals of }S^{-1}B \, \big\}   \quad \longleftrightarrow \quad   \big\{\, \text{prime $m$-ideals $\fp$ of $B$ with }\fp\cap S=\emptyset\,\big\}.
 \]
\end{prop}

%%%%%%%%%%%%%%%%%%%%%%%%%%%%%%%%%%%%%%%%%%%%%%%%%%%%%%%%%%%%%%%%%%%%%%%%%%%%%%%%%%%%%%%%%%%%%%%%%%%%%%%%%%%%%%%%%%%%%%%%%%%%%%%%%%%%%%%%%%%%%%%%%%%%%%%%%%%%%%%%%%%%%%%%%%%

\subsection{The spectrum}
\label{subsection: the spectrum}

Let $B$ be an ordered blueprint. We define the \emph{spectrum $\Spec B$ of $B$} as the following ordered blueprinted space. The topological space of $X=\Spec B$ consists of the prime $m$-ideals of $B$ and comes with the topology generated by the \emph{principal opens}
\[
 U_h \ = \ U_{B,h} \ = \ \{\, \fp\in\Spec B \,|\, h\notin \fp\, \}
\]
where $h$ varies through the elements of $B$. Note that the principal opens form a basis of the topology for $X$ since $U_h\cap U_g=U_{gh}$. 

Let $U$ be an open subset of $X$. A \emph{section on $U$} is a function $s:U\to \coprod_{\fp\in U}B_\fp$ such that $s(\fp)\in B_\fp$, such that there is a finite open covering $\{U_{h_i}\}$ of $U$ by principal open subsets $U_{h_i}$ and such that there are elements $a_i\in B[h_i^{-1}]$ whose respective images in $B_\fp$ equal $s(\fp)$ whenever $\fp\in U_{h_i}$. The set of sections $O_{X}(U)$ on $U$ comes naturally with the structure of an ordered blueprint. This defines the \emph{structure sheaf $\cO_{X}$} of $X$ and completes the definition of $X=\Spec B$ as an $\OBlpr$-space.

Note that the maximal $m$-ideal $\fm$ of $B$ is the unique closed point of $\Spec B$. The following fact has a much simpler proof than its classical analogue for the spectrum of a ring, cf.\ \cite[Prop.\ 4.4]{Baker-Lorscheid18}.

\begin{prop}\label{prop: sheaf properties of Spec B}
 Let $B$ be an ordered blueprint and $X=\Spec B$. Then $\cO_X(U_h)=B[h^{-1}]$ for every $h\in B$. In particular, $\cO_X(X)=B$.
\end{prop}

As usual, a morphism $f:B\to C$ of ordered blueprints defines a morphism $f^\ast:\Spec C\to \Spec B$ of $\OBlpr$-spaces by taking the inverse image of prime $m$-ideals and pulling back sections. This defines the contravariant functor
\[
 \Spec: \ \OBlpr \ \longrightarrow \ \OBlprSp.
\]
We call $\OBlpr$-spaces in the essential image of this functor \emph{affine ordered blue schemes}. 

The \emph{global section functor}
\[
 \Gamma: \ \OBlprSp \ \longrightarrow \ \OBlpr.
\]
that sends a locally blueprinted space $X$ to its ordered blueprint $\Gamma X=\cO_X(X)$ of global sections is a left-inverse of $\Spec$.

%%%%%%%%%%%%%%%%%%%%%%%%%%%%%%%%%%%%%%%%%%%%%%%%%%%%%%%%%%%%%%%%%%%%%%%%%%%%%%%%%%%%%%%%%%%%%%%%%%%%%%%%%%%%%%%%%%%%%%%%%%%%%%%%%%%%%%%%%%%%%%%%%%%%%%%%%%%%%%%%%%%%%%%%%%%

\subsection{Ordered blue schemes}
\label{subsection: ordered blue schemes}

Every open subset $U$ of an ordered blueprinted space $X$ is naturally an ordered blueprinted space with respect to the restriction of the structure sheaf $\cO_X$ to $U$. An \emph{affine open of $X$} is an open subset that is isomorphic to the spectrum of an ordered blueprint. An \emph{ordered blue scheme} is an ordered blueprinted space $X$ such that every point has an affine open neighbourhood. A morphism of ordered blue schemes is a morphism of $\OBlpr$-spaces. We denote the category of ordered blue schemes by $\OBSch$. 

We collect some facts about ordered blue schemes; for proofs cf.\ \cite{blueprints1} and \cite[Thm.\ 4.11]{Baker-Lorscheid18}. The category $\OBSch$ contains all finite limits. The fibre product of ordered blue schemes is constructed as in usual scheme theory. In particular, the fibre product of a diagram $X\to Z\leftarrow Y$ of affine ordered blue schemes is represented by $\Spec\bigl(\Gamma X\otimes_{\Gamma Z} \Gamma Y\bigr)$. Note that $\Spec \Fun$ is a terminal object.

More generally, we define for any ordered blueprint $B$ the category $\OBSch_B$ of \emph{ordered blue $B$-schemes} as the category whose objects are morphisms $X\to\Spec B$ of ordered blue schemes and whose morphisms are morphisms $X\to Y$ of ordered blue schemes that commute with the structure maps $X\to\Spec B$ and $Y\to\Spec B$.

The global section functor $\Gamma:\OBSch\to\OBlpr$ is adjoint to $\Spec:\OBlpr\to\OBSch$, i.e.\ $\Hom(X,\Spec B)=\Hom(B,\Gamma X)$ for all ordered blueprints $B$ and ordered blue schemes $X$; cf.\ \cite[Rem.\ 4.9]{Baker-Lorscheid18}.

Every morphism $\varphi:X\to Y$ of ordered blue schemes is \emph{locally algebraic}, by which we mean that $X$ and $Y$ have affine open coverings $\{U_i\}$ and $\{V_i\}$, respectively, such that $\varphi$ restricts for every $i$ to a morphism $\varphi_i:U_i\to V_i$ between affine ordered blue schemes that is induced by a morphism $f_i:\Gamma V_i\to\Gamma U_i$ of ordered blueprints.

%%%%%%%%%%%%%%%%%%%%%%%%%%%%%%%%%%%%%%%%%%%%%%%%%%%%%%%%%%%%%%%%%%%%%%%%%%%%%%%%%%%%%%%%%%%%%%%%%%%%%%%%%%%%%%%%%%%%%%%%%%%%%%%%%%%%%%%%%%%%%%%%%%%%%%%%%%%%%%%%%%%%%%%%%%%

\subsection{Stalks and residue fields}
\label{subsection: stalks and residue fields}

Let $X$ be an ordered blue scheme and $x$ a point of $X$. The \emph{stalk of $\cO_X$ in $x$} is the colimit $\cO_{X,x}=\colim\cO_X(U)$ over the system of all open neighbourhoods $U$ of $x$. The stalk only depends on the subsets of an affine open neighbourhood of $x$, so we can assume that $X=\Spec B$ is affine and that $x=\fp$ is a prime $m$-ideal of $B$. In this case, we have $\cO_{X,x}=B_\fp$.

Let $\fm_x$ be the maximal $m$-ideal of $\cO_{X,x}$. The \emph{residue field at $x$} is $k(x)=\cO_{X,x}/\fm_x$, which is either trivial or an ordered blue field; cf.\ \cite[Cor.\ 5.9.9]{Lorscheid18}. A morphism $\varphi:X\to Y$ of affine ordered blue schemes induces morphisms $\varphi_x:\cO_{Y,y}\to\cO_{X,x}$ and $\kappa(y)\to\kappa(x)$ of ordered blueprints for every $x\in X$ and $y=\varphi(x)$.

%%%%%%%%%%%%%%%%%%%%%%%%%%%%%%%%%%%%%%%%%%%%%%%%%%%%%%%%%%%%%%%%%%%%%%%%%%%%%%%%%%%%%%%%%%%%%%%%%%%%%%%%%%%%%%%%%%%%%%%%%%%%%%%%%%%%%%%%%%%%%%%%%%%%%%%%%%%%%%%%%%%%%%%%%%%

\subsection{Open and closed immersions}
\label{subsection: closed immersions}

A morphism $\varphi:X\to Y$ of ordered blue schemes is an \emph{open immersion} if it is an open topological embedding and if $\cO_X$ is the restriction of $\cO_Y$ to $X$. We say that $\varphi:X\to Y$ is isomorphic to another open immersion $\varphi:X'\to Y$ if there is an isomorphism $X\to X'$ commuting with $\varphi$ and $\varphi'$. An \emph{open (ordered blue) subscheme of $Y$} is an isomorphism class of open immersions into $Y$. Note that the open subschemes of $Y$ correspond bijectively to the open subsets of $Y$.

A morphism $\varphi:X\to Y$ of ordered blue schemes is \emph{affine} if for every open immersion $U\to Y$ from an affine ordered blue scheme $U$ to $Y$, the inverse image $\varphi^{-1}(U)=U\times_YX$ is affine. A morphism $\varphi:X\to Y$ of ordered blue schemes is a \emph{closed immersion} if it is affine and if $\cO_X(U)\to \cO_Y(\varphi^{-1}(U))$ is surjective for every open immersion $U\to Y$ from an affine ordered blue scheme $U$ to $Y$. 

We say that $\varphi:X\to Y$ is isomorphic to another closed immersion $\varphi:X'\to Y$ if there is an isomorphism $X\to X'$ commuting with $\varphi$ and $\varphi'$. A \emph{closed (ordered blue) subscheme of $X$} is an isomorphism class of closed immersions into $X$.

Note that the image of a closed immersion $X\to Y$ does not have to be a closed subset of $Y$. For example, the diagonal embedding $\A_B^1\to\A_B^2$ is a closed immersion, but its image is not a closed subset; cf.\ \cite[Rem.\ 1.8]{blueprints2}.

%%%%%%%%%%%%%%%%%%%%%%%%%%%%%%%%%%%%%%%%%%%%%%%%%%%%%%%%%%%%%%%%%%%%%%%%%%%%%%%%%%%%%%%%%%%%%%%%%%%%%%%%%%%%%%%%%%%%%%%%%%%%%%%%%%%%%%%%%%%%%%%%%%%%%%%%%%%%%%%%%%%%%%%%%%%

\subsection{Affine open coverings}
\label{subsection: affine open coverings}

Let $X$ be an ordered blue scheme. As in usual scheme theory, there is a maximal affine open covering of $X$ that consists of all affine open subsets of $X$. But the fact that every affine ordered blue scheme has a unique closed point has some remarkable consequences, which allows for some particularly simple techniques in the theory of ordered blue schemes.

Let $X$ be an ordered blue scheme and $x$ and $y$ two points of $X$. We say that $y$ is a \emph{generalization of $x$} or that $x$ is a \emph{specialization of $y$}, and write $y\leq x$, if $x$ is contained in the closure of $y$. We say that an ordered blue $X$ has \emph{enough closed points} if every point $y$ of $X$ is the generalization of a closed point $x$ of $X$. We denote the set of closed points of $X$ by $\norm X$.

\begin{ex}
 Not every ordered blue scheme has enough closed points. The following is an example of a quasi-affine ordered blue scheme without any closed point. This example is similar to Schwede's example of a scheme without closed points in \cite{Schwede05}, though our construction is more immediate. 
 
 Let $B=\bpgenquot{\Fun[T_i|i\in\N]}{T_i\=T_{i}T_{i+1}|i\in\N}$ and $X=\Spec B$. Then the points of $X$, are the prime ideals $\fp_i=(T_i)=(T_1,\dotsc,T_i)$ for $i\in \N$ and the maximal ideal $\fp_\infty=(T_i|i\in\N)$, and we have $\fp_i\subset \fp_j$ if and only if $i\leq j$. The principal open subset $U_{T_i}=\{\fp\subset B|T_i\notin\fp\}$ contains $\fp_j$ if and only if $j<i$. Thus the open subscheme $U=\bigcup_{i\in\N} U_{T_i}$ of $X$ contains all $\fp_i$ for $i\in\N$, but not the closed point $\fp_\infty$ of $X$. Thus $U$ is a quasi-affine ordered blue scheme without a closed point.
\end{ex}

Such examples can be considered pathological from our point of view. If the ordered blue scheme $X$ is sufficiently nice, then it has enough closed points, as we will show in the following lemma. We say that $X$ is \emph{topologically Noetherian} if the underlying topological space of $X$ is Noetherian. We say that $X$ is \emph{quasi-compact} if its underlying topological space is compact. Since the spectrum of an ordered blueprint is quasi-compact, $X$ is quasi-compact if and only if it has a finite covering by affine open subschemes. Note that a topological space is Noetherian if and only if every subspace is compact. In particular, a topologically Noetherian ordered blue scheme is quasi-compact.

\begin{lemma}\label{lemma: quasi-compact and topologically Noetherian imply the existence of closed points}
 Let $X$ be an ordered blue scheme. If $X$ is quasi-compact, then $X$ has enough closed points. If $X$ is topologically Noetherian, then every open subset $U$ of $X$ has enough closed points.
\end{lemma}

\begin{proof}
 If $X$ is quasi-compact, then $X$ is covered by a finite number of affine open subschemes $U_1,\dotsc,U_n$. Since $U_i$ is affine, it has a unique closed point $x_i$. Thus the closed points of $X$ form a subset of $\{x_1,\dotsc,x_n\}$ and every point of $X$ is the generalization of one of these closed points. This proves the first assertion. 
 
 If $X$ is topologically Noetherian and $U$ an open subset of $X$, then $U$ is quasi-compact and has enough closed points by what we have proven before.
\end{proof}

\begin{lemma}\label{lemma: affine opens and their closed points}
 Let $X$ be an ordered blue scheme and $U\subset X$ an open subset. Then $U$ is affine if and only if it has a unique closed point $x$ and if it has enough closed points. In this case, $U=\{y\in X|y\leq x\}=\Spec\cO_{X,x}$.
\end{lemma}

\begin{proof}
 If $U=\Spec B$ is affine, then the unique maximal $m$-ideal of $B$ is the unique closed point of $U$, and every other point is a generalization of this closed point. 

 Conversely, assume that $U$ has a unique closed point $x$. Since $U$ is open, it contains all generalizations $y$ of $x$, and by assumption, every point $y$ of $U$ is a generalization of $x$. Thus $U=\{y\in X|y\leq x\}$. Proposition \ref{prop: prime ideals in localizations} implies that $\{y\in X|y\leq x\}=\Spec\cO_{X,x}$, which verifies that $U$ is affine.
\end{proof}

\begin{prop}\label{prop: minimal covering of ordered blue schemes}
 Let $X$ be an ordered blue scheme. Sending an affine open subset $U$ of $X$ to its unique closed point defines an injection
 \[
  \Phi: \big\{\,\text{affine open subsets }U\subset X\ \big\} \quad \longrightarrow \quad X
 \]
 whose image consists of those points $x\in X$ that possess an affine open neighbourhood $V$ such that the restriction map $\Gamma V\to\cO_{X,x}$ extends to an isomorphism $\Gamma V[h^{-1}]\to\cO_{X,x}$ for some $h\in \Gamma V$.
\end{prop}

\begin{proof}
 The association $\Phi$ is injective since $U=\{y\in X|y\leq x\}$ is determined by $x$, which verifies the first assertion. 
 
 If $x$ is the image of $\Phi$, i.e.\ it is the unique closed point of an open subset $U$ of $X$, then we can use $V=U=\Spec\cO_{X,x}$ and obtain an isomorphism $\Gamma V[1^{-1}]\to\cO_{X,x}$. Conversely, if there is an isomorphism $\Gamma V[h^{-1}]\to\cO_{X,x}$ as in the proposition, then $U=\Spec\Gamma V[h^{-1}]$ is an affine open subset of $V$ that contains $x$ as its unique closed point. Thus $x$ is in the image of $\Phi$. This verifies the second assertion and completes the proof.
\end{proof}

If $x\in X$ is in the image of $\Phi$, then we denote the affine open subset of $X$ that has $x$ as its unique closed point by $U_x$. 

\begin{cor}\label{cor: closed points are closed points of affine opens}
 Every closed point of $X$ is in the image of $\Phi$. If $X$ has enough closed points, then every affine open subset $V$ of $X$ is contained in $U_x$ for some closed point $x$ of $X$ and $\{U_x|x\in\norm X\}$ is the unique minimal open affine covering of $X$.
\end{cor}

\begin{proof}
 If $x$ is a closed point of $X$ and $U$ an affine open neighbourhood of $x$, then $x$ is also closed in $U$ and therefore $\Phi(U)=x$. 
 
 Assume $X$ has enough closed points. Consider an affine open subset $V$ of $X$ and let $y=\Phi(V)$. Then $y$ is contained in $U_x$ for some $x$, i.e.\ $y$ is the generalization of $x$. Then $V\subset U_x$.
 
 If $x\in X$ is a closed point, then $U_x$ is the only affine open subset containing $x$. Thus every affine open covering family of $X$ must contain $\cU=\{U_x|x\in\norm X\}$. Since $X$ has enough closed points, $X$ is covered by $\cU$.
\end{proof}

\begin{cor}\label{cor: image of affine open is contained in affine open}
 Let $\varphi:X\to Y$ be a morphism of ordered blue schemes and assume that $Y$ has enough closed points. Let $U$ be an affine open subset of $X$. Then $\varphi(U)$ is contained in an affine open subset of $Y$.
\end{cor}

\begin{proof}
 Let $x$ be the unique closed point of $U$ and $y=\varphi(x)$. Then $y$ is the generalization of a closed point $y'$ of $Y$. Since $z\leq x$ implies $\varphi(z)\leq y\leq y'$, we have $\varphi(U)=\{\varphi(z)|z\leq x\}\subset U_{y'}$, as desired.
\end{proof}

%%%%%%%%%%%%%%%%%%%%%%%%%%%%%%%%%%%%%%%%%%%%%%%%%%%%%%%%%%%%%%%%%%%%%%%%%%%%%%%%%%%%%%%%%%%%%%%%%%%%%%%%%%%%%%%%%%%%%%%%%%%%%%%%%%%%%%%%%%%%%%%%%%%%%%%%%%%%%%%%%%%%%%%%%%%
%%%%%%%%%%%%%%%%%%%%%%%%%%%%%%%%%%%%%%%%%%%%%%%%%%%%%%%%%%%%%%%%%%%%%%%%%%%%%%%%%%%%%%%%%%%%%%%%%%%%%%%%%%%%%%%%%%%%%%%%%%%%%%%%%%%%%%%%%%%%%%%%%%%%%%%%%%%%%%%%%%%%%%%%%%%

\section{Endofunctors and base extension to semiring schemes}
\label{section: endofunctors and base extension to semiring schemes}

In this section we will extend the diverse endofunctors $(-)^\dagger:\OBlpr\to\OBlpr$ where $\dagger\in\{\hull,\inv,\idem,\pos,\conic,\core,\bullet,\mon,\padd,+\}$ to functors on ordered blue schemes. All of these functors take values in $\OBSch$ itself, with the exception of $(-)^+$, which we understand as a base extension from ordered blue schemes to semiring schemes.

%%%%%%%%%%%%%%%%%%%%%%%%%%%%%%%%%%%%%%%%%%%%%%%%%%%%%%%%%%%%%%%%%%%%%%%%%%%%%%%%%%%%%%%%%%%%%%%%%%%%%%%%%%%%%%%%%%%%%%%%%%%%%%%%%%%%%%%%%%%%%%%%%%%%%%%%%%%%%%%%%%%%%%%%%%%

\subsection{Affine presentations}
\label{subsection: affine presentation}

A useful tool for extending functors from ordered blueprints to ordered blue schemes are affine presentations. These are diagrams of affine ordered blue schemes and open immersions whose colimit is an ordered blue scheme. We review the definitions and some results from \cite{Lorscheid17} in a slightly simplified and restricted form, which is sufficient for our purposes.

Let $X$ be an ordered blue scheme. An \emph{affine presentation of $X$} is a category $\cU$ that consists of affine open subschemes $U$ of $X$ together with all inclusions between them such that the canonical morphism $\colim\cU\to X$ is an isomorphism where $\colim\cU$ is the colimit in $\OBSch$. An \emph{affine presentation} is a category $\cU$ of affine ordered blue schemes and open immersions such that $X=\colim\cU$ exists in the category of ordered blue schemes and such that $\cU$ is an affine presentation of $X$. An example of an affine presentation of $X$ is the family of all open affine subschemes together with all inclusions among them.

Let $\cU$ and $\cV$ be affine presentations. A \emph{morphism of affine presentations} is a functor $\Phi:\cU\to\cV$ together with a family of morphisms $\{\Phi_U:U\to \Phi(U)\}$, indexed by the objects $U$ of $\cU$, such that for every morphism $\varphi:U_1\to U_2$ in $\cU$ the resulting square
 \[
  \xymatrix@R=1pc@C=4pc{U_1 \ar[d]_{\varphi}\ar[r]^{\Phi_{U_1}} & \Phi(U_1)\ar[d]^{\Phi(\varphi)} \\ U_2 \ar[r]^{\Phi_{U_2}} & \Phi(U_2)}
 \]
commutes. A morphism of affine presentations $\Phi:\cU\to\cV$ induces a morphism $\colim\Phi:\colim\cU\to\colim \cV$ of ordered blue schemes. Conversely, every morphism of ordered blue schemes comes from a morphism of affine presentations in this way.

We say that an endofunctor $\cF:\OBlpr\to\OBlpr$ \emph{preserves finite localizations} if $\cF(B)\to\cF(C)$ is a finite localization for every finite localization $B\to C$. The key tool to extend functors from ordered blueprints to ordered blue schemes is the following. We denote by $\OBAff$ the category of affine ordered blue schemes, which is, by definition, the opposite category of $\OBlpr$.

\begin{lemma}\label{lemma: extension of a morphism of sites to a functor between schemes} 
 Let $\cG:\OBAff\to\OBAff$ be a functor that commutes with fibre products and that preserves finite localizations. Then there exists a unique functor $\overline \cG:\OBSch\to \OBSch$ such that for all morphisms $\Phi:\cU\to\cV$ of affine presentations in $\OBAff$, we have a natural identification $\overline\cG(\colim\Phi)=\colim\cG(\Phi)$. In particular, this yields $\overline\cG(\colim\cU)=\colim\cG(\cU)$.
\end{lemma}

\begin{proof}
 This statement follows from Lemma 1.3 in \cite{Lorscheid17} once that we can show that $\cG$ preserves covering families, i.e.\ if $\{U_i\}$ is an affine open covering of $X=\Spec B$, then $\{\cG(U_i)\}$ is an affine open covering of $\cG(X)$. It suffices to prove this for coverings of $X$ by principal opens $U_{i}=\Spec B[h_i^{-1}]$. Since $\cG$ commutes with finite localizations, we have $\cG(U_i)=\Spec C[g_i^{-1}]$ for some $g_i\in C$ where $C=\Gamma\cG(X)$. Thus $\{\cG(U_i)\}$ is a family of open subsets of $\cG(X)$.
 
 Since $X$ has a unique closed point, namely the maximal $m$-ideal of $B$, $\{U_i\}$ is an affine open covering of $X$ if and only if $U_i=X$ for some $i$. In this case, we have $\cG(U_i)=\cG(X)$, which shows that $\{\cG(U_i)\}$ is an affine open covering of $\cG(X)$. Thus we can apply \cite[Lemma 1.3]{Lorscheid17}, which concludes the proof.
\end{proof}

%%%%%%%%%%%%%%%%%%%%%%%%%%%%%%%%%%%%%%%%%%%%%%%%%%%%%%%%%%%%%%%%%%%%%%%%%%%%%%%%%%%%%%%%%%%%%%%%%%%%%%%%%%%%%%%%%%%%%%%%%%%%%%%%%%%%%%%%%%%%%%%%%%%%%%%%%%%%%%%%%%%%%%%%%%%

\subsection{Endofunctors}
\label{subsection: endofunctors}

Some of the subcategories $\cC$ of $\OBlpr$ from section \ref{subsection: overview of subcategories} extend to subcategories of $\OBSch$ as (co)reflective subcategories. 

\begin{lemma}\label{lemma: endofunctors respect finite localizations}
 The endofunctors $(-)^+$, $(-)^\hull$, $(-)^\inv$, $(-)^\idem$, $(-)^\pos$, $(-)^\conic$, $(-)^\core$, $(-)^\bullet$, $(-)^\mon$ and $(-)^\padd$ preserve finite localizations.
\end{lemma}

\begin{proof}
 Let $S$ be a finitely generated multiplicative subset of $B$. If $\cF$ is one of the functors $(-)^+$, $(-)^\hull$, $(-)^\inv$, $(-)^\idem$, $(-)^\pos$ or $(-)^\conic$, then $\cF(B)$ comes together with a canonical morphism $\varphi: B\to \cF(B)$, and $\cF(S^{-1}B) =\bigl(\varphi(S)\bigr)^{-1}\cF(B)$. 
 
 If $\cF$ is one of the functors $(-)^\core$, $(-)^\bullet$, $(-)^\mon$ or $(-)^\padd$, then $\cF(B)$ comes together with a canonical bijection $\cF(B)\to B$, and $\cF(S^{-1}B)=S^{-1}\cF(B)$. 
\end{proof}

\begin{thm}\label{thm: endofunctors of ordered blue schemes}
 The functors $(-)^\hull$, $(-)^\inv$, $(-)^\idem$, $(-)^\pos$, $(-)^\conic$, $(-)^\core$, $(-)^\bullet$, $(-)^\mon$ and $(-)^\padd$ extend to endofunctors of the category of ordered blue schemes and satisfy the following properties.
 \begin{enumerate}
  \item\label{endo1} For $\dagger\in\{\hull,\inv,\idem,\pos,\conic\}$, $(-)^\dagger:\OBlpr\to\OBlpr$ extends to an idempotent endofunctor $(-)^\dagger:\OBSch\to \OBSch$ that comes with a canonical morphism $X^\dagger\to X$ for every ordered blue scheme $X$. Let $\OBSch^\dagger$ be the essential image of $(-)^\dagger$ and let $\iota:\OBSch^\dagger\to\OBSch$ be the inclusion functor. Then the restriction $(-)^\dagger:\OBSch\to\OBSch^\dagger$ is right adjoint and left inverse to $\iota$.
  \item\label{endo2} For $\dagger\in\{\core,\bullet,\mon,\padd\}$, $(-)^\dagger:\OBlpr\to\OBlpr$ extends to an idempotent endofunctor $(-)^\dagger:\OBSch\to \OBSch$ that comes with a canonical morphism $X\to X^\dagger$ for every ordered blue scheme $X$. Let $\OBSch^\dagger$ be the essential image of $(-)^\dagger$ and let $\iota:\OBSch^\dagger\to\OBSch$ be the inclusion functor. Then the restriction $(-)^\dagger:\OBSch\to\OBSch^\dagger$ is left adjoint and left inverse to $\iota$.
 \end{enumerate}
\end{thm}

\begin{proof}
 We begin with verifying that all functors in question commute with fibre products. For $\dagger$ in $\{\hull,\inv,\idem,\pos,\conic\}$ and ordered blueprints $B$ and $C$ with $C=C^\dagger$, we have that $\Hom(B,C)=\Hom(B^\dagger,C)$. This means that $(-)^\dagger:\OBlpr\to\OBlpr^\dagger$ is left adjoint to the embedding $\OBlpr^\dagger\to\OBlpr$, and that the corresponding functor $(-)^\dagger:\OBAff\to\OBAff^\dagger$ between the dual categories is a right adjoint and therefore commutes with fibre products. It is obvious that $(-)^\core:\OBlpr\to\OBlpr$ commutes with tensor products, and consequently, its geometric version $(-)^\core:\OBAff\to\OBAff$ commutes with fibre products.

 Since all functors $(-)^\dagger$ considered above preserve finite localizations by Lemma \ref{lemma: endofunctors respect finite localizations} and commute with fibre products, Lemma \ref{lemma: extension of a morphism of sites to a functor between schemes} extends $(-)^\dagger$ to a functor $(-)^\dagger:\OBSch\to\OBSch$.
 
 We have to include an additional argument on the construction of $(-)^\bullet$, $(-)^\mon$ or $(-)^\padd$ since they do not preserve fibre products, cf.\ Example \ref{ex: mon bullet padd do not commute with tensor products}. Let $(-)^\dagger:\OBlpr\to\OBlpr$ be one of these functors and $X$ an ordered blue scheme. Then we define $X^\dagger$ as follows.
 
 We begin with the definition for objects. Let $X$ be an ordered blue scheme. The underlying topological space of $X^\dagger$ is that of $X$. For every open subset $U$ of $X$, we define $\cO_{X^\dagger}(U)=\cO_X(U)^\dagger$. That this is a well-defined ordered blue scheme can be seen as follows. Since for every ordered blueprint and $\dagger\in\{\bullet,\mon,\padd\}$, the underlying monoid of $B^\dagger$ agrees with that of $B$ and since a subset $\fp$ of $B$ is a prime $m$-ideal of $B$ if and only if it is a prime $m$-ideal of $B^\dagger$, we conclude that the underlying topological spaces of $\Spec B$ and $\Spec B^\dagger$ agree. By Lemma \ref{lemma: endofunctors respect finite localizations}, $(-)^\dagger$ commutes with localizations, and thus an inclusion $V\to U$ of affine open subsets of $X$ leads to a morphism $V^\dagger\to U^\dagger$ that comes from a finite localization $\Gamma U^\dagger\to \Gamma V^\dagger$, which shows that $X^\dagger$ is well-defined.
 
 We continue with the definition of $(-)^\dagger$ for morphisms $\varphi:X\to Y$. By Corollary \ref{cor: image of affine open is contained in affine open}, the image of an affine open set $U=\Spec B$ of $X$ is contained in an affine open set $V=\Spec C$ of $Y$. Then the restriction $\varphi_{U,V}:U\to V$ of $\varphi$ induces a morphism $\Gamma\varphi_{U,V}:C\to B$ and thus $(\Gamma\varphi_{U,V}^\dagger)^\ast:U^\dagger\to V^\dagger$. If we can show that this definition is independent from the affine open $V$ that contains $\varphi(U)$, then this local description of $\varphi^\dagger$ is well-defined. Indeed, if $W$ is another affine open subset of $Y$ that contains $\varphi(U)$, then $\varphi(U)$ is contained in the intersection $V\cap W$. Since $U$ and thus $\varphi(U)$ has a unique closed point, say $z$, every affine open neighbourhood of $z$ contains $\varphi(U)$. In particular, there is an affine open subset $Z$ of $V\cap W$ that contains $\varphi(U)$. Thus both $(\Gamma\varphi_{U,V}^\dagger)^\ast:U^\dagger\to V^\dagger$ and $(\Gamma\varphi_{U,W}^\dagger)^\ast:U^\dagger\to W^\dagger$ factor through $(\Gamma\varphi_{U,Z}^\dagger)^\ast:U^\dagger\to Z^\dagger$. This shows that the above description of $\varphi^\dagger$ is independent of the choice of the affine open $V$ that contains $\varphi(U)$. Thus $\varphi^\dagger$ is well-defined, which completes the definition of the functor $(-)^\dagger:\OBSch\to\OBSch$.
 
 To conclude the proof of the theorem, we note that the adjointness properties follow from the corresponding properties of $(-)^\dagger$ on affine ordered blue schemes and the fact that every morphism between ordered blue schemes is determined by its restrictions to affine opens.
\end{proof}

As a consequence, the essential images of the functors considered in Theorem \ref{thm: endofunctors of ordered blue schemes} define a variety of subcategories of $\OBSch$. If $\dagger$ is in $\{\hull,\inv,\idem,\pos,\conic,\core,\bullet,\mon,\padd\}$, then $\OBSch^\dagger$ is (co)reflective in $\OBSch$, by \eqref{endo1} and \eqref{endo2}. The essential image of the global section functor $\Gamma:\Sch^\dagger_\Fun\to\Gamma\OBlpr$ is the (co)reflective subcategory $\Gamma\OBlpr^\dagger$ of $\Gamma\OBlpr$ that consists of all ordered blueprints $B$ with $B^\dagger\simeq B$. 

\begin{ex}\label{ex: mon bullet padd do not commute with tensor products}
 The following example shows that none of the endofunctors $(-)^\bullet$, $(-)^\mon$ and $(-)^\padd$ on $\OBlpr$ commutes with tensor products. Namely let $B=\bpgenquot{\Fun}{1+1+1\=1+1}$ and $\F_2=\bpgenquot{\Fun}{1+1\=0}$. Then
 \[
  B\otimes_{\Fun}\F_2\ = \ \bpgenquot{\Fun}{1+1+1\=1+1\=0}
 \]
 is the trivial ordered blueprint $\{0\}$ since $1\=1+0\=1+1+1\=0$. Thus $(B\otimes_{\Fun}\F_2)^\dagger=\{0\}$ for all $\dagger\in\{\mon,\bullet,\padd\}$. But we have
 \[
 \begin{array}{rll}
  B^\bullet\otimes_{\F_1^\bullet}\F_2^\bullet &= \quad \Fun\otimes_\Fun\Fun                        &= \quad \Fun, \\[10pt]
  B^\mon\otimes_{\F_1^\mon}\F_2^\mon          &= \quad \Fun\otimes_\Fun\bpgenquot{\Fun}{0\leq 1+1} &= \quad \bpgenquot{\Fun}{0\leq 1+1}, \\[10pt]
  B^\padd\otimes_{\F_1^\padd}\F_2^\padd       &= \quad \Fun\otimes_\Fun\F_2                        &= \quad \F_2, 
 \end{array}
 \]
 which differs from the trivial ordered blueprint in each case.
\end{ex}

%%%%%%%%%%%%%%%%%%%%%%%%%%%%%%%%%%%%%%%%%%%%%%%%%%%%%%%%%%%%%%%%%%%%%%%%%%%%%%%%%%%%%%%%%%%%%%%%%%%%%%%%%%%%%%%%%%%%%%%%%%%%%%%%%%%%%%%%%%%%%%%%%%%%%%%%%%%%%%%%%%%%%%%%%%%

\subsection{Semiring schemes}
\label{subsection: semiring schemes}

In this section, we briefly introduce semiring schemes. For more details, cf.\ \cite{Lorscheid17} and \cite{Lorscheid18}. 

Let $R$ be a semiring. An \emph{ideal of $R$} is a nonempty subset $I$ of $R$ such that $IR=I$ and $I+I=I$. A \emph{prime ideal of $R$} is an ideal $\fp$ such that $R-\fp$ is a multiplicative subset of $R$. A \emph{maximal ideal of $R$} is an ideal $\fm\neq R$ such that for every other ideal $I\neq R$ with $\fm\subset I$, we have $\fm=I$. Note that every maximal ideal is prime; cf.\ \cite[Lemma 2.6.5]{Lorscheid18}. A semiring is \emph{local} if $\fm=R-R^\times$ is an ideal, which is necessarily the unique maximal ideal of $R$. 

Given a multiplicative subset $S$ of $R$, we define the localization $S^{-1}R=S\times R/\sim$ where $(s,a)\sim(s',a')$ if and only if there is a $t\in S$ such that $tsa'=ts'a$. We denote the class of $(s,a)$ in $S^{-1}R$ by $\frac as$. The set $S^{-1}R$ becomes a semiring with respect to the usual rules
\[\textstyle
 \frac as + \frac bt \ = \ \frac{ta+sb}{st} \quad \text{ and } \quad \frac as \cdot \frac bt \ = \ \frac{ab}{st}.
\]
For an element $h\in R$, we denote by $R[h^{-1}]$ the localization of $R$ at $S=\{h^i\}_{i\in\N}$, and for a prime ideal $\fp$ of $R$, we denote by $R_\fp$ the localization of $R$ at $S=R-\fp$.

A \emph{locally semiringed space} is a topological space $X$ together with a sheaf $\cO_X$ in $\SRings$ whose stalks are local semirings. A local morphism of locally semiringed spaces is a continuous map $\varphi:X\to Y$ together with a morphism $\varphi^\#:\varphi^\ast\cO_Y\to\cO_X$ of sheaves such that the induced maps $\varphi_x:\cO_{Y,\varphi(x)}\to\cO_{Xx}$ between stalks map nonunits to nonunits.

The \emph{spectrum of a semiring $R$} is the following semiringed space. Its underlying topological space $X=\Spec R$ is the set of all prime ideals of $R$ together with the topology that is generated by all \emph{principal opens}
\[
 U_h \ = \ \big\{ \, \fp\in X \, \big| \, h\notin \fp\, \big\}
\]
where $h$ varies through the elements of $R$. Its structure sheaf is the unique sheaf $\cO_X$ on $X$ such that $\cO_X(U_h)=R[h^{-1}]$.

The uniqueness follows from the fact that the principal opens form a basis of the topology of $X$ since $U_g\cap U_h=U_{gh}$. The existence of $\cO_X$ is a nontrivial fact, cf.\ \cite[section 9]{Lorscheid17}. The stalks of $\Spec R$ are local since $\cO_{X,\fp}=B_\fp$ is a local semiring for every prime ideal $\fp$ of $R$.

An \emph{affine semiring scheme} is a locally semiringed space that is isomorphic to the spectrum of a semiring. A \emph{semiring scheme} is a locally semiringed space that can be covered by affine open subschemes. A morphism of semiring schemes is a local morphisms of locally semiringed spaces. We denote the category of semiring schemes by $\Sch_\N^+$. Note that a morphism $f:R\to R'$ induces a morphism $f^\ast:\Spec R'\to \Spec R$ of semiring schemes in the usual way.

Since a (prime) ideal of a semiring $R$ that is a ring is the same thing as a (prime) ideal in the sense of ring theory, the spectrum of $R$ as a semiring is the same as the spectrum of $R$ as a ring. Consequently, the category $\Sch_\Z^+$ of schemes embeds naturally as a full subcategory into $\Sch_\N^+$.

Given a semiring scheme $X$, we write $\Gamma X=\cO_X(X)$ for the semiring of global sections. This construction is functorial via pulling back global sections. The following facts extend from usual schemes to all semiring schemes, cf.\ \cite[section 9]{Lorscheid17}.

\begin{thm}
 Given a semiring $R$ and a semiring scheme $X$, we have a natural bijection $\Hom(X,\Spec R)\to\Hom(R,\Gamma X)$, i.e.\ $\Spec$ and $\Gamma$ are adjoint functors. In particular, every morphism between affine semiring schemes comes from a morphism between the semirings of global sections.
\end{thm}

\subsection{Base extension to semiring schemes}
\label{subsection: base extension to semiring schemes}

The reflective endofunctor $(-)^+:\OBlpr\to\OBlpr$ also extends to a reflective endofunctor $\OBSch\to\OBSch$ for the same reason as the other functors in Theorem \ref{thm: endofunctors of ordered blue schemes} \eqref{endo1}, cf.\ Remark \ref{rem: base extension to semirings as ordered blue scheme}. But we will denote by $(-)^+$ the functor $\OBSch\to\Sch_\N^+$ that takes values in the category of semiring schemes. 

In brevity, it is the composition of the functor $(-)^\core:\OBSch\to\OBSch^\alg$ with $(-)^+:\OBSch^\alg\to\Sch_\N^+$ where the latter functor is the same as considered in \cite[Thm.\ 10.1]{Lorscheid17}. In the following, we shall spell out the definition of $(-)^+:\OBSch\to\Sch_\N^+$, but we will omit the proof that the functor is well-defined. The interested reader will find details in \cite{Lorscheid17}.

For an affine ordered blue scheme $X=\Spec B$, we define $X^+=\Spec B^+$ as an object in $\Sch_\N^+$. For an arbitrary ordered blue scheme $X$, we choose an affine presentation $\cU$ of $X$ and define $X^+$ as the colimit of $\cU^+$ in $\Sch_\N^+$, which exists and is independent from the chosen affine presentation, up to a canonical isomorphism.

Given a morphism $\varphi:X\to Y$ of ordered blue schemes, we can cover $X$ and $Y$ with affine open subsets $U_i$ and $V_i$, respectively, such that $\varphi$ restricts to morphisms $\varphi_i:U_i\to V_i$. We define $\varphi^+:X^+\to Y^+$ as the morphism whose restrictions $U_i^+\to V_i^+$ are equal to the induced morphisms $(\Gamma\varphi_i^+)^\ast$. Again, this morphism $\varphi^+$ exists uniquely and is independent from the choice of coverings $\{U_i\}$ and $\{V_i\}$.

This yields the desired functor
\[
 (-)^+: \ \OBSch \ \longrightarrow \ \Sch_\N^+.
\]

\begin{rem}\label{rem: base extension to semirings as ordered blue scheme}
 In a few situation, we will use the variant of $(-)^+:\OBSch\to\Sch_\N^+$ that maps into $\OBSch$. We denote this variant by $(-)^{+\blue}:\OBSch\to\OBSch$, cf.\ sections \ref{subsection: Kajiwara-Payne - the associated blue scheme}, \ref{section: Foster-Ranganathan tropicalization} and \ref{subsection: blue scheme associated with a fine and saturated log scheme}.
\end{rem}

%%%%%%%%%%%%%%%%%%%%%%%%%%%%%%%%%%%%%%%%%%%%%%%%%%%%%%%%%%%%%%%%%%%%%%%%%%%%%%%%%%%%%%%%%%%%%%%%%%%%%%%%%%%%%%%%%%%%%%%%%%%%%%%%%%%%%%%%%%%%%%%%%%%%%%%%%%%%%%%%%%%%%%%%%%%
%%%%%%%%%%%%%%%%%%%%%%%%%%%%%%%%%%%%%%%%%%%%%%%%%%%%%%%%%%%%%%%%%%%%%%%%%%%%%%%%%%%%%%%%%%%%%%%%%%%%%%%%%%%%%%%%%%%%%%%%%%%%%%%%%%%%%%%%%%%%%%%%%%%%%%%%%%%%%%%%%%%%%%%%%%%

\section{Rational points}
\label{section: rational points}

In this section, we endow the set of $T$-rational points $X(T)=\Hom_k(\Spec T,X)$ with a topology for every ordered blue $k$-scheme $X$, coming from a topology for the ordered blue $k$-algebra $T$. We follow the categorical approach of \cite{Lorscheid-Salgado14} where the case of topologies for rational point sets of usual schemes has been considered, and where it is shown that this definition generalizes previous concepts as the strong topology for varieties over topological fields and the adelic topology for varieties over global fields.

%%%%%%%%%%%%%%%%%%%%%%%%%%%%%%%%%%%%%%%%%%%%%%%%%%%%%%%%%%%%%%%%%%%%%%%%%%%%%%%%%%%%%%%%%%%%%%%%%%%%%%%%%%%%%%%%%%%%%%%%%%%%%%%%%%%%%%%%%%%%%%%%%%%%%%%%%%%%%%%%%%%%%%%%%%%

\subsection{The affine topology}
\label{subsection: the affine topology}

Let $k$ be an ordered blueprint. An \emph{ordered blue $k$-algebra} is an ordered blueprint $B$ together with a morphism $k\to B$, which is called the \emph{structure morphism of $B$}. A morphism of ordered blue $k$-algebras $B$ and $C$ is a morphism $f:B\to C$ of ordered blueprints that commutes with the structure morphisms $k\to B$ and $k\to C$. We denote the category of ordered blue $k$-algebras by $\Alg_k^\ob$.

Let $k$ be an ordered blueprint and $T$ an {ordered blue $k$-algebra} that is equipped with a topology, which is not assumed to satisfy any compatibility with the structure of $T$ as an ordered blueprint. 

Let $B$ be an ordered blue $k$-algebra and $h_B(T)=\Hom_k(B,T)$ the set of $k$-linear morphisms. The \emph{affine topology for $h_B(T)$} is the compact-open topology where we consider $B$ as a discrete ordered blueprint. Since the compact subsets of $B$ are precisely the finite subsets, the compact-open topology on $h_B(T)$ is generated by open subsets of the form $U_{a,V}=\{f:B\to T|f(a)\in V\}$ where $a\in B$ and $V\subset T$ is an open subset. In other words, the affine topology on $h_B(T)$ is the coarsest topology such that the maps
\[
 \begin{array}{cccc}
  \ev_a: &  \Hom(B,T) & \longrightarrow & T \\
         & (f:B\to T) & \longmapsto     & f(a)
 \end{array}
\]
are continuous for all $a\in B$.

%%%%%%%%%%%%%%%%%%%%%%%%%%%%%%%%%%%%%%%%%%%%%%%%%%%%%%%%%%%%%%%%%%%%%%%%%%%%%%%%%%%%%%%%%%%%%%%%%%%%%%%%%%%%%%%%%%%%%%%%%%%%%%%%%%%%%%%%%%%%%%%%%%%%%%%%%%%%%%%%%%%%%%%%%%%

\subsection{The fine topology}
\label{subsection: the fine topology}

For an arbitrary ordered blue scheme $X$, we endow $X(T)$ with the \emph{fine topology}, which is the finest topology such that for all morphisms $\alpha:U\to X$ from an affine ordered blue $k$-scheme $U$ to $X$, the induced map $\alpha_T:U(T)\to X(T)$ is continuous with respect to the affine topology on $U(T)=h_{\Gamma U}(T)$. 

The following properties can be shown by similar arguments as used to prove the analogous statements for rings and schemes in \cite{Lorscheid-Salgado14}. Since the adaptation of these arguments require some slight modifications, we provide a proof.

\begin{lemma}\label{lemma: fine topology is functorial}
 Let $k$ be an ordered blueprint and $T$ an ordered blue $k$-algebra with topology.
 \begin{enumerate}
  \item\label{functorial1} Let $B$ be an ordered blue $k$-algebra. Then the affine topology for $h_B(T)$ is functorial in both $B$ and $T$.
  \item\label{functorial2} Let $X=\Spec B$ be an affine ordered blue $k$-scheme. Then the fine topology and the affine topology for $X(T)=h_B(T)$ coincide.
  \item\label{functorial3} Let $X$ be an ordered blue $k$-scheme. Then the fine topology for $X(T)$ is functorial in both $X$ and $T$.
 \end{enumerate}
\end{lemma}

\begin{proof}
 We begin with \eqref{functorial1}. Since the topology of $h_B(T)$ is generated by open subsets of the form 
 \[
  U_{V,a} \quad = \quad \{ \, f:A\to R \, | \, f(a)\in V\, \}
 \]
 with $a\in B$ and $V\subset T$ open, it suffices to verify that the inverse images of such open subsets are open to verify the continuity of the maps in question. 

 Let $g:C\to B$ be a homomorphism of ordered blue $k$-algebras and $g^\ast:h_B(T)\to h_C(T)$ the pullback of morphisms. It is immediate that $(g^\ast)^{-1}(U_{V,b})=U_{V,g(b)}$, which shows that $g^\ast$ is continuous.

 Let $f:T\to S$ be a continuous homomorphism of ordered blue $k$-algebras with topologies and $f_\ast:h_B(T)\to h_B(S)$ the pushforward of morphisms. It is immediate that $f_\ast^{-1}(U_{V,a})=U_{f^{-1}(V),a}$, which shows that $f_\ast$ is continuous. This verifies \eqref{functorial1}.
 
 We continue with \eqref{functorial2}. The identity morphism $\id:X\to X$ yields a continuous map $\id_T:X(T)\to X(T)$ with respect to the affine topology for the domain and the fine topology for the image. This shows that the affine topology is finer than the fine topology.

 Conversely, note that every $k$-linear morphism $\alpha:U\to X$ factors through the identity $\id:X\to X$. We have already proven that the map $U(T)\to X(T)$ is continuous with respect to the affine topology for both domain and image. This shows that the fine topology is at least as fine as the affine topology. We conclude that both topologies coincide. This verifies \eqref{functorial2}.

 We continue with \eqref{functorial3}. Let $\varphi:X\to Y$ be a morphism of ordered blue $k$-schemes and $\varphi_T:X(T)\to Y(T)$ the induced map. Let $W\subset Y(T)$ be open. We have to show that $Z=\varphi_T^{-1}(W)$ is open in $X(T)$, which is the case if $\alpha_T^{-1}(Z)$ is open in $U(T)$ for every morphism $\alpha:U\to X$ from an affine ordered blue $k$-scheme $U$ to $X$.

 Since $\varphi\circ\alpha:U\to Y$ is a $k$-morphism from the affine ordered blue scheme $U$ to $Y$, the inverse image $\alpha_T^{-1}(Z)=(\varphi\circ\alpha)_T^{-1}(W)$ of $W$ in $U(T)$ is indeed open. This shows that the fine topology of $X(T)$ is functorial in $X$.

 Let $f:T\to S$ be a continuous homomorphism of ordered blue $k$-algebras with topologies and $f_X:X(T)\to X(S)$ the induced map. Let $W\subset X(S)$ be open. We have to show that $Z=f_X^{-1}(W)$ is open in $X(T)$, which is the case if $\alpha_T^{-1}(Z)$ is open in $U(T)$ for every morphism $\alpha:U\to X$ from an affine $k$-scheme $U$ to $X$.

 Since the affine topology is functorial, the homomorphism $f:T\to S$ induces a continuous map $f_U:U(T)\to U(S)$. Since $\alpha_S^{-1}(W)$ is open in $U(S)$, the inverse image $\alpha_T^{-1}(Z)=f_U^{-1}(\alpha_S^{-1}(W))$ is open in $U(T)$. This shows that the fine topology of $X(T)$ is functorial in $T$ and concludes the proof of the lemma.
\end{proof}

The fact that every ordered blueprint is local implies the following explicit description of the fine topology.

\begin{prop}\label{prop: the fine topology as colimit over an affine presentation}
 Let $X$ be an ordered blue scheme, $\cU$ an affine presentation of $X$ and $T$ an ordered blueprint with topology. Then the canonical map $\colim\cU(T)\to X(T)$ is a homeomorphism.
\end{prop}
 
\begin{proof}
 To begin with, we show that the canonical map $\Phi:\colim\cU(T)\to X(T)$ is surjective. Since $T$ has a unique maximal ideal, every morphism $\alpha:\Spec T\to X$ factors through $\iota_U:U(T)\to X(T)$ for some $U$ in $\cU$. Thus $\Phi$ is surjective.
 
 We continue with the injectivity of $\Phi$. Two morphisms $\beta_1:\Spec T\to U_1$ and $\beta_2:\Spec T\to U_2$ with $U_1$ and $U_2$ in $\cU$ that define the same point in $X(T)$ yield a commutative diagram
 \[
  \xymatrix@C=6pc@R=0pc{                                            &  U_1 \ar[dr]^{\iota_1} \\
                        \Spec T \ar[ur]^{\beta_1} \ar[dr]_{\beta_2} &                        & X \\
                                                                    &  U_2 \ar[ur]_{\iota_2} }
 \]
 Thus the unique closed point $z$ of $\Spec T$ is mapped to a point $x$ of $X$ that lies in the intersection of $\iota_1(U_1)$ and $\iota_2(U_2)$. Thus there is a $V$ in $\cU$ together with inclusions $\iota_{V,1}:V\to U_1$ and $\iota_{V,2}:V\to U_2$ such that $\iota_V(V)=\iota_{V,1}(V)=\iota_{V,2}(V)$ contains $x$. Since $z$ is the unique closed point of $\Spec T$, the whole image of $\Spec T$ in $X$ is contained in $\iota_V(V)$. This means that $\beta_i:\Spec T\to U_i$ factors into a (uniquely determined) morphism $\beta_V:\Spec T\to V$, followed by $\iota_{V,i}$, for $i=1,2$. This shows that $\beta_1$ and $\beta_2$ induce the same morphism $\Spec T\to\colim\cU$, which shows that $\Phi$ is injective.
 
 As the colimit of continuous maps $U(T)\to X(T)$, $\Phi$ is continuous. We are left with showing that $\Phi$ is open. Consider an open subset $W$ of $\colim\cU(T)$ and its image $W'=\Phi(W)$ in $X(T)$. Let $\alpha:Z\to X$ be a morphism from an affine ordered blue scheme $Z$ into $X$. Since $Z$ has a unique closed point, $\alpha$ factors into a morphism $\beta:Z\to U$, followed by the canonical map $\iota_U:U\to \colim\cU=X$ for some $U$ in $\cU$. Thus $\alpha_T^{-1}(W')=\beta_T^{-1}\big(\iota_{U,T}^{-1}(W')\big)$. Since $W$ is open in $\colim\cU(T)$, $\iota_{U,T}^{-1}(W')$ is open in $U(T)$. By Lemma \ref{lemma: fine topology is functorial}, $\beta_T^{-1}\big(\iota_{U,T}^{-1}(W')\big)$ is open in $Z(T)$. Since $\alpha:Z\to X$ was arbitrary, this shows that $W'$ is open in $X(T)$. This concludes the proof.
\end{proof}

Recall from section \ref{subsection: affine open coverings} that $\norm X$ denotes the set of closed points of $X$, and that every $x\in\norm X$ is contained in a unique affine open subscheme $U_x$ of $X$. Further recall that we say that $X$ has enough closed points if $X$ is covered by the $U_x$ for $x\in\norm X$.

\begin{cor}\label{cor: the fine topology is the covering if X has enough closed points}
 Let $X$ be an ordered blue scheme with enough closed points and $T$ an ordered blueprint with topology. Then the fine topology for $X(T)$ is the finest topology such that for every $x\in\norm X$, the map $U_x(T)\to X(T)$ is continuous where $U_x(T)=h_{\Gamma U_x}(T)$ is equipped with the affine topology.
\end{cor}

\begin{proof}
 By assumption, $X$ is covered by the $U_x$ with $x\in\norm X$. We can extend this to an affine presentation $\cU$ of $X$ if we include all affine opens $W$ in the pairwise intersections $U_x\cap U_y$, together with the inclusion maps $W\to U_x$. By Proposition \ref{prop: the fine topology as colimit over an affine presentation}, the canonical map $\colim \cU(T)\to X(T)$ is a homeomorphism. Since the collection $\{U_x|x\in\norm X\}$ is cofinal in the diagram $\cU$, the final topology on $X(T)=\colim\cU(T)$ is equal to the final topology on $X(T)$ with respect to the maps $U_x(T)\to X(T)$ for $x\in\norm X$, which proves our claim.
\end{proof}

%%%%%%%%%%%%%%%%%%%%%%%%%%%%%%%%%%%%%%%%%%%%%%%%%%%%%%%%%%%%%%%%%%%%%%%%%%%%%%%%%%%%%%%%%%%%%%%%%%%%%%%%%%%%%%%%%%%%%%%%%%%%%%%%%%%%%%%%%%%%%%%%%%%%%%%%%%%%%%%%%%%%%%%%%%%

\subsection{Properties for topological ordered blueprints}
\label{subsection: properties for topological ordered blueprints}

Up to this point, we did not assume any compatibility of the topology of $T$ with the algebraic structure of $T$. Such compatibilities are reflected by additional properties of the fine topology for $X(T)$, as explained in the following.

A \emph{topological ordered blueprint} is an ordered blueprint $T$ with a topology such that the multiplication map $T\times T\to T$ is continuous. A \emph{topological semiring} is a semiring $T$ with a topology such that multiplication and addition define continuous maps $T\times T\to T$. A topological ordered blueprint $T$ is \emph{with open unit group} if the multiplicative group $T^\times$ of invertible elements in $T$ forms an open subset of $T$ and if the multiplicative inversion $T^\times\to T^\times$ is a continuous map. 

\begin{thm}\label{thm: properties of the fine topology}
 Let $k$ be an ordered blueprint and $T$ a topological ordered blue $k$-algebra with open unit group. Then $T$ satisfies the following properties.
 \begin{enumerate}[label={(F\arabic*)}]
  \item\label{F1} The functor $X\mapsto X(T)$ commutes with limits.
  \item\label{F2} The canonical bijection $\A^1_k(T)\to T$ is a homeomorphism.
  \item\label{F3} An open immersion $Y\to X$ of ordered blue $k$-schemes yields an open embedding $Y(T)\to X(T)$ of topological spaces.
  \item\label{F4} A covering $X=\bigcup U_i$ of an ordered blue $k$-scheme $X$ by open subschemes yields a covering $X(T)=\bigcup U_i(T)$ by open subspaces.
  \item\label{F5} A closed immersion $Y\to X$ of ordered blue $k$-schemes yields an embedding $Y(T)\to X(T)$ of topological spaces.
 \end{enumerate}
 If in addition, $T$ is a topological semiring, then $T$ satisfies the following property.
 \begin{enumerate}[label={(F\arabic*)}]\addtocounter{enumi}{5}
  \item\label{F6} The canonical bijection $X^+(T)\to X(T)$ is a homeomorphism for every ordered blue $k$-scheme $X$.
 \end{enumerate}
 If in addition, the topological semiring $T$ is Hausdorff, then $T$ satisfies the following stronger version of \ref{F5}.
 \begin{enumerate}[label={(F\arabic*)}]\addtocounter{enumi}{6}
  \item\label{F7} A closed immersion $Y\to X$ of ordered blue $k$-schemes yields a closed embedding $Y(T)\to X(T)$ of topological spaces.
 \end{enumerate}
\end{thm}

Before we turn to the proof of the theorem, we point out that the theorem applies to the main example of interest for tropicalizations and analytifications.

\begin{ex}[Tropical numbers]\label{ex: real topology of the tropical numbers}
 The tropical numbers $\T$ inherit the real topology from the identification $\T=\R_{\geq0}$. With this topology, $\T$ is a topological Hausdorff semiring with open unit group. Thus $\T$ satisfies \ref{F1}--\ref{F7}.
\end{ex}

\begin{proof}
 The strategy of the proof is as follows. The proof of the properties \ref{F1}--\ref{F5} consists of two parts: a proof in the affine case and a reduction of the general case to the affine case. In so far, we will first establish affine variants \ref{A1}--\ref{A5} of \ref{F1}--\ref{F5} before we treat \ref{F1}--\ref{F5} in full generality. Finally, we will verify \ref{F6} and \ref{F7}.

 We begin with the proof of the affine variant of \ref{F5}.
 \begin{enumerate}[label={(A\arabic*)}]\addtocounter{enumi}{4}
  \item\label{A5} A surjection $B\to C$ of ordered blue $k$-algebras yields an embedding $h_C(T)\to h_B(T)$ of topological spaces.
 \end{enumerate}
 A basic open of $h_C(T)=\Hom_k(C,T)$ is of the form $U_{b,V}=\{g:C\to T|g(b)\in V\}$ with $b\in C$ and $V\subset T$ open. Since $B\to C$ is surjective, $b=f(a)$ for some $a\in B$. Note that $h_C(T)\to h_B(T)$ is an inclusion, thus 
 \[
  U_{b,V} \ = \ U_{f(a),V} \ = \ \{\, g:B\to T \, | \, g(a)\in V\text{ and }g\text{ factors through }C \, \} \ = \ U_{a,V} \ \cap \ h_C(T).
 \]
 This shows that $h_C(T)\to h_B(T)$ is a topological embedding and concludes the proof of \ref{A5}.
 
 We continue with the proof of the affine variant of \ref{F1}.
 \begin{enumerate}[label={(A\arabic*)}]\addtocounter{enumi}{0}
  \item\label{A1} The functor $B\mapsto h_B(T)$ commutes with colimits.
 \end{enumerate}
 Let $\cD$ be a diagram of ordered blue $k$-algebras with colimit $B$. We have to show that the canonical bijection $\Psi:h_B(T)\to h_{\colim\cD}(T)$ is a homeomorphism. Note that the Yoneda embedding commutes with colimits, i.e.\ $h_{\colim\cD}=\lim h_\cD$ where $h_\cD$ is the diagram of functors on $\Alg_k^\ob$ defined by $\cD$.
 
 The continuity of $\Psi$ is easily verified: the canonical projections $\pi_i:h_B\to h_{C_i}$ to objects $C_i$ of $\cD$ induce continuous maps $\pi_{i,T}:h_B(T)\to h_{C_i}(T)$ that commute with all maps $\varphi_T:h_{C_i}(T)\to h_{C_j}(T)$ coming from morphisms $\varphi:C_j\to C_i$ in $\cD$. These maps induce the canonical bijection $\Psi:h_B(T)\to\lim h_\cD(T)$, which, consequently, is a continuous map.
 
 We proceed to show that $\Psi$ is open. Since every colimit can be expressed in terms of coequalizers and coproducts, it is enough to prove the openness of $\Psi$ for these particular types of limits. 
 
 The coequalizer $E=\coeq(f,g)$ of two morphisms $f,g:C\to D$ comes with a surjection $D\to E$. By \ref{A5}, $h_E(T)\to h_D(T)$ is a topological embedding. Since the topological equalizer $\eq(f_T,g_T)$ of $f_T,g_T:h_D(T)\to h_C(T)$ has the subspace topology of $h_D(T)$, the canonical bijection $h_E(T)\to\eq(f_T,g_T)$ is a homeomorphism.
 
 The coproduct of a family $(B_i)_{i\in I}$ of ordered blue $k$-algebras is the (possibly infinite) tensor product $B=\bigotimes B_i$ over $k$, whose elements are tensors $a=a_{i_1}\otimes\dotsb\otimes a_{i_n}$ with coordinates in a finite set of factors $B_{i_1},\dotsc,B_{i_n}$. Thus a basic open of $h_B(T)$ is of the form $U_{a,V}$ where $a\in\bigotimes B_i$ and $V\subset T$ is an open. We have
 \[\textstyle
  \Psi(U_{a,V}) \ = \ \{ \, (f_i:B_i\to T) \, | \, \prod_{k=1}^n f_{i_k}(a_{i_k})\in V \, \} \ = \ \{ \, (f_i:B_i\to T) \, | \, \bigl(f_{i_1}(a_{i_1}),\dotsc,f_{i_n}(a_{i_n})\bigr)\in \mu_n^{-1}(V) \, \} 
 \]
 where $\mu_n:T^n\to T$ is the $n$-fold multiplication. Since the multiplication of $T$ is continuous, $\mu_n^{-1}(V)$ is open in $T^n$ and can be covered by basic open subsets $V_{k,1}\times \dotsb \times V_{k,n}$ where $k$ varies in some index set $I$. Thus we have
 \[
  \Psi(U_{a,V}) \ = \ \bigcup_{k\in I} \{(f_i:B_i\to T) \, | \, f_{i_l}(a_{i_l})\in V_{k,l}\text{ for }l=1,\dotsc,n \, \} \ = \ \bigcup_{k\in I} \ \bigcap_{l=1}^n \pi_{i_l}^{-1}\bigl(U_{f_{i_l}(a_{i_l}),V_{k,l}}\bigr)
 \]
 where $\pi_{i_l}:\prod X_i(T) \to X_{i_l}$ is the canonical projection. This shows that $\Psi(U_{a,V})$ is open in $\prod X_i(T)$, as desired. This finishes the proof of \ref{A1}.
 
 We continue with the proof of the affine variant of \ref{F2}.
 \begin{enumerate}[label={(A\arabic*)}]\addtocounter{enumi}{1}
  \item\label{A2} The canonical bijection $h_{k[t]}(T)\to T$ is a homeomorphism.
 \end{enumerate}
 The canonical bijection $\Phi:h_{k[t]}(T)\to T$ sends a homomorphism $f:k[t]\to T$ to $f(t)$. Given an open subset $V\subset T$, we have $\Phi^{-1}(V)=\{f\in\A^1_k(T)|f(t)\in V\}$, which is the basic open subset $U_{t,V}$ of $\A^1_k(T)$. Thus $\Phi$ is continuous. 

 A basic open of $h_{k[t]}(T)=\Hom_k(k[t],T)$ is of the form $U_{a,V}$ with $a\in k[t]$ and $V\subset T$ open. Every element $a$ of $k[t]$ is of the form $ct^i$ for some $c\in k$ and some $i\geq0$. Since the multiplication of $T$ is continuous, the evaluation of the monomial $a=ct^i$ in elements of $T$ defines a continuous map $a:T\to T$. The inverse image $W=a^{-1}(V)$ is open in $T$, and thus
 \[
  U_{a,V} \ = \ \{\, f:k[t]\to T \, | \, f(t)^i\in a^{-1}(V) \, \} \ = \ U_{t,W}
 \]
 is mapped to the open subset $\Phi(U_{a,V})=W$ of $T$. This completes the proof of \ref{A2}.
 
 We continue with the proof of the affine variant of \ref{F3}.
 \begin{enumerate}[label={(A\arabic*)}]\addtocounter{enumi}{2}
  \item\label{A3} The localization $B\to B[h^{-1}]$ of an ordered blue $k$-algebra $B$ at an element $h\in B$ yields an open embedding $h_{B[h^{-1}]}(T)\to h_B(T)$ of topological spaces.
 \end{enumerate}
 As a first case, we consider the localization $k[t]\to k[t^{\pm1}]$ and the inclusion $\iota_T:h_{k[t^{\pm1}]}(T)\to h_{k[t]}(T)$. As sets, $h_{k[t^{\pm1}]}(T)=T^\times$, and by \ref{A2}, the canonical bijection $h_{k[t]}(T)\to T$ is a homeomorphism. Since $T$ is with open unit group, $T^\times$ is an open subset of $T$ and the image 
 \[
  \iota_T(U_{at^i,V}) \ = \ \{ \, f:k[t]\to T \, | \, f(at^i)\in V, f(t)\in T^\times\} \ = \ U_{at^i,V} \, \cap \, U_{t,T^\times}
 \]
 of a basic open $U_{at^i,V}\subset h_{k[t^{\pm1}]}(T)$ is open in $h_{k[t]}(T)$ for $at^i\in k[t]$ and $V\subset T$ open. In particular, we can assume that $V\subset T^\times$.
 
 If $at^{-i}\in k[t^{\pm1}]$ is in the complement of $k[t]$, i.e.\ if $-i<0$, then we first observe that $U_{at^{-i},V}=U_{t^{-i},a^{-1}(V)}$ where $a^{-1}(V)=\{b\in T|ab\in V\}$ is open in $T$ since the multiplication of $T$ is continuous. Thus we might assume that $a=1$. Since the inversion $i:T^\times\to T^\times$, sending $b$ to $b^{-1}$ is continuous, $i^{-1}(V)$ is open in $T$. Thus the image of $U_{t^{-i},V}=U_{t^i,i^{-1}(V)}$ is the basic open $U_{t^i,i^{-1}(V)}$ of $h_{k[t]}(T)$. This shows that $\iota_T$ is an open topological embedding.
 
 In the general case of a localization $B\to B[h^{-1}]$, we have $B[h^{-1}]=B\otimes_{k[t]}k[t^{\pm1}]$ with respect to the $k$-linear morphism $k[t]\to B$ that maps $t$ to $h$. By \ref{A1}, $h_{B[h^{-1}]}(T)$ is homeomorphic to $h_B(T)\times_{h_{k[t]}(T)}h_{k[t^{\pm1}]}(T)$, i.e.\ $\iota'_T:h_{B[h^{-1}]}(T)\to h_B(T)$ is the base change of the open topological embedding $\iota_T:h_{k[t^{\pm1}]}(T)\to h_{k[t]}(T)$ along $h_B(T)\to h_{k[t]}(T)$, and therefore $\iota'_T$ itself is an open topological embedding, as desired. This concludes the proof of \ref{A3}.
 
 We continue with the proof of the affine variant of \ref{F4}.
 \begin{enumerate}[label={(A\arabic*)}]\addtocounter{enumi}{3}
  \item\label{A4} A family of localizations $\{B\to B[h_i^{-1}]\}$ such that $\Spec B$ is covered by $\Spec B[h_i^{-1}]$ yields a covering $h_B(T)=\bigcup h_{B[h_i^{-1}]}(T)$ by open subspaces.
 \end{enumerate}
 Since $T$ is local, every morphism $\Spec T\to \Spec B$ factors through one of the principal open subschemes $\Spec B[h_i^{-1}]$. Thus we have $h_B(T)=\bigcup h_{B[h_i^{-1}]}(T)$ as sets. By \ref{A3}, the subsets $h_{B[h_i^{-1}]}(T)$ of $h_B(T)$ are indeed open subspaces. This concludes the proof of \ref{A4}.
  
 Note that thanks to Lemma \ref{lemma: fine topology is functorial}, \ref{A2} immediately implies \ref{F2}, and \ref{A1} implies \ref{F1} for limits of affine ordered blue $k$-schemes. 
 
 We continue with the proof of \ref{F1} for arbitrary limits. If $X=\lim \cD$ for a diagram $\cD$ of ordered blue schemes, then the canonical projections $\pi_Z:X\to Z$ to the objects $Z$ of $\cD$ induce continuous maps $\pi_{Z,T}:X(T)\to Z(T)$, which induce the canonical bijection $\Psi:X(T)\to \lim\bigl(\cD(T)\bigr)$, which is thus continuous. We are left with showing that $\Psi$ is open. Since every limit is an equalizer of products, we can restrict ourselves to the treatment of these particular types of limits. We will demonstrate the proof for fibre products, which includes equalizers and finite products, and leave the case of infinite products, whose proof is analogous, to the reader.
 
 Consider morphisms $X\to Z\leftarrow Y$ and the fibre product $X\times_ZY$. For an open subset $\widetilde W$ of $X\times_Z Y(T)$, we have to show that $W=\Psi(\widetilde W)$ is open in $X(T)\times_{Z(T)} Y(T)$. Since $X(T)\times_{Z(T)}Y(T)$ carries the subspace topology of $X(T)\times Y(T)$, a basis open is of the form $W_X\times W_Y$ with opens $W_X$ of $X(T)$ and $W_Y$ of $Y(T)$. By definition, $W_X$ is open in $X(T)$ if and only if $\alpha_T^{-1}(W_X)$ is open in $U(T)$ for all morphisms $\alpha:U\to X$ where $U$ is affine, and $W_Y$ is open in $Y(T)$ if and only if $\beta_T^{-1}(W_Y)$ is open in $V(T)$ for all morphisms $\beta:V\to Y$ where $V$ is affine. Thus $W_X\times W_Y$ is open in $X(T)\times_{Z(T)} Y(T)$ if and only if $(\alpha_T,\beta_T)^{-1}(W_X\times W_Y)$ is open in $U(T)\times V(T)$ for all morphisms $(\alpha,\beta):U\times V\to X\times_Z Y$ where $U\times V$ is affine. 

 This shows that $W$ is open in $X(T)\times_{Z(T)} Y(T)$ if and only if $(\alpha_T,\beta_T)^{-1}(W)=(\alpha,\beta)_T^{-1}(\widetilde W)$ is open in $U(T)\times V(T)=U\times V(T)$ for all $(\alpha,\beta):U\times V\to X\times_Z Y$ where $U\times V$ is affine. This follows from the openness of $\widetilde W$. Thus $\Psi$ is open. This completes the proof of \ref{F1}.

 We turn to the proof of \ref{F3}. Let $\iota:U\to X$ be an open immersion. By functoriality, the inclusion $\iota_{T}:U(T)\to X(T)$ is continuous. We are left with showing that the $\iota_{T}$ is open.
 
 Let $W\subset U(T)$ be an open subset and $W'=\iota_T(W)$ its image in $X(T)$. Then $W'$ is open if and only if $Z=\alpha_T^{-1}(W')$ is open in $h_B(T)$ for every ordered blueprint $B$ and every morphism $\alpha:V\to X$ where $V=\Spec B$. The base change of $\iota:U\to X$ along $\alpha:V\to X$ yields the open immersion $\psi:Z\to V$ where $Z=U\times_XV$. Since $W'=\iota_T(W)$ is in the image of $\iota_T$, the subset $Z$ of $V(T)$ is contained in the image of $\psi_T$. 
 
 We can cover $V'$ by principal opens $V'_i$ of $V$, i.e.\ the restriction $\psi_i:V'_i\to V$ of $\psi:V'\to V$ is induced by a morphism $f_i:B\to B[h_i^{-1}]$ of ordered blueprints for each $i$. By \ref{F1}, the canonical bijection $\Spec V_i'\to h_{B[h_i^{-1}]}(T)$ is a homeomorphism and by \ref{A3}, the induced map $h_{B[h_i^{-1}]}(T)\to h_B(T)$ is an open topological embedding. We conclude that $Z$ is open in $V(T)$ if $Z_i=\psi_{i,T}^{-1}(Z)$ is open in $h_{B[h_i^{-1}]}(T)$ for every $i$. But $Z_i=\beta_{i,T}^{-1}(W)$ where $\beta_i$ is the composition of the inclusion $V'_i\to V'$ with the canonical projection $V'=U\times_XV\to U$. By the definition of the fine topology for $U(T)$, $\beta_{i,T}^{-1}(W)$ is open in $V_i'(T)$. This completes the proof of \ref{F3}.
 
 We turn to the proof of \ref{F4}. Let $X=\bigcup U_i$ be a covering of open subschemes. Since $T$ is local, we have an equality $X(T)=\bigcup U_i(T)$ of sets. By \ref{F3}, this is indeed a covering by open subspaces. Thus \ref{F4}.
 
 We turn to the proof of \ref{F5}. Since the property of being a topological embedding is a local property, we can assume that $X$ is affine. Since closed immersions are affine morphisms by definition, the closed subscheme $Y$ of $X$ is also affine and the closed immersion $Y\to X$ is induced by a surjection $B\to C$ of ordered blue $k$-algebras. Thus \ref{F5} follows from its analogue \ref{A5} for the affine topology. This completes the proof of \ref{F5}.

 We turn to the proof of \ref{F6}, assuming that $T$ is a topological semiring. Since being a homeomorphism is a local property, we can restrict ourselves to the affine case, i.e.\ $X=\Spec B$ and $X^+=\Spec B^+$. By Lemma \ref{lemma: fine topology is functorial}, $X(T)=h_B(T)$ and $X^+(T)=h_{B^+}(T)$ as topological spaces, so it suffices to study the latter spaces. By functoriality, the bijection $\Xi:h_{B^+}(T)\to h_B(T)$ that is induced by the morphisms $B\to B^+$ is continuous. 
 
 For verifying that $\Xi$ is open, consider a basic open $U_{\sum a_i,V}$ where $\sum_{i=1}^n a_i$ is an element of $B^+$, with $a_i\in B$, and $V\subset T$ is open. Then
 \[\textstyle
  \Xi(U_{\sum a_i,V}) \ = \ \{ \, f:B\to T \, | \, \sum f(a_i)\in V \, \} \ = \ \{ \, f:B\to T \, | \, \bigl(f(a_1),\dotsc,f(a_n)\bigr)\in \alpha_n^{-1}(V) \, \} 
 \]
 where $\alpha_n:T^n\to T$ is the $n$-fold addition, which is a continuous map since $T$ is a topological semiring. Thus $\alpha_n^{-1}(V)$ is an open subset of $T^n$, which can be covered by basic opens of the form $V_{k,1}\times\dotsb\times V_{k,n}$ where $k$ ranges through some index set $I$. Thus
 \[
  \Xi(U_{\sum a_i,V}) \ = \ \bigcup_{k\in I} U_{a_1,V_{k,1}}\cap\dotsb\cap U_{a_n,V_{k,n}}
 \]
 is an open subset of $h_B(T)$. This completes the proof of \ref{F6}.

 We turn to the proof of \ref{F7}, assuming that $T$ is a topological semiring that is Hausdorff. Since being a closed topological embedding is a local property, we can assume that $X=\Spec B$ and $Y=\Spec C$ are affine, and that the closed immersion $Y\to X$ is induced by a surjection $f:B\to C$ of ordered blue $k$-algebras. By \ref{F5}, we know already that the inclusion $Y(T)\to X(T)$ is a topological embedding. By Lemma \ref{lemma: fine topology is functorial}, the bijections $X(T)\to h_B(T)$ and $Y(T)\to h_C(T)$ are homeomorphisms. Therefore we are left with showing that the image of $\iota_T:h_C(T)\to h_B(T)$ is a closed subset of $h_B(T)$
 
 Since $f:B\to C$ is surjective, we have $C=\bpquot B\cR$ for some subaddition $\cR$ on $B^\bullet$, which is a relation on $\N[B^\bullet]$. Thus
 \[
  \iota_T(h_C(T)) \ = \ \{\, f:B\to T \, | \, \sum f(a_i) =\sum f(b_j)\text{ for all }(\sum a_i,\sum b_j)\in\cR \, \}.
 \]
 The condition $\sum f(a_i) =\sum f(b_j)$ can be rewritten as $\alpha_n\bigl(f(a_1),\dotsc,f(a_n)\bigr)=\alpha_m\bigl(f(b_1),\dotsc,f(b_m)\bigr)$ where $\alpha_n:T^n\to T$ denotes the $n$-fold addition, which is continuous since $T$ is a topological semiring. This is, in turn, equivalent to $\bigl(f(a_1),\dotsc,f(a_n),f(b_1),\dotsc,f(b_m)\bigr)\in (\alpha_n,\alpha_m)^{-1}(\Delta)$ where $\Delta$ is the diagonal of $T\times T$. Since $T$ is Hausdorff, $\Delta$ is a closed subset of $T\times T$, and therefore $\Delta'=(\alpha_n,\alpha_m)^{-1}(\Delta)$ is a closed subset of $T^n\times T^m$. This means that
 \[
  \Delta' \ = \ \bigcap_{k\in I} \ V_{k,1}\times \dotsb\times V_{k,n+m}
 \]
 for certain closed subsets $V_{k,l}$ of $T$ where $k$ varies through some index set $I$ and $l=1,\dotsc,n+m$. We conclude that 
 \[
  \iota_T(h_C(T)) \ = \ \bigcap_{k\in I} \ U_{a_1,V_{k,1}} \cup\dotsb\cup U_{a_n,V_{k,n}} \cup U_{b_1,V_{k,n+1}} \cup\dotsb\cup U_{b_m,V_{k,n+m}} 
 \]
 This shows that $\iota_T(h_C(T))$ is a closed subset of $h_B(T)$, which completes the proof of the theorem.
\end{proof}

%%%%%%%%%%%%%%%%%%%%%%%%%%%%%%%%%%%%%%%%%%%%%%%%%%%%%%%%%%%%%%%%%%%%%%%%%%%%%%%%%%%%%%%%%%%%%%%%%%%%%%%%%%%%%%%%%%%%%%%%%%%%%%%%%%%%%%%%%%%%%%%%%%%%%%%%%%%%%%%%%%%%%%%%%%%
%%%%%%%%%%%%%%%%%%%%%%%%%%%%%%%%%%%%%%%%%%%%%%%%%%%%%%%%%%%%%%%%%%%%%%%%%%%%%%%%%%%%%%%%%%%%%%%%%%%%%%%%%%%%%%%%%%%%%%%%%%%%%%%%%%%%%%%%%%%%%%%%%%%%%%%%%%%%%%%%%%%%%%%%%%%

\part{Tropicalization}
\label{part: Tropicalilzation}

%%%%%%%%%%%%%%%%%%%%%%%%%%%%%%%%%%%%%%%%%%%%%%%%%%%%%%%%%%%%%%%%%%%%%%%%%%%%%%%%%%%%%%%%%%%%%%%%%%%%%%%%%%%%%%%%%%%%%%%%%%%%%%%%%%%%%%%%%%%%%%%%%%%%%%%%%%%%%%%%%%%%%%%%%%%
%%%%%%%%%%%%%%%%%%%%%%%%%%%%%%%%%%%%%%%%%%%%%%%%%%%%%%%%%%%%%%%%%%%%%%%%%%%%%%%%%%%%%%%%%%%%%%%%%%%%%%%%%%%%%%%%%%%%%%%%%%%%%%%%%%%%%%%%%%%%%%%%%%%%%%%%%%%%%%%%%%%%%%%%%%%

In the second part of the paper, we give our definition of a tropicalization $\Trop_v(X)$ as a solution to the moduli problem of extensions of a given valuation $v:k\to T$ to a given ordered blue $k$-scheme $X$ with values in ordered blueprints over $T$. Our central results on tropicalizations show their existence for totally positive and for idempotent $T$. In both cases, our constructions of tropicalizations pass through a base change along $v$. In the latter case, we gain an alternative description in terms of a generalization of the Giansiracusa bend relation to ordered blueprints.

In the subsequent sections, we show how our definition of tropicalization recovers and improves other concepts to tropicalization and analytification. In subsequent sections, we address Berkovich analytification, Kajiwara-Payne tropicalization, Foster-Ranganathan tropicalization, Giansiracusa tropicalization, Maclagan-Rinc\'on weights, Macpherson analytification, Thuillier analytification and Ulirsch tropicalization.

%%%%%%%%%%%%%%%%%%%%%%%%%%%%%%%%%%%%%%%%%%%%%%%%%%%%%%%%%%%%%%%%%%%%%%%%%%%%%%%%%%%%%%%%%%%%%%%%%%%%%%%%%%%%%%%%%%%%%%%%%%%%%%%%%%%%%%%%%%%%%%%%%%%%%%%%%%%%%%%%%%%%%%%%%%%
%%%%%%%%%%%%%%%%%%%%%%%%%%%%%%%%%%%%%%%%%%%%%%%%%%%%%%%%%%%%%%%%%%%%%%%%%%%%%%%%%%%%%%%%%%%%%%%%%%%%%%%%%%%%%%%%%%%%%%%%%%%%%%%%%%%%%%%%%%%%%%%%%%%%%%%%%%%%%%%%%%%%%%%%%%%

\section{Scheme theoretic tropicalization}
\label{section: Scheme theoretic tropicalization}

Let $k$ and $T$ be ordered blueprints, $v:k\to T$ a valuation and $X$ an ordered blue $k$-scheme. In this section, we introduce the functor $\Val_v(X,-)$ of valuations of $X$ over $v$. We define a tropicalization of $X$ along $v$ as an ordered blue $T$-scheme that represents $\Val_v(X,-)$ and construct tropicalizations if $T$ is totally positive or idempotent.

%%%%%%%%%%%%%%%%%%%%%%%%%%%%%%%%%%%%%%%%%%%%%%%%%%%%%%%%%%%%%%%%%%%%%%%%%%%%%%%%%%%%%%%%%%%%%%%%%%%%%%%%%%%%%%%%%%%%%%%%%%%%%%%%%%%%%%%%%%%%%%%%%%%%%%%%%%%%%%%%%%%%%%%%%%%

\subsection{Tropicalization as a moduli space}
\label{subsection: Definition of tropicalization}

Let $k$ be an ordered blueprint. We denote the category of ordered blue $k$-algebras together with $k$-linear morphisms by $\Alg_k^\ob$.

Let $v:k\to T$ be a valuation, i.e.\ a morphism $v^\bullet:k^\bullet\to T^\bullet$ between the underlying monoids of $k$ and $T$ together with a morphism $\tilde v:k^\mon\to T^\pos$ such that the diagram
\[
 \xymatrix@R=1pc@C=6pc{k^\bullet \ar[d]\ar[r]^{v^\bullet} & T^\bullet \ar[d] \\ k^\mon \ar[r]^{\tilde v} & T^\pos}
\]
commutes. Let $B$ be an ordered blue $k$-algebra and $S$ an ordered blue $T$-algebra. A \emph{valuation $w:B\to S$ over $v$} is a morphism $w^\bullet:B^\bullet\to S^\bullet$ together with a morphism $\tilde w:B^\mon\to S^\pos$ such that the diagram 
\[
  \xymatrix@R=0,5pc@C=2pc{ k^\bullet \ar[rrr]^(0.6){v^\bullet} \ar[dd] \ar[dr] &&        & T^\bullet \ar[dd]|\hole\ar[dr] \\
                                                        & B^\bullet \ar[rrr]^(0.4){w^\bullet}\ar[dd] &                  && S^\bullet \ar[dd] \\ 
                           k^\mon \ar[rrr]^(0.6){\tilde v}|(0.375)\hole\ar[dr]       &&        & T^\pos\ar[dr]  \\
                                                        & B^\mon \ar[rrr]^(0.4){\tilde w}  &                  && S^\pos                    }
\]
commutes. We denote by $\Val_v(B,S)$ the set of valuations $w:B\to S$ over $v$.

A morphism $f:S\to S'$ of ordered blue $T$-algebras sends a valuation $w:B\to S$ to the valuation $f\circ w:B\to S'$ where we define $(f\circ w)^\bullet=f^\bullet\circ w^\bullet$ and $(f\circ w)^\sim=f^\pos\circ \tilde w$. This shows that $\Val_v(B,S)$ is functorial in $S$, and we obtain the \emph{functor of valuations $\Val_v(B,-):\Alg_k^\ob\to\Alg_T^\ob$ on $B$ over $v$}.

\begin{df}
 Let $k$ and $T$ be ordered blueprints, $v:k\to T$ be a valuation and $B$ an ordered blue $k$-algebra. A \emph{tropicalization of $B$ along $v$} is an ordered blue $T$-algebra $\Trop_v(B)$ together with a valuation $w^\univ:B\to \Trop_v(B)$ that is universal, i.e.\ for every valuation $w:B\to S$ over $v$, there is a unique morphism $f:S\to \Trop_v(B)$ of ordered blue $T$-algebras such that $w=f\circ w^\univ$.
\end{df}

In other words, a tropicalization of $B$ along $v$ is an ordered blue $T$-algebra that represents the functor $\Val_v(B,-)$, i.e.\ the fine moduli space of all valuations of $B$ over $v$. By the Yoneda lemma, the tropicalization of $B$ along $v$ is unique up to unique isomorphism if it exists. 

This notion can be geometrized as follows. Let $X$ be an ordered blue scheme over $k$ with structure morphism $X\to \Spec k$ and $Y$ be an ordered blue $T$-scheme. A \emph{valuation of $\omega:X\to Y$ over $v$} is a morphism $\omega^\bullet:Y^\bullet\to X^\bullet$ together with a morphism $\tilde\omega:Y^\pos\to X^\mon$ such that the diagram
\[
  \xymatrix@R=0,5pc@C=2pc{                                 & \Spec T^\bullet \ar[rrr]^(0.4){(v^\bullet)^\ast}  &&                       & \Spec k^\bullet \\
                          Y^\bullet \ar[rrr]^(0.6){\omega^\bullet}\ar[ur] &                                              && X^\bullet \ar[ur] \\ 
                                                           & \Spec T^\pos \ar[rrr]^(0.4){\tilde v^\ast}|(0.596)\hole \ar[uu]|\hole &&   & \Spec k^\mon \ar[uu] \\
                          Y^\pos \ar[rrr]^(0.6){\tilde \omega} \ar[uu]\ar[ur] &                                          && X^\mon \ar[ur]\ar[uu]    }
\]
commutes. 

\begin{df}
 Let $k$ and $T$ be ordered blueprints, $v:k\to T$ be a valuation and $X$ an ordered blue $k$-scheme. A \emph{tropicalization of $X$ along $v$} is an ordered blue $T$-scheme $\Trop_v(X)$ together with a valuation $\omega^\univ:\Trop_v(X)\to X$ that is universal, i.e.\ for every valuation $\omega:Y\to X$ over $v$, there is a unique morphism $\varphi:Y\to \Trop_v(X)$ of ordered blue $T$-schemes such that $\omega=\omega^\univ\circ\varphi$.
\end{df}

Like in the affine situation, a tropicalization $\Trop_v(X)$ of $X$ along $v$ represents the \emph{functor of valuations $\Val_v(X,-):\Sch^\ob_T\to \Sets$ on $X$ over $v$}, which sends a $T$-scheme $Y$ to the set $\Val_v(X,Y)$ of valuations $\omega:X\to Y$ over $v$.

We call $T$ the \emph{base of $\Trop_v(X)$} or the \emph{tropicalization base} if the context is clear.

\begin{ex}\label{ex: tropicalization of monoid algebras}
 A simple yet most important case is that of free algebras or, slightly more general, that of monoid algebras. Let $k$ be a field, $A$ a monoid and $v:k\to T$ a valuation. Then the valuations of $w:k[A]\to S$ over $v$ correspond to monoid morphisms $A\to S^\bullet$, which in turn correspond to $T$-linear morphisms $T[A]\to S$. This shows that $\Spec T[A]$ is a tropicalization of $\Spec k[A]$ along $v$, and its universal valuation $\omega^\univ:\Spec T[A]\to\Spec k[A]$ is determined as the extension $k[A]\to T[A]$ of $v$ that is the identity on $A$.
\end{ex}

%%%%%%%%%%%%%%%%%%%%%%%%%%%%%%%%%%%%%%%%%%%%%%%%%%%%%%%%%%%%%%%%%%%%%%%%%%%%%%%%%%%%%%%%%%%%%%%%%%%%%%%%%%%%%%%%%%%%%%%%%%%%%%%%%%%%%%%%%%%%%%%%%%%%%%%%%%%%%%%%%%%%%%%%%%%

\subsection{Tropicalization for a totally positive tropicalization base}
\label{subsection: tropicalization for a totally positive base}

Let $v:k\to T$ be a valuation. For an ordered blue $k$-algebra $B$, we can make precise the idea that its tropicalization is the base change of $B$ along the valuation $v:k\to T$ in case that $T$ is totally positive.

We begin with the following preliminary observation. Let $S$ be an ordered blue $T$-algebra and $B$ an ordered blue $k$-algebra. Then the morphism $\tilde w: B^\mon\to S^\pos$ is uniquely determined by the valuation $w:B\to S$ since $B^\mon\to B$ is a bijection. This yields a map 
\[
 \Val_v(B,S) \quad \longrightarrow \quad \Hom_T(B_v, S)
\]
where we write $B_v$ for the tensor product $B^\mon\otimes_{k^\mon} T^\pos$.

\begin{lemma}\label{lemma: valuations as homs if T is idempotent or totally positive}
 The map $\Val_v(B,S) \to\Hom_T(B_v, S^\pos)$ is a bijection if $T$ is idempotent or totally positive.
\end{lemma}

\begin{proof}
 If $T$ is idempotent or totally positive, then every ordered blue $T$-algebra $S$ is idempotent or totally positive as well. By Lemma \ref{lemma: idempotent or totally positive implies strictly conic}, $S$ is strictly conic, and by Corollary \ref{cor: properties of totally positive blueprints}, the canonical morphism $S\to S^\pos$ is a bijection. This shows that a valuation $w$ is determined by $\tilde w$ in this case.
\end{proof}

\begin{thm}\label{thm: tropicalization for totally positive tropical base}
 Let $X$ be an ordered blue $k$-scheme and $v:k\to T$ a valuation into a totally positive blueprint $T$. Then $\Trop_v(X)=X^\mon\otimes_{k^\mon} T$ is a tropicalization of $B$ along $v$.
\end{thm}

\begin{proof}
 Since the base change $X^\mon\otimes_{k^\mon} T$ commutes with affine presentations, we can assume without loss of generality that $X=\Spec B$. If $T$ is totally positive, then every ordered blue $T$-algebra $S$ is totally positive as well, i.e.\ $S=S^\pos$. Therefore, the theorem follows at once from Lemma \ref{lemma: valuations as homs if T is idempotent or totally positive}.
\end{proof}

\begin{ex}\label{ex: tropicalization with a totally positive base}
 As a concrete example, let us a field $k$ with non-archimedean valuation $v:k\to\T$, which can also be seen as a valuation $v:k\to\T^\pos$ into the totally positive ordered blueprint $\T^\pos=\bpquot{\T}{\gen{0\leq1}}$ since the canonical map $\T\to\T^\pos$ is a bijection by Corollary \ref{cor: idempotent implies core=core after pos and B subset B-pos}.
 
 Let $X=\Spec B$ be a closed subscheme of $\A^n_k$ with $B=\bpquot{k[T_1,\dotsc,T_n]}{\cR}$. Then $X^\mon$ is the spectrum of $B^\mon=\bpquot{k^\mon[T_1,\dotsc,T_n]}{\cR^\mon}$ where
 \[\textstyle
  R^\mon \ = \ \big\langle a \leq \sum b_j \, \big| \, a\leq\sum b_j\text{ in }\cR\big\rangle.
 \]
 If we identify $k^\mon[T_1,\dotsc,T_n]\otimes_{k^\mon} \T^\pos$ with $\T^\pos[T_1,\dotsc,T_n]$ (see Example \ref{ex: tropicalization of monoid algebras}) and denote the universal valuation by $w:k[T_1,\dotsc,T_n]\to \T^\pos[T_1,\dotsc,T_n]$, then we get 
  \[\textstyle
  \Trop_v(X)=\Spec \big(\bpquot{\T^\pos[T_1,\dotsc,T_n]}{\gen{w(a)\leq\sum w(b_j) \mid a\leq \sum b_j\text{ in }\cR}}\big).
 \]
\end{ex}

\begin{rem}\label{rem: tropicalization as base change}
 Theorem \ref{thm: tropicalization for totally positive tropical base} makes precise the idea of the tropicalization as a base change along a valuation, as mentioned in the paragraph \emph{Analytification as a base change} of the introduction. As we explained in Remark \ref{rem: halos}, $(-)^\mon:\Halos\to\OBlpr^\mon$ is a fully faithful embedding of the category of halos, which are ordered semirings together with subadditive morphisms, into the category of monomial blueprints. In the following, we identify halos with their images in $\OBlpr^\mon$. 
 
 Let $k\to B$ and $v:k\to T$ be a morphism of halos and assume that $T$ is totally positive, e.g.\ $T=(\R_{\geq0}^\pos)^\mon$ or $T=(\T^\pos)^\mon$. Then the tensor product $B\otimes_k T$ exists in $\OBlpr^\mon$ and it is a tropicalization of $B$ along $v$.
 
 Note that the tensor product $B\otimes_k T$ is totally positive, but that it is typically not a halo. At the moment of writing, it is not clear to me if $\Halos$ contains all tensor products.
\end{rem}

%%%%%%%%%%%%%%%%%%%%%%%%%%%%%%%%%%%%%%%%%%%%%%%%%%%%%%%%%%%%%%%%%%%%%%%%%%%%%%%%%%%%%%%%%%%%%%%%%%%%%%%%%%%%%%%%%%%%%%%%%%%%%%%%%%%%%%%%%%%%%%%%%%%%%%%%%%%%%%%%%%%%%%%%%%%

\subsection{The bend functor}
\label{subsection: The bend functor}

In this section, we introduce the bend relation, which generalizes the corresponding concept from \cite{Giansiracusa13} to our setting; see section \ref{section: Giansiracusa tropicalization} for details on the connection to \cite{Giansiracusa13}.

\begin{df}
 Let $v:k\to T$ be a valuation and $B$ an ordered blue $k$-algebra. The \emph{bend of $B$ along $v$} is the ordered blue $T$-algebra
 \[
  \Bend_v(B) \quad = \quad \bpquot{B^\bullet\otimes_{k^\bullet}T}{\bend_v(B)}
 \]
whose subaddition is generated by the subaddition of $T$ and the \emph{bend relation}
 \[\textstyle
  \bend_v(B) \quad = \quad \left\langle \ a\otimes 1 + \sum b_j\otimes 1 \= \sum b_j\otimes 1 \ \left| \ a\leq\sum b_j\text{ in }B \ \right.\right\rangle.
 \]
\end{df}

A morphism $f:B\to C$ of ordered blue $k$-algebras defines a map 
\[
 \begin{array}{cccc}
  \Bend_v(f): & \Bend_v(B) & \longrightarrow & \Bend_v(C). \\
              & b\otimes t & \longmapsto     & f(b)\otimes t
 \end{array}
\]
This map is clearly multiplicative and $T$-linear. The bend relations are preserved for the following reason. If a relation $a\leq\sum b_j$ in $B$ induces the relation $a\otimes 1+\sum b_j\otimes 1\=\sum b_j\otimes 1$ on $\Bend_v(B)$, then the image relation $f(a)\leq\sum f(b_j)$ in $C$ induces the relation $f(a)\otimes 1+\sum f(b_j)\otimes 1\=\sum f(b_j)\otimes 1$ on $\Bend_v(C)$. This defines the \emph{bend functor} 
\[
 \Bend_v: \ \Alg^\ob_k \ \longrightarrow \ \Alg^\ob_T.
\]

Note that the bend of $B$ along $v$ is idempotent since $1\leq 1$ in $B$ implies $1\otimes 1+1\otimes 1\=1\otimes 1$ in $\bend_v(B)$. Since the bend of $B$ depends only on relations of the form $a\leq\sum b_j$, we have $\Bend_v(B^\mon)=\Bend_v(B)$. By the very definition of $\bend_v(B)$, the bend of $B$ is algebraic if $T$ is so.

\begin{lemma}\label{lemma: bend relations commute with localizations}
 Let $B$ be an ordered blue $k$-algebra and $S\subset B$ a multiplicative subset. Let $S_v=S\otimes\{1\}$ be its image in $\Bend_v(B)$. Then the association $\frac{a\otimes b}s\mapsto\frac as\otimes b$ defines a canonical isomorphism $S_v^{-1}\Bend_v(B)\to\Bend_v(S^{-1}B)$.
\end{lemma}

\begin{proof}
 The canonical morphism $S_v^{-1}\Bend_v(B)\to\Bend_v(S^{-1}B)$ is induced by the isomorphism of ordered blue $T$-algebras
\[
  S_v^{-1}\bigl(B^\bullet\otimes_{k^\bullet}T\bigr) \quad \longrightarrow \quad \bigl(S^{-1}B^\bullet\otimes_{k^\bullet}T\bigr).
 \]
 We have to show that this morphism of monoids identifies the respective subadditions of $S_v^{-1}\bend_v(B)$ and $\bend_v(S^{-1}B)$. 

 The generators for the bend relations of $S_v^{-1}\Bend_v(B)$ and $\Bend_v(S^{-1}B)$ are of the respective forms
 \[\textstyle
  \frac{a\otimes 1}{s\otimes 1} \ + \ \sum \frac{b_j\otimes 1}{s\otimes 1} \ \= \ \sum\frac{b_j\otimes 1}{s\otimes 1} \qquad \text{and} \qquad \frac{a}{s_a}\otimes 1 \ + \ \sum \frac{b_j}{s_j}\otimes 1 \ \= \ \sum\frac{b_j}{s_j}\otimes 1
 \]
 where $s,s_a,s_j\in S$ and $a\leq\sum b_j$ in $B$. Note that we use that the subaddition of $S^{-1}B$ is generated by the subaddition of $B$. Relations of the former type are relations of the latter type, which implies that the canonical map $S_v^{-1}\Bend_v(B)\to\Bend_v(S^{-1}B)$ is a morphism of ordered blueprints. Conversely, every relation of the latter form can be rewritten as
 \[\textstyle
  \frac{s^{(a)}a\otimes 1}{s\otimes 1} \ + \ \sum \frac{s^{(j)}b_j\otimes 1}{s\otimes 1} \quad \= \quad \sum \frac{s^{(j)}b_j\otimes 1}{s\otimes 1}
 \]
 where $s=s_a\cdot\prod s_j$, $s^{(a)}=\prod s_j$ and $s^{(j)}=s_a\cdot\prod_{i\neq j}s_i$.
\end{proof}

\begin{lemma}\label{lemma: bend commutes with tensor product}
 Let $v:k\to T$ be a valuation. The functor $\Bend_v:\Alg_k^\ob\to\Alg_T^\ob$ commutes with non-empty colimits.
\end{lemma}

\begin{proof}
 The statement follows if we can show that $\Bend_v$ commutes with cofibre products and non-empty coproducts.

 Let $C\leftarrow B\to D$ be a diagram of $k$-algebras. The cofibre product of $C$ and $D$ over $B$ is the tensor product $C\otimes_BD$. We have $(C\otimes_BD)^\bullet = C^\bullet\otimes_{B^\bullet}D^\bullet$ and
 \[
  \Bend_v(C\otimes_BD) \ = \ \bpquot{C^\bullet\otimes_{B^\bullet}D^\bullet\otimes_{k^\bullet} T}{\bend_v(C\otimes_BD)}.
 \]
 A relation $a\otimes b\leq \sum c_i\otimes d_i$ in $C\otimes_BD$ must be the sum of a relation of the form $a\otimes 1\leq \sum c_i\otimes 1$ or $1\otimes b\leq\sum 1\otimes d_j$ with relations of the form $0\leq\sum c'_j\otimes d'_j$. Since addition of relations of the latter type do not contribute to $\bend_v(C\otimes_BD)$, we can restrict our attention to relations of the former types. If we write 
 \[\textstyle
  \bend_v(C)\otimes 1 \ = \ \Big\langle \, a\otimes 1\otimes 1+\sum c_i\otimes 1\otimes1 \= \sum c_i\otimes 1\otimes1 \, \Bigl| \, a\leq \sum c_i \text{ in }C \, \Bigr\rangle
 \]
  and
 \[\textstyle
  1\otimes \bend_v(D) \ = \ \Big\langle \, 1\otimes b\otimes 1+\sum 1\otimes d_j\otimes1 \= \sum 1\otimes d_j\otimes1 \, \Bigl| \, b\leq \sum d_j \text{ in }D \, \Bigr\rangle,
 \]
 then we conclude that $\bend_v(C\otimes_BD)=\gen{\bend_v(C)\otimes1,1\otimes\bend_v(D)}$. Since
 \[
  C^\bullet\otimes_{B^\bullet}D^\bullet\otimes_{k^\bullet} T \ = \ \bigl(C^\bullet\otimes_{k^\bullet}T\bigr) \otimes_{(B^\bullet\otimes_{k^\bullet}T)} \bigl(D^\bullet\otimes_{k^\bullet} T \bigr),
 \]
 we conclude that
 \begin{multline*}
  \bpgenquot{C^\bullet\otimes_{B^\bullet}D^\bullet\otimes_{k^\bullet} T}{\bend_v(C)\otimes1,1\otimes\bend_v(D)}\\ 
  = \ \bigl(\bpquot{C^\bullet\otimes_{k^\bullet}T}{\bend_v(B)} \bigr) \otimes_{(\bpquot{B^\bullet\otimes_{k^\bullet}T}{\bend_v(B)})} \bigl(\bpquot{D^\bullet\otimes_{k^\bullet} T}{\bend_v(D)} \bigr),
 \end{multline*}
 which is $\Bend_v(C)\otimes_{\Bend_v(B)}\Bend_v(D)$ as desired.
 
 Since the coproduct of a non-empty family $\{B_i\}$ of ordered blue $k$-algebras is represented by the (possibly infinite) tensor product $\bigotimes_k B_i$ over $k$, it follows from the same argument as for the cofibre product that $\Bend_v$ commutes with non-empty coproducts. This completes the proof of the lemma.
\end{proof}

\begin{rem}
 Since $\Bend_v(k)$ is idempotent, it is clear that $\Bend_v$ does not preserve initial objects if $T$ is not idempotent. If, however, $T$ is idempotent, then it can be proven that $\Bend_v(k)$ is isomorphic to $T$. It follows that $\Bend_v$ preserves all colimits if $T$ is idempotent.
\end{rem}

\begin{prop}\label{prop: construction of the bend functor for schemes}
 Let $v:k\to T$ be a valuation. The bend functor extends to a functor 
 \[
  \Bend_v:\Sch_k\to\Sch_T
 \]
 that sends an affine open $U=\Spec B$ of an ordered blue $k$-scheme $X$ to the affine open $\Bend_v(U)=\Spec\bigl(\Bend_v(B)\bigr)$ of the ordered blue $T$-scheme $\Bend_v(X)$.
\end{prop}

\begin{proof}
 This follows from Lemma \ref{lemma: extension of a morphism of sites to a functor between schemes} whose hypotheses are verified by Lemmas \ref{lemma: bend relations commute with localizations} and \ref{lemma: bend commutes with tensor product}.
\end{proof}

For later reference, we provide the following fact.

\begin{prop}\label{prop: the bend preserves open closed immersions}
 Let $v:k\to T$ be a valuation. The bend functor $\Bend_v:\Sch_k\to\Sch_\T$ preserves open and closed immersions.
\end{prop}

\begin{proof}
 That $\Bend_v$ preserves open immersions follows immediately from Lemma \ref{lemma: bend relations commute with localizations}. The question of whether $\Bend_v$ preserves closed immersions can be reduced to the affine case by choosing an open affine covering. Thus it suffices to consider a surjection $f:B\to C$ of ordered blue $k$-algebras and to show that $\Bend_v(f)$ is a surjection. But this follows at once from the definition of $\Bend_v(f)$ as the map $f\otimes\id_T:\bigl(\bpquot{B^\bullet\otimes_{k^\bullet}T}{\bend_v(B)}\bigr) \to \bigl(\bpquot{C^\bullet\otimes_{k^\bullet}T}{\bend_v(C)}\bigr)$.
\end{proof}

\begin{ex}\label{ex: bend}
 We revisit Example \ref{ex: tropicalization with a totally positive base}, this time with tropicalization basis $\T$. Let $k$ be a field and $v:k\to\T$ a non-archimedean absolute value. Let $X=\Spec B$ be a closed subscheme of $\A^n$ with $B=\bpquot{k[T_1,\dotsc,T_n]}{\cR}$. If we identify $k^\bullet[T_1,\dotsc,T_n]\otimes_{k^\bullet} \T$ with $\T[T_1,\dotsc,T_n]$ (see Example \ref{ex: tropicalization of monoid algebras}) and denote the universal valuation by $w:k[T_1,\dotsc,T_n]\to \T[T_1,\dotsc,T_n]$, then we get $\Bend_v(X)=\Spec \bpquot{\T[T_1,\dotsc,T_n]}{\bend_v(B)}$ where
 \[\textstyle
  \bend_v(B) \ = \ \big\langle w(a)+\sum w(b_j) \= \sum w(b_j) \,\big|\, a\leq \sum b_j\text{ in }B \big\rangle.
 \]

\end{ex}

\begin{comment}

 As a concrete example, let us a field $k$ with non-archimedean valuation $v:k\to\T$, which can also be seen as a valuation $v:k\to\T^\pos$ into the totally positive ordered blueprint $\T^\pos=\bpquot{\T}{\gen{0\leq1}}$ since the canonical map $\T\to\T^\pos$ is a bijection by Corollary \ref{cor: idempotent implies core=core after pos and B subset B-pos}.
 
 Let $X=\Spec B$ be a closed subscheme of $\A^n_k$ with $B=\bpquot{k[T_1,\dotsc,T_n]}{\cR}$. Then $X^\mon$ is the spectrum of $B^\mon=\bpquot{k^\mon[T_1,\dotsc,T_n]}{\cR^\mon}$ where
 \[\textstyle
  R^\mon \ = \ \big\{ b \leq \sum a_i \, \big| \, b\leq\sum a_i\text{ in }\cR\big\}.
 \]
 If we identify $k^\mon[T_1,\dotsc,T_n]\otimes_{k^\mon} \T^\pos$ with $\T^\pos[T_1,\dotsc,T_n]$ (see Example \ref{ex: tropicalization of monoid algebras}) and denote the universal valuation by $w^\univ:k[T_1,\dotsc,T_n]\to \T^\pos[T_1,\dotsc,T_n]$, then we get $\Trop_v(X)=\Spec \bpquot{\T^\pos[T_1,\dotsc,T_n]}{\widetilde\cR}$ where 
 \[\textstyle
  \widetilde\cR \ = \ \big\langle w^\univ(b)\leq\sum w^\univ(a_i) \, \big| \, b\leq \sum a_i\text{ in }\cR \big\rangle.
 \]
\end{comment}

%%%%%%%%%%%%%%%%%%%%%%%%%%%%%%%%%%%%%%%%%%%%%%%%%%%%%%%%%%%%%%%%%%%%%%%%%%%%%%%%%%%%%%%%%%%%%%%%%%%%%%%%%%%%%%%%%%%%%%%%%%%%%%%%%%%%%%%%%%%%%%%%%%%%%%%%%%%%%%%%%%%%%%%%%%%

\subsection{Tropicalization for an idempotent tropicalization base}
\label{subsection: tropicalization for an idempotent base}

We show that the bend functor is a tropicalization if the tropicalization base is idempotent.

\begin{thm}\label{thm: tropicalization for idempotent base}
 Let $v:k\to T$ be a valuation in an idempotent ordered blueprint $T$. Let $X$ be an ordered blue $k$-scheme. Then there is a canonical isomorphism 
 \[
  \Bend_v(X) \quad \stackrel\sim\longrightarrow \quad (X^\mon\otimes_{k^\mon}T^\pos)^\core\otimes_{T^\core}T
 \]
 of ordered blue $T$-schemes, and $\Trop_v(X)=\Bend_v(X)$ is a tropicalization of $X$ along $v$.
\end{thm}

\begin{proof}
 Since all functors are defined in terms of affine presentations, we can immediately reduce the claim of the theorem to the affine situation $X=\Spec B$.

 Let $B=\bpquot A\cR$ be a representation of $B$ and $B_v=B^\mon\otimes_{k^\mon}T^\pos$. Then $\Val_v(B,S)$ equals the morphism set $\Hom_T(B_v,S^\pos)$ of ordered blue $T$-algebras, and there are surjections 
 \[
  A\otimes_{k^\bullet}T \quad \longrightarrow \quad B_v \qquad \text{and} \qquad A\otimes_{k^\bullet}T \quad \longrightarrow \quad \Bend_v(B)
 \]
 of ordered blue $T$-algebras. Note further that the canonical morphism $B_v^\core\to B_v$ is a bijection, which factors into bijections 
 \[
  B_v^\core \quad \longrightarrow \quad B_v^\core\otimes_{T^\core}T \quad \longrightarrow \quad B_v.
 \]
 Every ordered blue $T$-algebra $S$ is idempotent. Thus by Corollary \ref{cor: idempotent implies core=core after pos and B subset B-pos}, $S^\core\to(S^\pos)^\core$ is an isomorphism and $S\to S^\pos$ is a bijection.
  
 Putting these facts together, we see that each of the homomorphism sets
 \[
  \Hom_T(B_v,S^\pos), \quad \Hom_T(B_v^\core\otimes_{T^\core}T,S) \quad \text{and} \quad \Hom_T(\Bend_v(B),S)
 \]
 of ordered blue $T$-algebras embeds canonically into $\Hom_T(A\otimes_{K^\bullet}T,S^\pos)$. We claim that these three homomorphism sets are equal as subsets of $\Hom_T(A\otimes_{K^\bullet}T,S^\pos)$.
 
 Once we have proven this, it follows that $\Bend_v(B)$ is isomorphic to $B_v^\core\otimes_{T^\core}T$ and represents the functor $\Val_v(B,-)=\Hom_T(B_v,-)$. Note that we can describe the isomorphism $f:\Bend_v(B)\rightarrow B_v^\core\otimes_{T^\core}T$ explicitly as the following map: for an element $a\in \Bend_v(B)$, choose an inverse image $a'$ in $A\otimes_{k^\bullet}T$ and define $f(a)$ as the image of $a'$ in $B_v^\core\otimes_{T^\core}T$.
 
 We prove the equality of the three homomorphism sets in question by circular inclusions. We begin with the inclusion $\Hom_T(B_v,S^\pos)\subset\Hom_T(B_v^\core\otimes_{T^\core}T,S)$. The map
 \[
  (-)^\core: \ \Hom_T(B_v,S^\pos) \quad \longrightarrow \quad \Hom_{T^\core}(B_v^\core,(S^\pos)^\core),
 \]
 is injective since the canonical morphisms $B_v^\core\to B_v$ and $(S^\pos)^\core\to S^\pos$ are bijections. The inclusion $(S^\pos)^\core=S^\core \hookrightarrow S$ yields an inclusion 
 \[
  \Hom_{T^\core}(B_v^\core,(S^\pos)^\core) \quad \subset \quad \Hom_{T^\core}(B_v^\core,S)
 \]
 whose codomain is equal to $\Hom_T(B_v^\core\otimes_{T^\core}T,S)$. Altogether this yields the desired inclusion $\Hom_T(B_v,S^\pos) \subset \Hom_T(B_v^\core\otimes_{T^\core}T,S)$.
 
 We proceed with the inclusion $\Hom_T(B_v^\core\otimes_{T^\core}T,S) \subset \Hom_T(\Bend_v(B),S)$. Since $\Bend_v(B)=\bpquot{A\otimes_{k^\bullet}T}{\bend_v(B)}$, this inclusion follows if we can show that the subaddition of $B_v^\core$ contains the bend relation $\bend_v(B)$. Consider $a\leq\sum b_j$ in $B$. Then $a\otimes 1\leq\sum b_j\otimes 1$ in $B_v=B^\mon\otimes_{k^\mon}T^\pos$. Since $T^\pos$ is idempotent and totally positive, $B_v$ is so, too. This implies that
 \[\textstyle
  a\otimes 1 \ + \ \sum b_j\otimes 1 \quad \leq \quad \sum b_j\otimes 1 \ + \ \sum b_j\otimes 1 \quad \= \quad \sum b_j\otimes 1
 \]
 and
 \[\textstyle
  \sum b_j\otimes 1 \quad \= \quad 0 \ + \ \sum b_j\otimes 1 \quad \leq \quad a\otimes 1 \ + \ \sum b_j\otimes 1.
 \]
 Thus $\sum b_j\otimes 1\=a\otimes 1+\sum b_j\otimes 1$ in $B_v^\core$. This shows that the subaddition of $B_v^\core$ contains $\bend_v(B)$.

 We proceed with the inclusion $\Hom_T(\Bend_v(B),S) \subset \Hom_T(B_v,S^\pos)$. Consider a $T$-morphism $f:\Bend_v(B)\to S$. We can represent $B_v$ as $\bpquot{A\otimes_{k^\bullet}T^\pos}{\cR_v}$ where $\cR_v$ is the relation on $A\otimes_{k^\bullet}T^\pos$ that is generated by the image of the relation of $B^\mon$ in $B_v$. The composition
 \[
  f': \ A\otimes_{k^\bullet}T \quad \longrightarrow \quad \Bend_v(B) \quad \longrightarrow \quad S \quad \longrightarrow \quad S^\pos
 \]
 induces a $T$-morphism $B_v\to S^\pos$ if $f'$ maps the relation $\cR_v$ to the subaddition of $S^\pos$. This can be verified on the generators of $\cR_v$, i.e.\ we have to consider only the relations of the form $a\otimes 1\leq \sum b_j\otimes 1$ for which $a\leq\sum b_j$ in $B$. The latter relation yields the bend relation
 \[\textstyle
  a\otimes 1 \ + \ \sum b_j\otimes 1 \quad \=\quad \sum b_j\otimes 1
 \]
 in $\Bend_v(B)$. Consequently, $S$ contains the relation
 \[\textstyle
  f(a\otimes 1) \ + \ \sum f(b_j\otimes 1) \quad \=\quad \sum f(b_j\otimes 1),
 \]
 and by Lemma \ref{lemma: totally positive blueprints}, we have $f'(a\otimes 1)\leq \sum f'(b_j\otimes 1)$ in $S^\pos$. This shows that $f'$ induces a $T$-morphism $B_v\to S^\pos$, which implies the last inclusion and completes the proof of the theorem.
\end{proof}

Let $v:k\to T$ be a valuation and $X$ an ordered blue $k$-scheme. We obtain the following immediate consequences.

\begin{cor}\label{cor: tropicalization for algebraic and idempotent base}
 If $T$ is algebraic and idempotent, then $\Trop_v(X)$ is algebraic and isomorphic to $\Bend_v(X)=(X^\mon\otimes_{k^\mon}T^\pos)^\core$.\qed
\end{cor}

A comparison of Theorem \ref{thm: tropicalization for idempotent base} with Theorem \ref{thm: tropicalization for totally positive tropical base} yields:

\begin{cor}\label{cor: tropicalization for totally positive and idempotent base}
 If $T$ is totally positive and idempotent, then $\Bend_v(X)=X^\mon\otimes_{k^\mon}T$ is a tropicalization of $X$ along $v$.\qed
\end{cor}

\begin{rem}\label{rem: proofs for totally positive and idempotent bases}
 It seems odd that we need two completely different proofs for the existence of $\Trop_v(B)$ in the case that $T$ is totally positive and the case that $T$ is idempotent. The question whether there is a unified proof that extends to a possibly larger class of tropical bases suggests itself. Note, however, that the non-obvious identity 
 \[
  (B^\mon\otimes_{k^\mon}T^\pos)^\core\otimes_{T^\core}T \quad = \quad B^\mon\otimes_{k^\mon}T^\pos
 \]
 for totally positive and idempotent $T$, which results from the two different proofs, fails to hold for a merely totally positive $T$ or a merely idempotent $T$.
\end{rem}

%%%%%%%%%%%%%%%%%%%%%%%%%%%%%%%%%%%%%%%%%%%%%%%%%%%%%%%%%%%%%%%%%%%%%%%%%%%%%%%%%%%%%%%%%%%%%%%%%%%%%%%%%%%%%%%%%%%%%%%%%%%%%%%%%%%%%%%%%%%%%%%%%%%%%%%%%%%%%%%%%%%%%%%%%%%
%%%%%%%%%%%%%%%%%%%%%%%%%%%%%%%%%%%%%%%%%%%%%%%%%%%%%%%%%%%%%%%%%%%%%%%%%%%%%%%%%%%%%%%%%%%%%%%%%%%%%%%%%%%%%%%%%%%%%%%%%%%%%%%%%%%%%%%%%%%%%%%%%%%%%%%%%%%%%%%%%%%%%%%%%%%

\section{Berkovich analytification}
\label{section: Berkovich analytification}

To begin with, we review the definition of the Berkovich analytification as a topological space. See \cite{Baker07} or \cite{Berkovich90} for more details.

Let $k$ be a field with (non-archimedean) valuation $v:k\to \T$ and $X$ a $k$-scheme. We consider $\T$ together with its real topology, cf.\ Example \ref{ex: real topology of the tropical numbers}. 

If $X=\Spec B$ is affine, then the \emph{Berkovich analytification of $X$} is the set $X^\an$ of all valuations $w:B\to \T$ whose restriction to $k$ is $v$, endowed with the compact-open topology with respect to the discrete topology on $B$. In other words, $X^\an$ is equipped with the coarsest topology such that the evaluation maps
\[
 \begin{array}{cccc}\ev_a: & X^\an & \longrightarrow & \T \\ & w & \longmapsto & w(a) \end{array}
\]
are continuous for all $a\in B$.

If $X$ is an arbitrary $k$-scheme and $\cU$ an affine presentation of $X$, then we define $X^\an=\colim\cU^\an$ as a topological space. Note that this definition is independent of the chosen affine presentation for similar arguments as used in the proof of Theorem \ref{thm: properties of the fine topology}.

In order to connect the Berkovich space of $X$ with scheme theoretic tropicalization of $X$, we need to choose a \emph{blue model of $X$}, which is an ordered blue $k$-scheme $Z$ together with an isomorphism $Z^+\to X$. Given such a blue model $Z$ of $X$, we can consider the space $\Bend_v(X)(\T)$ of $\T$-rational points together with the fine topology.

\begin{ex}
 One can define a blue model of a $k$-scheme $X$ in the following way. Let $\{U_i\}$ be an open affine covering of $X$ and for every pair of indices $i$ and $j$, let $\{U_{i,j,k}\}$ be an open affine covering of $U_i\cap U_j$. Then we can define an affine presentation $\cU$ of $X$ as follows. Its objects are all the affine open subschemes $U_i$ and $U_{i,j,k}$ and its morphisms are the open immersions $\iota_{i,j,k,l}:U_{i,j,k}\to U_l$ whenever $U_{i,j,k}$ is a subset of $U_l$ in $X$. Then $X$ is the colimit of $\cU$ in the category $\Sch_k^+$ of $k$-schemes.
 
 Note that strictly speaking, one has to identify $U_{i,j,k}$ with $U_l$ whenever $\iota_{i,j,k,l}$ is an isomorphism to obtain an affine presentation in the stricter sense of this paper. In the more ample sense of \cite{Lorscheid17}, this would not be necessary. We will ignore this issue in the following.
 
 Every affine $k$-scheme $U$ has a canonical blue model, which is the ordered blue scheme $U^\blue=\Spec\Gamma U^\blue$ where $\Gamma U^\blue$ is the ordered blueprint associated with the semiring $\Gamma U$. Clearly, there is a canonical isomorphism $(U^\blue)^+\to U$. Similarly, a morphism $\iota:V\to U$ of affine ordered blue schemes has a canonical blue model, which is the induced morphism $\iota^\blue=(\Gamma\iota)^\ast: V^\blue\to U^\blue$ between the associated blue models of $V$ and $U$. The canonical isomorphisms $(V^\blue)^+\simeq Z$ and $(U^\blue)^+\simeq U$ identify $(\iota^\blue)^+$ with $\iota$.
 
 Let $\cV=\cU^\blue$ be the affine presentation that consists of the affine ordered blue $k$-schemes $U_i^\blue$ and $U_{i,j,k}^\blue$ and the open immersions $\iota_{i,j,k,l}^\blue$. Note that $\cV^+$ is canonically identified with $\cU$, i.e.\ $X=\colim\cV^+$. We define the \emph{blue model of $X$ associated with $\cU$} as the colimit $Z$ of $\cV$ in $\OBSch$. It comes together with a canonical isomorphism $Z^+\to X$.
 
 Note that the blue models $Z$ obtained in this way have the property that for every open subset $U$ of $Z$, $\cO_Z(U)$ is a ring. However, blue models stemming from different affine presentations are in general not isomorphic. In particular, the closed points of a blue model $Z$ associated with $\cU$ stay bijectively in correspondence with the open subsets $U_i$ of $\cU$, provided that $U_i\not\subset U_j$ for $i\neq j$.
\end{ex}

\begin{thm}\label{thm: Berkovich analytification as rational point set}
 Let $k$ be a field with valuation $v:k\to\T$, $X$ a $k$-scheme and $Z$ a blue model of $X$ such that $\cO_Z(U)$ is a ring for all open subsets $U$ of $Z$. Then the Berkovich space $X^\an$ is naturally homeomorphic to $\Bend_v(Z)(\T)$.
\end{thm}

\begin{proof}
 Let $\cV$ be an affine presentation of $Z$. Then $\cU=\cV^+$ is an affine presentation of $X$ since $\colim\cV^+\simeq Z^+\simeq X$. By definition, we have $X^\an=\colim \cU^\an$. For every $U$ in $\cU$ and $V=U^\blue$, we have $U^\an=\Val_v(\Gamma U,\T)=\Val_v(V,\T)$. Thus $X^\an=\colim\cU^\an=\colim\Val_v(\cV,\T)$.
  
 On the other hand, we have $\Bend_v(Z)=\colim\Bend_v(\cZ)$, by definition. By Proposition \ref{prop: the fine topology as colimit over an affine presentation}, we have an identification $\big(\colim\Bend_v(\cV)\big)(\T)=\colim\big(\Bend_v(\cV)(\T)\big)$ of topological spaces. This reduces the claim of the theorem to the affine case.
 
 If $X=\Spec R$ for a $k$-algebra $R$ and $Z=X^\blue=\Spec B$ where $B$ is $R$, considered as an ordered blueprint, then it follows from Theorem \ref{thm: tropicalization for idempotent base} that $X^\an=\Val_v(B,\T)=\Bend_v(Z)(\T)$ as sets since $\T$ is idempotent.

 We are left with showing that the topologies of $X^\an$ and $\Bend_v(B)(\T)$ agree. The canonical bijection $\Psi:\Bend_v(Z)(\T)\to \Val_v(B,\T)$ is given explicitly as
 \[
   \Big(f:\Bend_v(B) \longrightarrow \T\Big) \quad \longmapsto \quad \Big( w: B \stackrel{1:1}\longrightarrow B^\bullet \longrightarrow \Bend_v(B) \stackrel f\longrightarrow \T\Big)
 \]
 where we use that $\Bend_v(B)$ is defined as $\bpquot{B^\bullet\otimes_{k^\bullet}\T}{\bend_v(B)}$. By definition of the affine topology, the topology of $\Bend_v(Z)(\T)$ is generated by subsets of the form
 \[
  U_{a\otimes t, W} \quad = \quad \big\{ \, f:\Bend_v(B)\to \T \, \big| \, f(a\otimes t)\in W \, \big\}
 \]
 where $a\in B$, $t\in \T$ and $W\subset \T$ is an open subset, while the topology of $B^\an$ is generated by subsets of the form
 \[
  U_{a,W} \quad = \quad \big\{ \, w:B\to \T \, \big| \, w(a)\in W \, \big\}
 \]
 where $a\in B$ and $W\subset \T$ is an open subset. It is easily verified that $\Psi^{-1}(U_{a,W})=U_{a\otimes 1,W}$. Thus $\Psi$ is continuous. 
 
 We are left with proving that $\Psi$ is open. Consider a basic open $U_{a\otimes t,V}$ of $\Bend_v(Z)(\T)$ and denote by $m_t:\T\to \T$ the multiplication with $t$. Then we have
 \begin{align*}
  U_{a\otimes t,V} \quad &= \quad \big\{ \, f:\Bend_v(B)\to \T \, \big| \, f(a\otimes t)\in V \, \big\} \quad \\
                        &= \quad \big\{ \, f:\Bend_v(B)\to \T \, \big| \, f(a\otimes 1)\in m_t^{-1}(V) \, \big\} \quad = \quad U_{a\otimes 1,m_t^{-1}(V)}.
 \end{align*}
 Since the multiplication of $\T$ is continuous, $m_t^{-1}(V)$ is an open subset of $T$. Thus we see that $\Psi(U_{a\otimes t,V})=U_{a,m_t^{-1}(V)}$ is open in $X^\an$, which concludes the proof of the theorem.
\end{proof}

\begin{rem}
 Although the topological space $\Bend_v(Z)(\T)$ is canonically homeomorphic to the Berkovich space $X^\an$, the tropicalization $\Bend_v(Z)$ depends as much as $Z$ on additional choices. In general, different blue models $Z$ and $Z'$ of $X$ give rise to non-isomorphic tropicalizations $\Bend_v(Z)$ and $\Bend_v(Z')$. In particular, the number of closed points of $\Bend_v(Z)$ agrees with the number of closed points of $Z$, which can vary among the different blue models of $X$.
\end{rem}

\begin{rem}
 The same technique of proof can be used to show that $X^\an$ is homeomorphic to $\Hom_\T(X^\mon\otimes_{k^\mon}\T^\pos,\T^\pos)$. This viewpoint can be extended to the Berkovich space $\cM(B)$ of all, possibly archimedean, valuations of a ring $B$ as follows. Also confer the work Berkovich (\cite[Section 1]{Berkovich90}) and Poineau (\cite{Poineau10}) on the Berkovich space $\cM(\Z)$ of the arithmetic line.
 
 Consider the trivial valuation $v:\Fun\to \Fun$. Then every valuation $w:B\to \R_{\geq0}$ is an extension of $v$. Define $B_v=B^\mon\otimes_{\F_1^\mon} \F_1^\pos=(B^\mon)^\pos$ and consider $\R_{\geq0}^\pos$ with the real topology. Then $\cM(B)$ is naturally homeomorphic to $\Hom_\Fun(B_v,\R_{\geq0}^\pos)$, endowed with the fine topology. 
\end{rem}

\begin{ex}\label{ex: Berkovich analytification}
 As an explicit example, we consider an algebraically closed field $k$ with non-archimedean valuation $v:k\to\T$ and the polynomial algebra $k[T]^+$. As a monoid, $(k[T]^+)^\bullet$ is freely generated over $k^\bullet$ by the irreducible poynomials $T-c$ for $c\in k$. So the underlying monoid of $\Bend_v(k[T]^+)$ is isomorphic to $\T[T_c\mid c\in k]$ where $T_c=(T-c)\otimes 1$. The complexity of the Berkovich affine line $\A^{1,\an}_k$ over $k$ is reflected by the complexity of the bend relation 
 \[\textstyle
  \bend_v(k[T]^+) \ = \ \big\langle a\otimes 1+\sum b_j\otimes 1\=\sum b_j\otimes 1 \,\big|\, a=\sum b_j\text{ in }k[T]^+ \big\rangle.
 \]
\end{ex}

%%%%%%%%%%%%%%%%%%%%%%%%%%%%%%%%%%%%%%%%%%%%%%%%%%%%%%%%%%%%%%%%%%%%%%%%%%%%%%%%%%%%%%%%%%%%%%%%%%%%%%%%%%%%%%%%%%%%%%%%%%%%%%%%%%%%%%%%%%%%%%%%%%%%%%%%%%%%%%%%%%%%%%%%%%%
%%%%%%%%%%%%%%%%%%%%%%%%%%%%%%%%%%%%%%%%%%%%%%%%%%%%%%%%%%%%%%%%%%%%%%%%%%%%%%%%%%%%%%%%%%%%%%%%%%%%%%%%%%%%%%%%%%%%%%%%%%%%%%%%%%%%%%%%%%%%%%%%%%%%%%%%%%%%%%%%%%%%%%%%%%%

\section{Kajiwara-Payne tropicalization}
\label{section: Kajiwara-Payne tropicalization}

Before we explain how to recover the Kajiwara-Payne tropicalization as a rational point set of a scheme theoretic tropicalization, we review the definition of toric varieties and the theory of Kajiwara (\cite{Kajiwara08}) and Payne (\cite{Payne09}). This shall serve the reader as a reminder and allows us to fix notation.

\subsection{Toric varieties}\label{subsection: toric varieties}
Let $N_\R$ be a real vector space with a lattice $N_\Z$. A \emph{(strongly convex polyhedral rational) cone in $N_\R$} is a convex cone $\tau$ of $N_\R$ such that $\tau\cap(-\tau)=\{0\}$ and that is generated over $\R_{\geq0}$ by finitely many elements of $N_\Z$. A \emph{face of $\tau$} is the intersection of $\tau$ with a half space $H$ of $\N_\R$ such that the intersection is either equal to $\tau$ or is contained in the boundary of $\tau$. A \emph{fan in $N_\R$} is a collection $\Delta$ of cones in $N_\R$ such that every face of a cone in $\Delta$ is in $\Delta$ and such that the intersection of two cones in $\Delta$ is a face of each of these cones.

The \emph{dual cone} of a cone $\tau$ is the subsemigroup $\tau^\vee=\{x\in N_\R^\vee| <x,y>\geq 0\text{ for all }y\in\tau\}$ of the dual vector space $N_\R^\vee$. We define the monoid $A_\tau$ with $0$ that results from writing $\tau^\vee\cap N_\Z^\vee$ multiplicatively and adjoining an additional element $0$. Then an inclusion $\sigma\subset\tau$ of cones in $\Delta$ yields a finite localization $A_\tau\to A_\sigma$ of monoids. This defines a diagram $\cD$ of monoids $A_\tau$ with zero and finite localizations.

Applying the functor $\Spec$ to the diagram $\cD$ yields an affine presentation $\cU$ in affine monoid schemes. The base extension $\cU_k^+$ is an affine presentation in affine $k$-schemes. We define the \emph{toric variety associated with $\Delta$} as $X(\Delta)=\colim\cU_k^+$.

\subsection{Tropicalization of closed subvarieties}
We begin with the tropicalization of an affine toric variety. Let $\tau$ be a cone in $N_\R$ and $U_\tau=\Spec A_\tau$. Then $X_\tau=U_{\tau,k}^+$ is an affine toric variety. Its analytification $X_\tau^\an$ is the set $\Val_v(k[A_\tau]^+,\T)$ of all valuations $w:k[A_\tau]^+\to \T$ that extend $v$, endowed with the compact-open topology. The \emph{Kajiwara-Payne tropicalization $\Trop_v^{KP}(X_\tau)$ of $X_\tau$} is defined as the set $\Hom(A_\tau,\T)$ of monoid morphisms, endowed with the compact-open topology where we regard $A_\tau$ as a discrete monoid. It comes with the continuous surjection
\[
 \trop_{v,\tau}^{KP}:\quad X_\tau^\an \ = \ \Val_v(k[A_\tau]^+,\T) \quad \longrightarrow \quad \Hom(A_\tau,\T) \ = \ \Trop_v^{KP}(X_\tau)
\]
that restricts a valuation $w:k[A_\tau]^+\to \T$ to the monoid morphism $w\vert_{A_\tau}:A_\tau\to \T$.

This construction is compatible with gluing affine pieces along principal open subsets. Namely, let $\Delta$ be a fan in $N_\R$. Then $X(\Delta)=\colim X_\tau$, with respect to the inclusions $X_\sigma\subset X_\tau$ whenever $\sigma\subset\tau$, and $X(\Delta)^\an=\colim X_\tau^\an$ as topological space. We define $\Trop_v^{KP}(X(\Delta))$ as colimit of the topological spaces $\Hom(A_\tau,\T)$ with respect to the open topological embeddings $\Hom(A_\sigma,\T)\to\Hom(A_\tau,\T)$ that are induced by inclusions $\sigma\subset \tau$. Moreover the map $X_\tau^\an\to \Trop_v^{KP}(X_\tau)$ extends to a continuous surjection
\[
 \trop_{v,\Delta}^{KP}: \quad X(\Delta)^\an \quad \longrightarrow \quad \Trop_v^{KP}(X(\Delta)).
\]

A closed immersion $\iota:Y\to X(\Delta)$ of a $k$-scheme $Y$ into the toric $k$-variety $X(\Delta)$ yields a closed embedding $\iota^\an:Y^\an\to X(\Delta)^\an$ of topological spaces. We define the \emph{Kajiwara-Payne tropicalization $\Trop_{v,\iota}^{KP}(Y)$ of $Y$} as the image $\trop_{v,\Delta}^{KP}(Y)$, endowed with the subspace topology with respect to the inclusion $\iota^\trop:\Trop_{v,\iota}^{KP}(Y)\to \Trop_v^{KP}(X(\Delta))$. The Kajiwara-Payne tropicalization of $Y$ comes with a surjective continuous map
\[
 \trop_{v,\iota}^{KP}: \quad Y^\an \quad \longrightarrow \quad \Trop_{v,\iota}^{KP}(Y).
\]

\subsection{The associated blue scheme}
\label{subsection: Kajiwara-Payne - the associated blue scheme}
We explain how to obtain $\Trop_{v,\iota}^{KP}(Y)$ as a set of rational points of a scheme theoretic tropicalization $\Trop_v(Z)$ of a suitable blue $k$-scheme $Z$ that we define in the following. 

If $X(\Delta)=X_\tau=\Spec k[A_\tau]^+$ is affine, then the closed immersion $\iota:Y\to X_\tau$ corresponds to a surjection $\pi:k[A_\tau]^+\to \Gamma Y$ of rings. We define the blue $k$-scheme $Z_\tau$ as $\Spec\bpquot{k[A_\tau]}{\cR_\tau}$ where 
\[\textstyle
 \cR_\tau \quad = \quad \left\{ \ \sum a_i\=\sum b_j \ \left| \ \sum \pi(a_i)=\sum\pi(b_j) \ \right.\right\}.
\]
Note that the natural inclusion $\bpquot{k[A_\tau]}{\cR_\tau}\to \Gamma Y$ induces a morphism $\beta_\tau:Y\to Z_\tau$ whose associated morphism $\beta_\tau^+:Y\to Z_\tau^+$ of $k$-schemes is an isomorphism. 

For a closed subscheme $Y$ of an arbitrary toric $k$-variety $X(\Delta)$, we define the affine presentation $\cV_k(\Delta)$ as the diagram of morphisms $Z_\sigma\to Z_\tau$ whenever $\sigma\subset\tau$ and we define the blue $k$-scheme $Z$ as $\colim\cV_k(\Delta)$. The morphisms $\beta_\tau$ glue to a morphism $\beta:Y\to Z$, which induces an isomorphism $\beta^+:Y\to Z^+$ of $k$-schemes. We say that $Z$ is the \emph{blue model of $Y$ induced by $\iota$}.

The affine presentation $\cV_k(\Delta)$ gives also rise to a blue model $Z^{+\blue}$ of $Y$ that satisfies the hypotheses of Theorem \ref{thm: Berkovich analytification as rational point set}. Namely, we define $Z^{+\blue}$ as the colimit of $\cV_k(\Delta)^{+\blue}$, which is the affine presentation that consists of the affine ordered blue schemes $(U^+)^\blue$ where $U$ ranges through $\cV_k(\Delta)$.

\begin{thm}\label{thm: Kajiwara-Payne tropicalization as rational point set}
 The Kajiwara-Payne tropicalization $\Trop_{v,\iota}^{KP}(Y)$ is naturally homeomorphic to $\Bend_v(Z)(\T)$ and the diagram
 \[
  \xymatrix@R=1pc@C=6pc{Y^\an \ar[r]^{\trop_{v,\Delta}^{KP}} \ar[d]^\simeq & \Trop_{v,\iota}^{KP}(Y) \ar[d]^\simeq \\ \Bend_v(Z^{+\blue})(\T) \ar[r]^{\Bend_v(\beta)(\T)} & \Bend_v(Z)(\T)}
 \]
 of continuous maps commutes.
\end{thm}

\begin{proof}
 Since all functors in questions are defined in terms of affine coverings, it is enough to prove the theorem in the affine case. Since the set $X_\tau^\trop=\Hom(A_\tau,\T)$ of monoid morphisms stays in natural bijection with the set $\Val_v(k[A_\tau],\T)$ of valuations on $k[A_\tau]$ in $\T$ that extend $v$, the image $\Trop_{v,\iota}^{KP}(Y)$ of $Y^\an=\Val_v(k[A_\tau]^+/I,\T)$ under $\Trop_{v,\tau}^{KP}$ equals $\Val_v(\bpquot{k[A_\tau]}{\cR_\tau},\T)$, which stays in natural bijection with $\Bend_v(Z)(\T)$ by Theorem \ref{thm: tropicalization for idempotent base}.
 
 Note that the commutativity of the diagram follows from the functoriality of the bend functor. Thus we are left with showing that the bijection $\Trop_{v,\iota}^{KP}(Y)\to\Bend_v(Z)(\T)$ is a homeomorphism. By Theorem \ref{thm: Giansiracusa tropicalization as bend} and Example \ref{ex: real topology of the tropical numbers} and the local definition of the topology of $\Trop_{v,\iota}^{KP}(Y)$, this can be verified on affine patches.
 
 Since $\Bend_v(Z)(\T)$ inherits the subspace topology from $\Hom_\T(\Bend_v(k[A_\tau]),\T)$ and since $\Trop_{v,\iota}^{KP}(Y)$ is defined as a topological subspace of $X_\tau^\trop=\Hom(A_\tau,\T)$, it suffices to show that the natural bijection 
 \[
  \Psi: \ \Hom(A_\tau,\T) \quad \longrightarrow \quad \Hom_\T(\Bend_v(k[A_\tau]),\T)
 \]
 is a homeomorphism. Since $\Bend_v(k[A_\tau])=\bpquot{k[A_\tau]^\bullet\otimes_{k^\bullet}\T}{\bend_v(k[A_\tau])}$, the image of a basic open of the former morphism set is
 \[
  \Psi(U_{a,V}) \quad = \quad \bigl\{ \ f: k[A_\tau]^\bullet\otimes_{k^\bullet}\T \to\T \ \bigl| \ f(a\otimes 1)\in V \ \bigr\} \quad = \quad U_{a\otimes 1,V}
 \]
 where $a\in A_\tau$ and $V\subset \T$ open, and $U_{a\otimes 1,V}$ is a basic open in $\Hom_\T(\Bend_v(k[A_\tau]),\T)$. This shows that $\Psi$ is an open map. 
 
 Conversely, consider a basic open $U_{ca\otimes t,V}$ of $\Hom_\T(\Bend_v(k[A_\tau]),\T)$ where $c\in k$, $a\in A_\tau$, $t\in \T$ and $V\subset \T$ is an open subset. Denote by $m_{v(c)t}:\T\to\T$ the multiplication by $v(c)t$, which is a continuous map. Then we have
 \begin{align*}
  U_{ca\otimes t,V} \quad &= \quad \big\{ \, f:\Bend_v(k[A_\tau])\to \T \, \big| \, f(ca\otimes t)\in V \, \big\} \quad \\
                          &= \quad \big\{ \, f:\Bend_v(k[A_\tau])\to \T \, \big| \, f(a\otimes 1)\in m_{v(c)t}^{-1}(V) \, \big\} \quad = \quad U_{a\otimes 1,m_{v(c)t}^{-1}(V)}
 \end{align*}
 and thus $\Psi^{-1}(U_{ca\otimes t,V})= U_{a,m_{v(c)t}^{-1}(V)}$, which is a basic open of $\Hom(A_\tau,\T)$. This shows that $\Psi$ is continuous and completes the proof of the theorem.
\end{proof}

\begin{ex}\label{ex: Kajiwara-Payne tropicalization as bend}
 We illustrate the Kajiwara-Payne tropicalization of the line $T_1+T_2+1=0$ in $\A_k^2$ for an algebraically closed field $k$ with trivial absolute value $v:k\to\T$; also see \cite[Ex.\ 3.1, 3.4, 3.7]{Lorscheid22} for further details on this example. 
 
 The ambient toric variety is $\A_k^2=\Spec k[A]^+$ for the monoid $A=\{T_1^iT_2^j\mid i,j\in\}$ and the closed subscheme in question is $Y=\Spec k[A]^+/(T_1+T_2+1)$, embedded via $\iota:Y\to\A^2_k$. See Figure \ref{figure: KP tropicalization} for an illustration of the morphism $Y^\an\to Y^\trop$ from the Berkovich space $Y^\an$ of $Y$ to the Kajiwara-Payne tropicalization $Y^\trop=\Trop_{v,\iota}^{KP}(Y)$. The Berkovich space $Y^\an$ is a star whose rays are labelled by the elements in $Y(k)\cup\{\infty\}$; the rays corresponding to elements of $Y(k)$ are closed (its endpoint corresponds to the seminorm that factors through the corresponding residue field), the ray at infinity is open. The map to the tropicalization $Y^\trop$ contracts all rays to the central point, but for the ray to infinity and the two rays corresponding to the points $(1,0)$ and $(0,1)$ of $Y(k)\subset k^2$, which are mapped one-to-one to $Y^\trop$.
 \begin{figure}[h]
  \centering
  \begin{tikzpicture}[inner sep=0,x=25pt,y=25pt,font=\tiny]
   \node (origin) {};
   \foreach \a in {1,...,32}{\draw[gray] (0,0) -- ++(\a*360/32:2);}
   \draw[very thick,black] (0,0) -- ++(45:2);
   \draw[very thick,black] (0,0) -- ++(270:2);
   \draw[very thick,black] (0,0) -- ++(180:2);
   \filldraw (0,0) circle (2pt);
   \filldraw (180:2) circle (2pt);
   \filldraw (270:2) circle (2pt);
   \draw[fill=white] (45:2) circle (2pt);
   \draw[very thick,->] (3,0) -- (6,0);
   \node[font=\normalsize] at (4.5,0.3) {$\trop$};
   \node[font=\normalsize] at (2.5,1.6) {$Y^\an$};
   \node[font=\normalsize] at (7,1.6) {$Y^\trop$};
   \draw[very thick] (9,0) -- ++(45:2);
   \draw[very thick] (9,0) -- ++(180:2);
   \draw[very thick] (9,0) -- ++(270:2);
   \filldraw (9,0) circle (2pt);
   \filldraw (9,0)++(180:2) circle (2pt);
   \filldraw (9,0)++(270:2) circle (2pt);
   \draw[fill=white] (9,0)++(45:2) circle (2pt);
  \end{tikzpicture}
  \caption{The Kajiwara-Payne tropicalization of the line $T_1+T_2+1=0$ in $\A^2_k$}
  \label{figure: KP tropicalization}
 \end{figure}
 
 The inclusion $A\to k[A]^+$ singles out the blue model $B=\bpquot{k[A]}{\cR}$ with $\cR=\gen{T_1+T_2+1\=0}$. Using the identification $B[A]^\bullet\otimes_{k^\bullet}\T=\T[A]$, the bend of $B$ along $v$ is $\Bend_v(B)=\bpquot{\T[A]}{\bend_v(B)}$ for the bend relation
 \[\textstyle
  \bend_v(B) \ = \ \big\{ a+\sum b_j\=\sum b_j \,\big|\, a\=\sum b_j\text{ in }\cR \big\}.
 \]
 So $\bend_v(B)$ contains the relations
 \[
  T_1+T_2+1 \ = \ T_2+1 \ = \ T_1+1 \ = \ T_1+T_2,
 \]
 which determine the tropical variety $Y^\trop$ as the subset $\{(a_1,a_2)\in\T^2\mid a_2+1=a_1+1=a_1+a_2\}$ of $T^2$. Note however that the bend relation is not generated by these relations. For instance, $\bend_v(B)$ contains the relation
 \[
  T_1^2+T_1T_2+T_2+1 \ = \ T_1T_2+T_2+1
 \]
 since $T_1^2+T_1T_2+T_2+1$ is the tropicalization of the polynomial $T_1(T_1+T_2+1)-(T_1+T_2+1)$ in the ideal $\gen{T_1+T_2+1}$ of definition of $Y$.
\end{ex}

%%%%%%%%%%%%%%%%%%%%%%%%%%%%%%%%%%%%%%%%%%%%%%%%%%%%%%%%%%%%%%%%%%%%%%%%%%%%%%%%%%%%%%%%%%%%%%%%%%%%%%%%%%%%%%%%%%%%%%%%%%%%%%%%%%%%%%%%%%%%%%%%%%%%%%%%%%%%%%%%%%%%%%%%%%%
%%%%%%%%%%%%%%%%%%%%%%%%%%%%%%%%%%%%%%%%%%%%%%%%%%%%%%%%%%%%%%%%%%%%%%%%%%%%%%%%%%%%%%%%%%%%%%%%%%%%%%%%%%%%%%%%%%%%%%%%%%%%%%%%%%%%%%%%%%%%%%%%%%%%%%%%%%%%%%%%%%%%%%%%%%%

\section{Foster-Ranganathan tropicalization}
\label{section: Foster-Ranganathan tropicalization}

Motivated by the paper \cite{Banerjee15} of Banerjee, Foster and Ranganathan generalize in \cite{Foster-Ranganathan15} the Kajiwara-Payne tropicalization to higher rank valuations of the ground field $k$. In this section, we will show how the Foster-Ranganathan tropicalization fits into the context of scheme theoretic tropicalization. We recall the setup of \cite{Foster-Ranganathan15}.

We denote by $\T^{(n)}$ the idempotent semiring $\R_{>0}^n\cup\{0\}$ with componentwise multiplication and whose addition addition is defined as taking the maximum with respect to the lexicographical order. The order topology turns $\T^{(n)}$ into a topological Hausdorff semifield. In particular, $\T^{(n)}$ satisfies all the hypotheses of Theorem \ref{thm: properties of the fine topology}.

Let $k$ be a field with valuation $v:k\to\T^{(n)}$ and $Y=\Spec B$ an affine $k$-scheme. The \emph{Foster-Ranganathan analytification of $Y$ along $v$} is the set
\[
 \An^{FR}_v(Y) \ = \ \{ \, w: B\to \T^{(n)} \, | \, w\vert_k=v_k \, \}
\]
of all valuations $w$ that extend $v$ to $B$. It is endowed with the compact-open topology where $B$ is considered as a discrete $k$-algebra.

Let $A$ be a commutative monoid with zero and $\pi:k[A]^+\to B$ a surjection of $k$-algebras. This corresponds to a closed immersion $\iota:Y\to X$ of $k$-schemes where $X=\Spec k[A]^+$. For instance, if $A-\{0\}$ is a finitely generated abelian group, then $X$ is a split $k$-torus. The \emph{Foster-Ranganathan tropicalization of $Y$ along $v$ with respect to $\iota$} is the image $\Trop^{FR}_{v,\iota}(Y)$ of the map
\[
 \trop^{FR}_{v,\iota}: \An_v^{FR}(Y) \ \longrightarrow \ \Hom(A,\T^{(n)})
\]
that sends a valuation $w:B\to \T^{(n)}$ to the composition $A\to k[A]^+\to B\to \T^{(n)}$. We endow $\Hom(A,\T^{(n)})$ with the compact-open topology, with respect to the discrete topology for $A$, and $\Trop^{FR}_{v,\iota}(Y)$ with the subspace topology.

Let $Z=\Spec \bpquot{k[A]}\cR$ be the blue $k$-scheme associated with $\iota:Y\to X$ where $\cR=\{\sum a_i\=\sum b_j|\sum \pi(a_i)=\sum\pi(B_j)\}$, and $\beta:Y\to Z$ the morphism induced by the inclusion $\bpquot{k[A]}\cR\to B$. Let $Z^{+\blue}$ be the affine blue scheme $\Spec B$ where we consider $B$ as a blueprint.

\begin{thm}\label{thm: Foster-Ranganathan analytification as rational point set}
 The Foster-Ranganathan analytification $\An^{FR}_v(Y)$ is naturally homeomorphic to $\Bend_v(Y)(\T^{(n)})$, the Foster-Ranganathan tropicalization $\Trop^{FR}_{v,\iota}(Y)$ is naturally homeomorphic to $\Bend_v(Z)(\T^{(n)})$ and the diagram
 \[
   \xymatrix@R=1pc@C=6pc{\An_v^{FR}(Y) \ar[r]^{\trop_{v,\iota}^{FR}} \ar[d]^\simeq & \Trop_{v,\iota}^{FR}(Y) \ar[d]^\simeq \\ \Bend_v(Z^{+\blue})(\T^{(n)}) \ar[r]^{\Bend_v(\beta)(\T^{(n)})} & \Bend_v(Z)(\T^{(n)})}
 \]
 of continuous maps commutes.
\end{thm}

\begin{proof}
 The proofs of Theorems \ref{thm: Berkovich analytification as rational point set} and \ref{thm: Kajiwara-Payne tropicalization as rational point set} apply verbatim with $\T$ exchanged by $\T^{(n)}$.
\end{proof}

\begin{rem}
 Foster and Ranganathan consider in \cite{Foster-Ranganathan15} also the topology for $\T^{(n)}$ induced from the natural embedding into the Euclidean space $\R^n$. Since $\T^{(n)}$ is also a topological Hausdorff semiring with respect to the Euclidean topology, the analogue of Theorem \ref{thm: Foster-Ranganathan analytification as rational point set} also holds for the Euclidean topology.
\end{rem}

\begin{rem}
 Note that the scheme theoretic Foster-Ranganathan analytification is compatible with the order preserving projections $\T^{(n)}\to \T^{(l)}$ for $l<n$, as considered in section 2.1 of \cite{Foster-Ranganathan15}, in the following sense.
 
 A rank $n$ valuation $v_n:k\to\T^{(n)}$ of the field $k$ induces a rank $l$ valuation $v_j:k\to\T^{(l)}$ for every $l<n$ by composing with the projection $\pi_{n,l}:\T^{(n)}\to \T^{(l)}$ onto the first $l$ factors. This map induces a continuous map $\pi_{n,l}^{FR}(Y):\An_{v_n}^{FR}(Y) \to \An_{v_{l}}^{FR}(Y)$ of Foster-Ranganathan analytifications.
 
 This continuous map can be recovered by applying $\Phi_Y=\Hom_{\T^{(n)}}(\Spec(-),\Bend_{v_n}(Z^{+\blue}))$ to the morphism $\pi_{n,l}:\T^{(n)}\to\T^{(l)}$. In other words,
 \[
  \xymatrix@R=1pc@C=6pc{\An_{v_n}^{FR}(Y) \ar[r]^{\pi_{n,l}^{FR}(Y)} \ar[d]^\sim & \An_{v_{l}}^{FR}(Y)\ar[d]^\sim \\ \Bend_{v_n}(Z^{+\blue})(\T^{(n)}) \ar[r]^{\Phi_Y(\pi_{n,l})} &\Bend_{v_l}(Z^{+\blue})(\T^{(l)}) }
 \]
 is a commutative diagram of continuous maps where we use that $\Bend_{v_n}(Y)\otimes_{\T^{(n)}}\T^{(l)}$ is isomorphic to $\Bend_{v_l}(Z^{+\blue})$.
 
 There is a similar commutative diagram for the projections $\Trop_{v_n,\iota}^{FR}(Y) \to \Trop_{v_{l},\iota}^{FR}(Y)$ of Foster-Ranganathan tropicalizations, which play a prominent role in a forthcoming paper by Foster and Hully.
\end{rem}

\subsection{What is new?}
As a consequence, we can extend the Foster-Ranganathan tropicalization beyond the affine case. For instance, a closed immersion $\iota:Y\to X(\Delta)$ of $Y$ into a toric variety $X(\Delta)$ yields an associated blue $k$-model $Z$ of $Y$. We define the corresponding \emph{Foster-Ranganathan tropicalization $\Trop_{v,\iota}^{FR}(Y)$} as the topological space $\Bend_v(Z)(\T^{(n)})$. By Theorem \ref{thm: properties of the fine topology}, the affine open subschemes $Y_\sigma=Y\cap X_\sigma$ for $\sigma\in \Delta$ yield open topological embeddings $\Trop_{v,\iota}^{FR}(Y_\sigma)\to\Trop_{v,\iota}^{FR}(Y)$, which cover $\Trop_{v,\iota}^{FR}(Y)$.

More generally, the Foster-Ranganathan tropicalization can be extended to toroidal embeddings and log structures via their associated blue schemes, as considered in section \ref{section: Ulirsch tropicalization}.

%%%%%%%%%%%%%%%%%%%%%%%%%%%%%%%%%%%%%%%%%%%%%%%%%%%%%%%%%%%%%%%%%%%%%%%%%%%%%%%%%%%%%%%%%%%%%%%%%%%%%%%%%%%%%%%%%%%%%%%%%%%%%%%%%%%%%%%%%%%%%%%%%%%%%%%%%%%%%%%%%%%%%%%%%%%

\section{Giansiracusa tropicalization}
\label{section: Giansiracusa tropicalization}

Jeff and Noah Giansiracusa introduce in \cite{Giansiracusa13} the bend relation for a closed subscheme of a toric variety or, more generally, of a monoid scheme over a non-archimedean field $k$. We recall this theory and explain how to recover the Giansiracusa bend relation from our point of view.

\subsection{The bend functor for morphisms into monoid schemes}\label{subsection: bend functor for morphisms into monoid schemes}
Let $k\to R$ be a ring homomorphism and $v:k\to T$ a valuation into an idempotent semiring $T$ that satisfies $v(-1)=1$. Let $A_0$ be a monoid with zero and $\eta:A_0\to R$ a multiplicative map such that the induced homomorphism $\eta^+_k:k[A_0]^+\to R$ of $k$-algebras is surjective. Let $I=(\eta_k^+)^{-1}(0)$ be the kernel. The \emph{Giansiracusa tropicalization of $R$ with respect to $v$ and $\eta$} is the semiring 
\[
 \Trop_{v,\eta}^{GG}(R) \quad = \quad T[A_0]^+ \, \sslash \, \bend_{v,\eta}^{GG}(I)
\]
where the \emph{Giansiracusa bend relation} $\bend_{v,\eta}^{GG}(I)$ is generated by the relations 
\[\textstyle
 v(c_a)a \ + \ \sum v(c_j)b_j \quad \= \quad \sum v(c_j)b_j \qquad \text{for which} \qquad c_aa \ + \ \sum c_jb_j \ \in \  I
\]
with $c_a,c_j\in k$ and $a,b_j\in A_0$.

\begin{rem}
 Note that the congruence $\bend_{v,\eta}^{GG}(I)$ is in fact already generated by those relations of the above form for which $a$ and the $b_j$ are pairwise different. This observation explains that our definition of $\Trop_{v,\eta}^{GG}(R)$ coincides with the original definition in \cite{Giansiracusa13}.
\end{rem}

For integral monoids $A$, \cite{Giansiracusa13} shows that the bend relations are compatible with localizations. Therefore the Giansiracusa tropicalization can be extended to $k$-schemes $Y$ with respect to a closed immersion $\iota:Y\to X^+_k$ where $X^+_k$ is the $k$-scheme associated with an integral monoid scheme $X$. 

Choosing compatible affine presentations of $Y$ and $X$ and applying $\Trop_{v,\eta}^{GG}$, where $\eta$ stands for the map between the coordinate blueprints of objects of the chosen affine presentations, yields the Giansiracusa tropicalization $\Trop_{v,\iota}^{GG}(Y)$ of $Y$ with respect to $v$ and $\iota$. In more detail, this works as follows.

Let $\cU$ be an affine presentation of $X$, e.g.\ the family of all affine open subschemes of $X$ together with the inclusion morphisms. Then $\cU^+_k$ is an affine presentation of $X^+_k$, which consists of all open subschemes of the form $U^+_k$ of $X^+_k$ where $U$ is an affine open in $X$, and the system $\cV=\iota^{-1}\cU^+_k$ of all affine opens $V_U=\iota^{-1}(U^+_k)$ together with the induced morphisms from $\cU^+_k$ is an affine presentation of $Y$.

For every affine open $U$ of $X$ that is in $\cU$, we obtain a map $\eta_U:\Gamma U\to \Gamma U^+_k\to\Gamma V_U$, which yields a family of affine schemes $\Trop_{v,\eta_U}^{GG}(\Gamma V_U)$. The open immersions $U\to U'$ in $\cU$ yield open immersions $\Trop_{v,\eta_U}^{GG}(\Gamma V_U)\to\Trop_{v,\eta_{U'}}^{GG}(\Gamma V_{U'})$ since all functors involved commute with finite localization. This yields an affine presentation $\Trop_{v,\iota}^{GG}(\cV)$ of affine ordered blue schemes. 

Let $\Trop_{v,\iota}^\blue(Y)$ be the colimit of $\Trop_{v,\iota}^{GG}(\cV)$ in $\OBSch$. The \emph{Giansiracusa tropicalization of $V$ along $v:k\to T$ with respect to $\iota:V\to X^+_k$} is the associated semiring scheme $\Trop_{v,\iota}^{GG}(V)=\Trop_{v,\iota}^\blue(Y)^+$.

\subsection{The associated blue scheme}\label{subsection: blue scheme associated with a toric embedding}
The connection with the bend functor from section \ref{subsection: The bend functor} is as follows. A surjection $\eta:A_0\to k[A_0]^+ \to R$ yields the blueprint $B=\bpquot{A}{\cR}$ with $A=k[A_0]$ and
\[\textstyle
 \cR \quad = \quad \left\{ \ \sum a_i \= \sum b_j  \ \left| \ \sum\eta^+(a_i)=\sum\eta^+(b_j)\text{ in } R \ \right.\right\},
\]
together with the morphism $B\to R$, which induces an isomorphism $B^+\simeq R$. A closed immersion $\iota:Y\to X^+_k$ where $X$ is a monoid scheme yields a blue $k$-scheme $Z$ by choosing compatible affine presentations for $Y$ and $X$ and applying the above definition to the coordinate blueprints.

\begin{thm} \label{thm: Giansiracusa tropicalization as bend}
 There is a canonical morphism $\Bend_v(Z)\to \Trop_{v,\iota}^{\blue}(Y)$ of ordered blue schemes that induces an isomorphism $\Bend_v(Z)^+\simeq \Trop_{v,\iota}^{GG}(Y)$ of semiring schemes.
\end{thm}

\begin{proof}
 Since the definition of $\Trop_{v,\iota}^{\blue}(Y)$ in terms of affine presentations is compatible with the definition of $\Bend_v(Z)$, we are reduced to the affine case of a multiplicative morphism $\eta:A_0\to R$ from a monoid $A_0$ with zero into a $k$-algebra $R$ that yields a surjection $\eta^+_k:k[A_0]^+\to R$. Let $B$ be the associated blueprint. We have to show that there is a canonical morphism $\Bend_v(B)\to \Trop_{v,\eta}^{GG}(R)$ that induces an isomorphism $\Bend_v(B)^+\simeq \Trop_{v,\eta}^{GG}(R)$ of semirings.
 
 The association $c\cdot a\otimes t\mapsto v(c)t\cdot a$ for $c\in k$, $a\in A_0$ and $t\in T$ defines an isomorphism $\varphi:k[A_0]^\bullet\otimes_{k^\bullet}T\to T[A_0]$ of ordered blue $T$-algebras. If we can show that $\varphi$ identifies the relation $\bend_v(B)$ on $(k[A_0]^\bullet\otimes_{k^\bullet}T)^+$ with the Giansiracusa bend relation $\bend_{v,\eta}^{GG}(I)$ on $T[A_0]^+$, then it follows that $\varphi$ induces an isomorphism $\Bend_v(B)\to \bpquot{T[A_0]}{\bend_{v,\eta}^{GG}(I)}$, and the lemma follows after applying $(-)^+$.
 
% For the identification between the bend relations, we use the hypothesis that $v(-1)=1$ and thus $v(-c)=v(c)$ for all $c\in k$. 
 
 We show that $c_aa\otimes 1+\sum c_jb_j\otimes 1 \=\sum c_jb_j\otimes 1$ is a generator of $\bend_v(B)$ if and only if $v(c_a)a+\sum v(c_j)b_j\=\sum v(c_j)b_j$ is a generator of $\bend_{v,\eta}^{GG}(I)$. Indeed, the former relation is in $\bend_v(B)$ if and only if $c_aa\leq \sum c_jb_j$ in $B$. Since $B$ is algebraic and with $-1$, this is equivalent to $c_aa+\sum (-c_j)b_j\=0$, which, in turn, is equivalent to $c_aa+\sum (-c_j)b_j\in I$. By our assumptions on $v$, we have $v(-1)=1$ and therefore $v(-c)=v(c)$ for $c\in k$. Thus the latter condition is equivalent to $v(c_a)a+\sum v(c_j)b_j\=\sum v(c_j)b_j$ in $\Bend_{v,\eta}^{GG}(R)$, as claimed.
\end{proof}

As a consequence, we recover the central result of \cite{Giansiracusa14} about the universal tropicalization of an affine $k$-scheme along a valuation $v:k\to T$ with $v(-1)=1$. A concise formulation of this result requires the following variant of the functor $\Val_v(B,-)$.

Let $k$ be a ring and $\Alg_k^+$ be the category of $k$-algebras, whose objects are homomorphisms $k\to R$ of rings. It embeds as a full subcategory into the category $\Alg_k^\ob$ of ordered blue $k$-algebras. We denote this embedding by $(-)^\blue:\Alg_k^+\to\Alg_k^\ob$. For an affine $k$-scheme $Y=\Spec R$, let $\Val_v^+(Y,-):\Alg_k^+\to\Sets$ be the restriction of $\Val_v(R,-):\Alg_k^\ob\to\Sets$ to this subcategory. We denote by $Y^\blue=\Spec R^\blue$ the spectrum of the blueprint $R^\blue$ and call the semiring $T$-scheme $\Trop_{v,\iota}^{GG}(Y^\blue)^+$ the \emph{universal Giansiracusa tropicalization of $Y$ along $v$}.

\begin{cor}\label{cor: the universal tropicalization}
 Let $v:k\to T$ be a valuation from a ring $k$ into a totally ordered idempotent semiring $T$ and $Y$ an affine $k$-scheme. Then $\Val_v^+(Y,-)$ is represented by the universal Giansiracusa tropicalization $\Trop_{v,\iota}^{GG}(Y^\blue)^+$ of $Y$ along $v$.
\end{cor}

\begin{proof}
 By Theorem \ref{thm: tropicalization for idempotent base}, we have $\Val_v(Y^\blue,-)=\Hom_T(\Spec(-),\Bend_v(Y^\blue))$. Since every morphism $\Spec S\to \Bend_v(Z)$ factors uniquely through a $T$-morphism $\Spec S\to \Bend_v(Z)^+$, the functor $\Val_v^+(Y,-)$ is isomorphic to the functor $\Hom_T(\Spec(-),\Bend_v(Z)^+)$. Therefore the corollary follows from Theorem \ref{thm: Giansiracusa tropicalization as bend}.
\end{proof}

\begin{rem}
 Note that the definition of $Y^\blue$ makes sense for affine $k$-schemes, but that it is not clear how to extend this to a functor from $\Sch_k^+$ to $\OBSch_k$; also cf.\ the paragraph ``Differences to previous versions'' of the introduction. Therefore the statement of Corollary \ref{cor: the universal tropicalization} makes only sense for affine $k$-schemes. 
\end{rem}

\begin{ex}\label{ex: Giansiracusa tropicalization}
 We revisit Example \ref{ex: Kajiwara-Payne tropicalization as bend} to illustrate the Giansiracusa tropicalization and its relation to the bend of the associated blue scheme. Let $k$ be a field with non-archimedean valuation $v:k\to\T$. Let $Y=\Spec \big(k[X_1,X_2]^+/I\big)$ be the closed subscheme of $\A_{k}^{2,+}$ that is defined by the ideal $I=\gen{X_1+X_2+1}$. Then the Giansiracusa bend 
 \[\textstyle
  \bend_{v,\eta}^{GG}(I) \ = \ \big\langle v(c_a)a+\sum v(c_j)b_j\=\sum v(c_j)b_j \, \big| \, c_aa+\sum c_jb_j\in I \big\rangle
 \]
 of $I$ contains
 \[
  X_1+X_2+1 \ \= \ X_2+1 \ \= \ X_1+1 \ \= \ X_1+X_2.
 \]
 Note that $\bend_{v,\eta}^{GG}(I)$ is not generated by these relations stemming from the generator $T_1+T_2+1$ of $I$ as witnessed by the relation $T_1^2+T_1T_2+T_2+1\=T_1^2+T_2+1$, even though $T_1+T_2+1$ is a tropical basis for the Kajiwara-Payne tropicalization of $Y$. This behaviour is typical if $I$ is generated by polynomials with more than $3$ terms; cf.\ Example \ref{ex: Kajiwara-Payne tropicalization as bend} and \cite[section 8.1]{Giansiracusa13}.

 The associated blue scheme of $Y$ is
 \[
  Z \ = \ \Spec\big(\bpquot{k[X_1,X_2]}{\gen{X_1+X_2+1\=0}}\big),
 \]
 and its bend is defined by the same relations as they occur in the Giansiracusa bend, which yields
 \[
  \Bend_v(Z) \ = \ \Spec\big(\bpquot{\T[X_1,X_2]}{\bend_v(Z)}\big).
 \]
\end{ex}

\subsection{What is new?}
To speak in an analogy, the difference between the Giansiracusa tropicalization and the scheme theoretic tropicalization in terms of blue schemes is similar to the difference between subvarieties of a projective space and projective varieties. In other words, the enhancement of a variety $Y$ in an ambient monoid scheme with the structure of a blue scheme allows us to detach $Y$ from the monoid scheme and tropicalize it as an independent abstract geometric object. Besides this conceptual novelty, we have eliminated the following two technical restrictions of \cite{Giansiracusa13}.

Theorem \ref{thm: Giansiracusa tropicalization as bend} guarantees that the bend functor from this paper incorporates the Giansiracusa tropicalization completely. Therefore we can extend the Giansiracusa tropicalization to morphism $\iota:Y\to X$ into monoid schemes that are not necessarily integral. This can also be seen by directly generalizing \cite{Giansiracusa13}.
 
Moreover, our approach generalizes the Giansiracusa tropicalization to valuations $v:k\to T$ for which $v(-1)$ differs from $1$ \emph{in a functorial way}; cf.\ \cite[Rem.\ 6.4.2]{Giansiracusa13}.

%Note that for this generalization, it is important to adapt the sign convention for the bend relation of this text since $v(c)=v(-c)$ does not hold true for all valuations into idempotent semirings.

%%%%%%%%%%%%%%%%%%%%%%%%%%%%%%%%%%%%%%%%%%%%%%%%%%%%%%%%%%%%%%%%%%%%%%%%%%%%%%%%%%%%%%%%%%%%%%%%%%%%%%%%%%%%%%%%%%%%%%%%%%%%%%%%%%%%%%%%%%%%%%%%%%%%%%%%%%%%%%%%%%%%%%%%%%%

\section{Maclagan-Rinc\'on weights}
\label{section: Maclagan-Rincon weights}

Maclagan and Rinc\'on show in \cite{Maclagan-Rincon14} that the weights of the tropicalization of a closed subvariety of a torus can be recovered from the Giansiracusa tropicalization. We will extend the  argument of \cite{Maclagan-Rincon14} to the scheme theoretic tropicalizations considered in this paper.

\subsection{Weights from the classical variety} Let $k$ be a field with valuation $v:k\to \T$ and value group $\Gamma=v(k^\times)$. Assume that $\Gamma$ is dense in $\T$ and that there exists a section $s:\Gamma\to k^\times$ to $v:k^\times\to \Gamma$ as a group homomorphism. Note that such a section always exists after passing to a suitable finite field extension of $k$. 

In this situation, the tropicalization $\Trop(Y)=\Trop_{v,\iota}^{KP}(Y)$ of a closed $k$-subscheme $Y$ of a split torus $\G_{m,k}^{n,+}$ can be endowed with the structure of a polyhedral complex whose top dimensional cells come with weights that satisfy a certain balancing condition with respect to the embedding $\Trop(Y)\subset (\T^\ast)^n\simeq \R^n$ where we identify $(\T^\times)^n$ with $\R^n$ by taking coordinatewise logarithms. Note our specific usage of $\G_{m,k}=\Spec k[X^{\pm1}]$ and $\G_{m,k}^+=\Spec k[X^{\pm1}]^+$, cf.\ section \ref{section: conventions}. In the following, we recall the definition of the weights of $\Trop(Y)$. 

Let $\cO_k=\{a\in k|v(a)\leq 1\}$ be the \emph{integers of $k$}, $\fm=\{a\in\cO_k|v(a)<1\}$ its unique maximal ideal and $k_0=\cO_k/\fm$ the \emph{residue field}. Let $\iota: Y\to\G_{m,k}^{n,+}$ be a closed immersion of $k$-schemes and $I\subset R$ the defining ideal where $R_n^+=k[X_1^{\pm1},\dotsc,X_n^{\pm1}]^+$ is the coordinate ring of $\G_{m,k}^{n,+}$. 

Let $\Trop(Y)\subset (\T^\times)^n$ be the set-theoretic tropicalization of $Y$. We say that a point $w\in\Trop(Y)$ is \emph{smooth} if it has a neighbourhood $U$ in $\Trop(Y)$ that is a smooth submanifold of $(\T^\times)^n\simeq\R^n$. Note that the set of smooth points of $\Trop(Y)$ is the compliment of the corner locus of $\Trop(Y)$.

Let $w=(w_i)$ be a smooth point of $\Trop(Y)$. For an element $f=\sum c_e X^e$ of $I$ where $e\in\Z^n$ is a multi-index and $c_e\in k$, we define its \emph{tropicalization} as $v(f)=\sum v(c_e)X^e$ and its \emph{value at $w$} as $v(f)(w)=\sum v(c_e)w^e$. We denote by $\overline a\in k_0$ the residue class of an element $a\in\cO_k$. The \emph{initial form of $f$ in $w$} is 
\[
 \init_w(f) \quad =  \sum_{v(c_e)w^e=v(f)(w)} \overline{s\bigl(v(c_e)^{-1}\bigr) \ v(c_e)} \ X^e,
\]
which is an element of $\overline R_n^+=k_0[X_1^{\pm1},\dotsc,X_n^{\pm1}]^+$. The \emph{initial ideal of $I$ in $w$} is the ideal $\init_w(I)$ of $\overline R_n^+$ that is generated by the initial forms $\init_w(f)$ of all $f\in I$. Let $\init_w(T)=\bigcap \fq_i$ be a primary decomposition for $I$ and $\fp_i$ the radical of $\fq_i$. We denote by $\mult(\fp_i,\init_w(I))$ the length of the $\overline R_n^+$-module $(\overline R_n^+/\fq_i)_{\fp_i}$. The \emph{multiplicity of $\Trop_{v,\iota}^{KP}(Y)$ in $w$} is defined as
\[
 \mult(w) \quad = \quad \sum \mult(\fp_i,\init_w(I))
\]
where the $\fp_i$ vary through all minimal associated primes of $\init_w(I)$. Note that this multiplicity does neither depend on the choice of section $s:\Gamma\to k^\times$ nor on the choice of primary decomposition $\init_w(I)=\bigcap \fq_i$.

The structure theorem for tropical varieties asserts the following. For details, cf.\ \cite[Thm.\ 3.3.6]{Maclagan-Sturmfels15}.

\begin{thm}
 The tropicalization of a purely $d$-dimensional $k$-subscheme $Y$ of $\G_{m,k}^n$ can be endowed with the structure of a balanced weighted polyhedral complex of dimension $d$ such that the weight of its $d$-dimensional polyhedra $\sigma$ equals $\mult(w)$ for each $w$ in the relative interior of $\sigma$. 
\end{thm}

Note that every point $w$ in the relative interior of a top-dimensional polyhedron is smooth.

\subsection{Weights from the scheme theoretic tropicalization}
Let $Y$ be a purely $d$-dimensio\-nal $k$-subscheme of $\G_{m,k}^{n,+}$ and $Z$ the associated blue $k$-scheme. We will show in the following that the weights $\mult(w)$ can be determined from the scheme theoretic tropicalization $\Bend_v(Z)$.

For a more readable notation, we define 
\begin{align*}
           R_n &= k[X_1^{\pm1},\dotsc,X_n^{\pm1}],   &           S_n &= \T[X_1^{\pm1},\dotsc,X_n^{\pm1}],\\
 \overline R_n &= k_0[X_1^{\pm1},\dotsc,X_n^{\pm1}], & \overline S_n &= \B[X_1^{\pm1},\dotsc,X_n^{\pm1}].
\end{align*}

Let $I\subset R_n^+$ be the ideal of definition of $Y$. Then $Z$ is the spectrum of $B=\bpquot{R_n}{\cR_I}$ where $\cR_I=\{\sum a_i\=\sum b_j|\sum a_i-\sum b_j\in I\}$, and $\Bend_v(Z)$ is the spectrum of $\bpquot{B^\bullet\otimes_{k^\bullet}\T}{\bend_v(B)}$.

By the definition of $B$, we can choose a surjection $\pi:S_r\to \Bend_v(B)$, which yields a presentation of the form 
\[\textstyle
 \Bend_v(B)=\bpquot{S_r}{\cR_\pi} \qquad \text{with} \qquad \cR_\pi=\left\{\, \sum a_i\=\sum b_j \, \left| \, \sum\pi(a_i)\=\sum\pi(b_j) \, \right.\right\}.
\]
This corresponds to a closed immersion $\Bend_v(Z)\to \G_{m,\T}^r$. Note that we consider $\Bend_v(Z)$ as an abstract $\T$-scheme; therefore we allow $\pi$ to be any surjection and $r$ to differ from $n$.

Consider a smooth point $w\in\Bend_v(Z)(\T)$, which is a $\T$-morphism $w:\Spec\T\to \Bend_v(Z)$. Let $w^\#:\Bend_v(B)\to\T$ be the morphism between the respective blueprints of global sections. Given a polynomial $f=\sum a_eX^e\in S_r^+$ where $a_e\in\T$ and $e\in\Z^r$ is a multi-index, we define $a_ew^e=w^\#(\pi(a_eX^e))$ and $f(w)=\sum a_ew^e$, which are elements of $\T$. The \emph{initial form of $f$ in $w$} is 
\[
 \init_w(f) \quad = \quad \sum_{a_ew^e=f(w)} X^e,
\]
considered as a polynomial in $\overline S_r^+$. Note the formal analogy with the initial form of a polynomial with coefficients in $k$: the identity valuation $\id:\T\to \T$ yields the \emph{integers} $\cO_\T=\{a\in\T|a\leq1\}$, which is a local semiring with maximal ideal $\fm_\T=\{a\in\cO_\T|a<1\}$ and residue field $\B=\cO_\T/\fm_\T$. The unique section to $\id$ is the identity $\id:\T\to \T$, thus the coefficients $\overline{\id(a_e^{-1})a_e}$ of the initial form are $1$.

The \emph{initial preaddition of $\cR_\pi$ in $w$} is the preaddition 
\[\textstyle
 \init_w(\cR_\pi) \quad \= \quad \left\langle\left. \ \init_w\bigl(\sum a_i\bigr) \= \init_w\bigl(\sum b_j\bigr) \ \right| \ \sum a_i\=\sum b_j\text{ in }\cR_\pi \ \right\rangle
\]
on $\overline S_r$. Since $w$ is smooth, the set $L=\Hom(\bpquot{\overline S_r}{\init_w(\cR_\pi)},\T)$ is a linear subspace of $(\T^\times)^r\simeq\R^r$. 

We can exchange $\pi:S_r\to\Bend_v(B)$ by $\pi\circ\varphi$ for a suitable automorphism $\varphi$ of $S_r$, so that $L$ coincides with the span $\langle e_1,\dotsc,e_d\rangle$ of the first $d$ unit vectors of $\R^r$. We define $\init_w(\cR_\pi)'$ as the restriction of $\init_w(\cR_\pi)$ to $\B[X_{d+1}^{\pm1},\dotsc,X_r^{\pm1}]$. 

A subset $S$ of a $\B$-linear space $V$ is \emph{linearly independent over $\B$} if every element of $V$ can be written in at most one way as a finite $\B$-linear combination of the elements of $S$. The \emph{dimension $\dim_\B V$ of $V$} is the supremum of the cardinalities of all linear independent sets.

\begin{df}
 Let $w\in\Bend_v(Z)(\T)$ be a smooth point. The \emph{Maclagan-Rinc\'on weight of $w$} is
 \[
  \mu(w) \quad = \quad \dim_\B \, \big(\, \B[X_{d+1}^{\pm1},\dotsc,X_r^{\pm1}]^+ / \, \init_w(\cR_\pi)' \, \bigr).
 \]
\end{df}

The following theorem is essentially Theorem 1.2 of \cite{Maclagan-Rincon14}. Since the weights $\mu(w)$ are not explicitly exhibited in \cite{Maclagan-Rincon14} and an additional argument is required, we include a proof.

\begin{thm}\label{thm: Maclagan-Rincon weights}
 The multiplicity $\mult(w)$ coincides with the Maclagan-Rinc\'on weight $\mu(w)$ for every smooth point $w\in\Trop(Y)$. In particular, $\mu(w)$ is an invariant of $w$ and does not depend on the choice of $\pi$.
\end{thm}

\begin{proof}
 As a first point, we observe that the weights of $\Trop(Y)$ are invariant under interchanging coordinates, scaling coordinates and inclusions $\T^n\to\T^r$ as the first $n$ coordinates. This means that the weights of $\Trop(Y)=\Bend_v(Z)(\T)$ are invariant under automorphisms of $R_n$ and under a change $Y\to\G_{m,k}^{n,+}\to\G_{m,k}^{r,+}$ of the ambient torus that comes from a morphism $\G_{m,k}^n\to\G_{m,k}^r$ of blue $k$-schemes. 
 
 Therefore we can choose a morphism $\iota:Y\to \G_{m,k}^r$ such that $\iota^+:Y\to \G_{m,k}^{r,+}$ is a closed immersion of $k$-schemes and such that $\pi=\Bend_v(\eta)$ where $\eta=\iota^\ast:R_r\to\Gamma Y$ is the corresponding morphism between the respective global sections. In this case, the preaddition $\cR_\pi$ on $S_r$ equals the Giansiracusa bend relation $\bend_{v,\iota}^{GG}(I)$ where $I\subset R_r^+$ is the ideal defining $Y$. This reduces the proof to the situation of \cite{Maclagan-Rincon14}.

 By our choice of $\pi$ for a given $w$, the linear subspace $L$ of $\T^r\simeq\R^r$ equals the span $\langle e_1,\dotsc,e_d\rangle$ of the first $d$ unit vectors. By \cite[Lemma 3.4.7]{Maclagan-Sturmfels15}, we have 
 \[
  \mult(w) \quad = \quad \dim_{k_0}(k_0[X_{d+1}^{\pm1},\dotsc,X_r^{\pm1}]^+ /\init_w(I)')
 \]
 where $\init_w(I)'$ is the restriction of $\init_w(I)$ to $k_0[X_{d+1}^{\pm1},\dotsc,X_r^{\pm1}]^+$.

 Let $v_0:k_0\to \B$ be the trivial valuation. By \cite[Prop.\ 3.4]{Maclagan-Rincon14}, we have 
 \[
  \init_w\bigl(\bend_{v,\eta}^{GG}(I)\bigr) \quad = \quad \bend_{v_0,\eta}^{GG}\bigl(\init_w(I)\bigr).
 \]
 Both $\init_w$ and $\bend_{v,\eta}^{GG}$ commute with the restriction to the variables $X_{d+1},\dotsc,X_r$. Therefore we obtain $\init_w\bigl(\bend_{v,\eta}^{GG}(I)\bigr)' = \bend_{v_0,\eta}^{GG}\bigl(\init_w(I)'\bigr)$ and 
 \[
  \B[X_{d+1}^{\pm1},\dotsc,X_r^{\pm1}]^+ \, \bigl/ \ \init_w\bigl(\bend_{v,\eta}^{GG}(I)\bigr)' \quad = \quad \Trop_{v_0,\eta}^{GG}\bigl(k_0[X_{d+1}^{\pm1},\dotsc,X_r^{\pm1}]^+/ \, \init_w(I)'\bigr).
 \]
 
 Since $k_0[X_{d+1}^{\pm1},\dotsc,X_r^{\pm1}]^+/\init_w(I)'$ is a $k_0$-vector space of finite dimension $\mult(w)$, its tropicalization is a $\B$-linear space of the same dimension, cf.\ \cite[Lemma 7.1.3]{Giansiracusa13}. Since $\cR_\pi=\bend_{v,\iota}^{GG}(I)$, this dimension equals, by definition, the Maclagan-Rinc\'on weight $\mu(w)$. This shows that $\mu(w)=\mult(w)$ and finishes the proof.
\end{proof}

Theorem \ref{thm: Maclagan-Rincon weights} allows us to apply the structure theorem for tropicalizations to the scheme theoretic tropicalization $\Bend_v(Z)$ of the blue $k$-scheme $Z$ associated with a purely $d$-dimensional closed subscheme $Y$ of $\G_{m,k}^{n,+}$.

\begin{cor}
 Let $\Bend_v(Z)\to\G_{m,\T}^r$ be a closed immersion of blue $\T$-schemes. Then we can endow $\Bend_v(Z)(\T)\subset (\T^\times)^r\simeq\R^r$ with the structure of a balanced weighted polyhedral complex of dimension $d$ such that the weight of its $d$-dimensional polyhedra $\sigma$ equals $\mu(w)$ for each $w$ in the relative interior of $\sigma$. \qed
\end{cor}

\begin{ex}\label{ex: Maclagan-Rincon weights}
 Let $Y=\Spec \big(k[X_1^{\pm1},X_2^{\pm2}]/\gen{X_1^i+X_j^j+1}\big)$ be the closed subscheme of $\G_{m,k}^2$ that is defined by the equation $X_1^i+X_2^j+1$ for some $i,j\geq1$. Then $X_1^i+X_2^j+1$ (considered as a tropical polynomial) is a tropical basis of the ``weighted tropical line'' $\Trop(Y)$ in $(\T^\times)^2$: it has one central vertex ($0$-dimensional polyhedra) and three infinite rays ($1$-dimensional polyhedras) that connect each to the central vertex at one end. We determine the Maclagan-Rinc\'on weights for the three rays in the following.
 
% As a first observation, note that the tropical 
 
 As a first case, we consider the ray $\sigma_1=\{(a_1,a_2)\mid a_1^i\leq a_2^j=1\}=\{(a_1,1)\mid a_1\leq 1\}$ where the terms $X_2^j$ and $1$ of the defining equation assume the maximum. A point $w=(a_1,1)\in\sigma_1$ is smooth if and only if it is not equal to the central vertex $w_0=(1,1)$. For smooth $w\in\sigma_1$, the initial preaddition $\init_w(\cR_\pi)$ is generated by the single relation
 \[
  X_2^j \ \= \ X_2^j+1 \ \= \ 1,
 \]
 which is particular to the fact that $X_2^j+1$ contains only $2$ terms. The linear subspace
 \[
  L \ = \ \Hom(\bpquot{\B[X_1^{\pm1},X_2^{\pm1}]}{\init_w(\cR_\pi)},\T) \ \simeq \ \big\{ (x_1,x_2)\in\R^2 \, \big| \, x_2=0 \big\}
 \]
 of $\R^2$ is already in the desired position, which allows us to choose $\varphi=\id$. A free submodule of maximal rank of $\bpquot{\B[X_2^{\pm1}]}{\init_w(\cR_\pi)'}=\bpquot{\B[X_2^{\pm1}]}{\gen{X_2^j\=1}}$ is the $\B$-submodule generated by $1,X_2,\dotsc,X_2^{j-1}$, whose rank is $j$. Thus
 \[
  \mu(w) \ = \ \dim_\B \big( \bpquot{\B[X_2^{\pm1}]}{\init_w(\cR_\pi)'} \big) \ = \ j.
 \]

 The case of a smooth $w\in\sigma_2=\{(a_1,a_2)\mid a_2^j\leq a_1^i=1\}=\{(1,a_2)\mid a_2\leq 1\}$ is similar. We need to exchange the standard basis of $\R^2$, i.e.\ $\varphi$ is given by the permutation matrix $\big(\begin{smallmatrix} 0 & 1 \\ 1 & 0 \end{smallmatrix}\big)$, and gain 
 \[
  \mu(w) \ = \ \dim_\B \big( \bpquot{\B[X_2^{\pm1}]}{\init_w(\cR_\pi)'} \big) \ = \ i
 \]
 by the same reasoning as in the previous case.
 
 The situation for the third ray $\sigma_3=\{(a_1,a_2)\mid 1\leq a_1^i=a_2^j\}$ is slightly more intriguing. Let $w=(a_1,a_2)\in\sigma_3$ be smooth, i.e.\ $a_1,a_2>1$. Then $\init_w(\cR_\pi)$ is generated by $X_1^i\=X_2^j$ and 
 \[
  L \ = \ \Hom(\bpquot{\B[X_1^{\pm1},X_2^{\pm1}]}{\init_w(\cR_\pi)},\T) \ \simeq \ \big\{ (x_1,x_2)\in\R^2 \, \big| \, ix_1=jx_2 \big\}.
 \]
 Let $i_0=i/\gcd(i,j)$ and $j_0=j/\gcd(i,j)$. Then $\gcd(i_0,j_0)=1$ and therefore $dj_0-ci_0=1$ for some $c,d\in\Z$. If we let $\varphi$ be the automorphism of $S_2$ that is represented by the matrix $\big(\begin{smallmatrix} j_0 & i_0 \\ c & d \end{smallmatrix}\big)$, then the relation $X_1^{-i}X_2^j\=1$ transforms to $X_2^{\gcd(i,j)}\=1$. Following the same arguments as before yields thus
 \[
  \mu(w) \ = \ \dim_\B \big( \bpquot{\B[X_2^{\pm1}]}{\init_w(\cR_\pi)'} \big) \ = \ \gcd(i,j)
 \]
 in this last case.
\end{ex}

\subsection{What is new?}
In this section, we have extended the results from \cite{Maclagan-Rincon14} to the intrinsic tropicalization of a closed subvariety of a torus as a blue $\T$-scheme. This is a subtle step towards a more rigorous setting of scheme theory for tropical geometry. 

Moreover, we have exhibited an explicit formula of the Maclagan-Rinc\'on weights, which opens the door to investigate weights in more general situations, for example for valuations $v:k\to\T$ whose image is not dense or even trivial, or in the case of blue $\T$-schemes that are not tropicalizations of classical varieties.

 %%%%%%%%%%%%%%%%%%%%%%%%%%%%%%%%%%%%%%%%%%%%%%%%%%%%%%%%%%%%%%%%%%%%%%%%%%%%%%%%%%%%%%%%%%%%%%%%%%%%%%%%%%%%%%%%%%%%%%%%%%%%%%%%%%%%%%%%%%%%%%%%%%%%%%%%%%%%%%%%%%%%%%%%%%%

\section{Macpherson analytification}
\label{section: Macpherson analytification}

Let $k$ be a ring and $A$ a $k$-algebra. One of the key ideas of Macpherson's paper \cite{Macpherson13} is that the semiring $\An(A,k)$ of finitely generated $k$-submodules of $A$ represents the functor of valuations on $A$ in idempotent semirings that are integral on $k$. The focus of \cite{Macpherson13} lies on extending this concept to non-Archimedean analytic geometry, for which reason the pair $(A,k)$ is assumed to form a non-Archimedean ring. We refrain from an excursion into non-Archimedean geometry, but we will extend $\An(A,k)$ to ordered blueprints $k$ and $A$.

\subsection{The universal Giansiracusa tropicalization as Macpherson analytification} 
\label{subsection: universal Giansiracusa tropicalization as Macpherson analytification}
Let us begin with reviewing the results and immediate implications of \cite{Macpherson13}. The \emph{Macpherson analytification} $\An(A,k)$ is the semiring of all finitely generated $k$-submodules of $A$ with respect to the addition
\[
 M_1+M_2 \quad = \quad \bigl\{ \ m\in A \ \bigl| \ m=m_1+m_2\text{ for some }m_1\in M_1\text{ and }m_2\in M_2 \ \bigl\}
\]
and the multiplication
\[
 M_1\cdot M_2 \quad = \quad \bigl\langle \ m\in A \ \bigl| \ m=m_1\cdot m_2\text{ for some }m_1\in M_1\text{ and }m_2\in M_2 \ \bigl\rangle.
\]
The semiring $\An(A,k)$ is idempotent and comes with the valuation $v:A\to \An(A,k)$, which sends $a\in  A$ to the $k$-submodule of $A$ generated by $a$. This valuation is \emph{integral on $k$}, i.e.\ $v(a)+1=1$ for all $a\in k$.

Let $\Val(A,k;-):\Alg_\B\to\Sets$ be the functor of all valuations on $A$ in idempotent semirings that are integral on $k$. The key observation of Macpherson is that $\An(A,k)$ represents $\Val(A,k;-)$.

From this, we can deduce the following description of the \emph{universal Giansiracusa tropicalization} $\Trop_{v,\eta}^{GG}(A)$ (cf.\ Corollary \ref{cor: the universal tropicalization}) where $v:k\to T$ is a valuation into a totally ordered idempotent semiring $T$ and $\eta:k[A^\bullet]\to A$ is the natural $k$-linear map. Let
\[
 \cO_k \quad = \quad \bigl\{ \ a\in k \ \bigl| \ v(a)+1=1 \bigr\}
\]
be the subring of integral elements in $k$. Then the valuation $v:k\to T$ corresponds to a homomorphism $\An(k,\cO_k)\to T$ of semirings by the universal property of $\An(k,\cO_k)$. Since $\Trop_{v,\eta}^{GG}(A)$ represents $\Val_v(A,-)$ and $\Val_v(A,S)$ corresponds to the valuations $w$ in $\Val(A,\cO_k;S)$ that restrict to $w\vert_{k}=v$, we have a canonical isomorphism of semirings
\[
 \Trop_{v,\eta}^{GG}(A) \quad \stackrel\sim\longrightarrow \quad \An(A,\cO_k)\Sotimes_{\An(k,\cO_k)}T
\]
where $B\Sotimes_DC$ stays for $(B\otimes_DC)^+$, cf.\ section \ref{section: conventions}.

%By gluing affine patches, this connection generalizes to a description of the universal Giansiracusa tropicalization of a $k$-scheme $X$. 

\subsection{The Macpherson analytification as bend}
A variation of the definition of $\An(A,k)$ yields a description of the Giansiracusa tropicalization for a closed immersion $\iota:\Spec\to X$ into an affine toric variety $X$, see \cite[para.\ 7.3]{Macpherson13}. In the following, we will make this precise by different means: we extend the Macpherson analytification to all ordered blueprints $k$ and $A$, which yields a description of the bend functor in terms of $\An(A,k)$ and, conversely, a description of $\An(A,k)$ as bend.

Let $k$ be an ordered blueprint and $B$ an ordered blue $k$-algebra. A \emph{$k$-span in $B$} is a subset $M$ of $B$ that is closed under multiplication by elements of $k$ and that contains all $a\in B$ for which there are $b_j\in M$ such that $a\leq\sum b_j$. We write $\gen{a_i}$ for the smallest $k$-span in $B$ that contains the elements $a_i$. A $k$-span $M$ in $B$ is \emph{finitely generated} if $M=\gen{a_i}$ for finitely many elements $a_i\in B$. For $a\in k$, we write $\gen a=\gen{\bar a}$ where $\bar a$ is the image of $a$ in $B$.

The semiring $\An(B,k)$ is the set of all finitely generated $k$-spans in $B$ together with the addition
\[
 M_1+M_2 \quad = \quad \bigl\{ \ m\in B \ \bigl| \ m\leq m_1+m_2\text{ for some }m_1\in M_1\text{ and }m_2\in M_2 \ \bigl\}
\]
and the multiplication
\[
 M_1\cdot M_2 \quad = \quad \bigl\{ \ m\in B \ \bigl| \ m=m_1\cdot m_2\text{ for some }m_1\in M_1\text{ and }m_2\in M_2 \ \bigl\}.
\]
Note that this recovers the definition of $\An(B,k)$ in the case of rings $k$ and $B$, and that $\An(B,k)$ is an idempotent semiring for all ordered blueprints $k$ and $B$. In other words, $\An(B,k)$ is a $\B$-algebra.

The $\B$-algebra $\An(B,k)$ comes with the map $v:B\to \An(B,k)$ that sends $a$ to $\gen a$. This map is a morphism between the underlying monoids. If $a\leq \sum b_j$ in $B^\mon$, then $\gen a\subset \gen{b_j}$ as subsets of $B$, which implies that $\gen{a}+\sum \gen{b_j}=\sum\gen{b_j}$ in $\An(B,k)$. Thus $\gen a\leq\sum \gen{b_j}$ in $\An(B,k)^\pos$. This shows that $v:B\to \An(B,k)$ is a valuation. 

For $a\in k$, we have $\gen{a}\subset\gen{1}$ as subsets of $B$ and thus $\gen{a}+\gen{1}=\gen{1}$ in $\An(B,k)$. Therefore $v(a)\leq 1$ in $\An(B,k)^\pos$, which means, by definition, that $v$ is integral on $k$.

\begin{lemma}
 Let $S$ be a multiplicative subset of $B$. Then the association $\gen{\frac as}\mapsto \frac{\gen a}{\gen s}$ defines an isomorphism $\An(S^{-1}B,k)\to v(S)^{-1}\An(B,k)$.
\end{lemma}

\begin{proof}
 Since $s\in S$ is mapped to the invertible element $\gen s$ of $v(S)^{-1}\An(B,k)$, the association $\gen{\frac as}\mapsto \frac{\gen a}{\gen s}$ defines a morphism $\An(S^{-1}B,k)\to v(S)^{-1}\An(B,k)$. Conversely, $\gen s\in v(S)$ is invertible in $\An(S^{-1}B,k)$. Thus the inverse association $\frac{\gen a}{\gen s}\mapsto \gen{\frac as}$ defines an inverse morphism $v(S)^{-1}\An(B,k)\to \An(S^{-1}B,k)$.
\end{proof}

As a consequence of this lemma, an affine presentation $\cU$ of an ordered blue $k$-scheme $X$ yields an affine presentation $\An(\cU,k)$ in $\B$-algebras. We define $\An(X,k)$ as the colimit of $\An(\cU,k)$. It comes with a valuation $v:\An(X,k)\to X$ that is integral on $k$, which means that for all affine open subschemes $U=\Spec B$ of $X$, the induced valuation $\Gamma v\vert_U:B\to \An(B,k)$ is integral on $k$. We denote by $\Val(X,k;-)$ the functor that takes a $\B$-algebra $S$ to the set of valuations $\Spec S\to X$ that are integral on $k$. 

Let $B$ be an ordered blue $k$-algebra. We define 
\[
 B^\mon_{k\leq 1} \quad = \quad \bpgenquot{B^\mon}{\bar a\leq 1|a\in k}
\]
where $\bar a$ is the image of $a$ in $B$. This definition is obviously invariant under localization. Thus we can define for an ordered blue $k$-scheme $X$ with affine presentation $\cU$ the affine presentation $\cU^\mon_{k\leq 1}$, and $X^\mon_{k\leq 1}$ as its colimit.

\begin{thm}\label{thm: Macpherson analytification and the bend functor}
 Let $k$ be an ordered blueprint, $X$ an ordered blue $k$-scheme and $v:\Fun\to\B$ the trivial valuation. Then there is a canonical morphism 
 \[
  \An(X,k) \quad \longrightarrow \quad \Bend_v(X^\mon_{k\leq1})
 \]
 that induces an isomorphism $\An(X,k)\simeq \Bend_v(X^\mon_{k\leq1})^+$ of semiring schemes, and $\An(X,k)$ represents $\Val(X,k;-)$.
\end{thm}

\begin{proof}
 Since all constructions in questions are defined in terms of affine presentations, it is enough to prove the theorem in the affine case $X=\Spec B$. In this case, the canonical map
 \[
  \psi: \ \Bend_v(B^\mon_{k\leq 1}) \ = \ \bpquot{B^\bullet\otimes_{\Fun}\B}{\bend_v(B^\mon_{k\leq1})} \quad \longrightarrow \quad \An(B,k)
 \]
 is given by $a\otimes 1\mapsto \gen a$. This map is clearly a morphism between the respective underlying monoids, and since $\gen a+\gen a=\gen a$, it is $\B$-linear. In order to see that $\psi$ respects the bend relations, consider a relation $a\leq\sum b_j$ in $B^\mon_{k\leq 1}$. Then $\gen a \subset \gen{b_j}$ and thus $\gen{a,b_j}=\gen{b_j}$. Thus the relation $a+\sum b_j\=\sum b_j$ in $\bend_v(B^\mon_{k\leq1})$ implies that $\gen a + \sum \gen{b_j}\=\sum \gen{b_j}$ in $\An(B,k)$. This shows that $\psi$ is a morphism of ordered blueprints.
 
 Since $\An(B,k)$ is generated by the principal ideals $\gen a$ as a semiring, it suffices to show that every relation in $\An(B,k)$ is already contained in $\Bend_v(B^\mon_{k\leq1})$ in order to prove that $\psi^+:\Bend_v(B^\mon_{k\leq1})^+\to \An(B,k)$ is an isomorphism. Therefore, let us consider an equality $\gen{a_i}=\gen{b_j}$ of $k$-spans of $B$. Then we have for all $i$ a relation $a_i\leq \sum c_{i,j}b_j$ for certain $c_{i,j}\in k$ and for all $j$ a relation $b_j\leq \sum d_{j,i}a_i$ for certain $d_{j,i}\in k$. Since in $B^\mon_{k\leq1}$, we have $c_{i,j}\leq 1$ and $d_{j,i}\leq1$, this implies $a_i\leq\sum b_j$ and $b_j\leq \sum a_i$ in $B^\mon_{k\leq1}$. Therefore we find the relations $a_i+\sum b_j\=\sum b_j$ and $\sum a_i\=\sum a_i+b_j$ in $\bend_v(B^\mon_{k\leq1})$. Using that $\Bend_v(B^\mon_{k\leq1})$ is idempotent, we find
 \[\textstyle
  \sum_i a_i \ = \ \sum_{j,i} a_i \ = \ \sum_j \bigl( \sum_i a_i +b_j \bigr) \ = \ \sum_i a_i +\sum_j b_j \ = \ \sum_i \bigl(a_i+\sum_j b_j \bigr) \ = \ \sum_{i,j} b_j \ = \ \sum_j b_j
 \]
 in $\Bend_v(B^\mon_{k\leq1})$. This shows that $\psi^+$ is an isomorphism of semirings.
 
 By Theorem \ref{thm: tropicalization for idempotent base}, we know that $\Bend_v(B^\mon_{k\leq1})$ represents the functor $\Val_v(B^\mon_{k\leq1},-)$ on ordered blue $\B$-algebras. Since $\psi^+$ is an isomorphism of semirings, $\An(B,k)$ represents the restriction $\Val_v^+(B^\mon_{k\leq1},-)$ of $\Val_v(B^\mon_{k\leq1},-)$ to $\B$-algebras. A valuation $w:B \to S$ in a $\B$-algebra $S$ that is integral on $k$ is the same as a valuation $w:B^\mon_{k\leq1}\to S$, and every valuation in a $\B$-algebra is an extension of the trivial valuation $v:\Fun\to\B$. Therefore the functors $\Val_v^+(B^\mon_{k\leq1},-)$ and $\Val(B,k;-)$ are isomorphic. We conclude that $\An(B,k)$ represents $\Val(B,k;-)$, which completes the proof of the theorem.
\end{proof}

\begin{rem}
 Note that for a description of the Macpherson analytification, we use non-algebraic blueprints in an essential way. Though the bends $\Bend_v(B^\mon)$ and $\Bend_v(B)$ coincide, the relation $a\leq 1$ for $a\in k$ implies $a=1$ in $B_{k\leq1}$ in the typical case that $B$ is with $-1$. Moreover, we have to endow $B$ with the relations $a\leq 1$ for $a\in k$ to guarantee that the bend represents only valuations that are integral on $k$.
\end{rem}

\begin{cor}\label{cor: the bend as Macpherson analytification}
 Let $k$ be an ordered blueprint, $X$ an ordered blue $k$-scheme and $v:k\to T$ a valuation in an idempotent semiring. Let $\cO_k=\{a\in k|v(a)\leq 1\text{ in }T^\pos\}$. Then there exists a canonical isomorphism
 \[
  \Bend_v(X)^+ \quad \stackrel\sim\longrightarrow \quad \An(X,\cO_k)\Sotimes_{\An(k,\cO_k)}T.
 \]
\end{cor}

\begin{proof}
 Similar to our argument in section \ref{subsection: universal Giansiracusa tropicalization as Macpherson analytification}, this can be proven by showing that both semirings represent the functor $\Val(X,k;-)$. An alternative proof is the following direct calculation.
 
 Let $v_0:\Fun\to \B$ be the trivial valuation and $C$ an ordered blue blueprint. Then we have
 \[
  \Bend_{v_0}(C) \quad = \quad \bpquot{C^\bullet\otimes_\Fun \B}{\bend_{v_0}(C)} \quad = \quad \bpquot{C}{\bend(C)}
 \]
 where $\bend(C)=\{a+\sum b_j\=\sum b_j|a\leq\sum b_j\text{ in }C\}$ does not depend on the valuation $v_0$. Therefore, we obtain
 \begin{align*}
     \Bend_{v_0}(B^\mon_{\cO_k\leq1})\otimes_{\Bend_{v_0}(k^\mon_{\cO_k\leq1})} T \quad 
  &= \quad \bpquot{B^\bullet}{\bend(B^\mon_{\cO_k\leq1})} \otimes_{\bpquot{k^\bullet}{\bend(k^\mon_{\cO_k\leq1})}} T \\
  &= \quad \bpquot{B^\bullet\otimes_{k^\bullet}T}{\bend_v(B^\mon_{\cO_k\leq 1})}.
 \end{align*}
 Since, by definition of $\cO_k$, an element $a\in \cO_k$ implies the bend relation $a+1\=1$ in $T$, we conclude that this ordered blueprint is isomorphic to
 \begin{align*}
  \bpquot{B^\bullet\otimes_{k^\bullet}T}{\bend_v(B^\mon)} \quad = \bpquot{B^\bullet\otimes_{k^\bullet}T}{\bend_v(B)} \quad = \quad \Bend_v(B).
 \end{align*}
 The claim of the Corollary follows from applying $(-)^+$ to the above equations and using the isomorphism $\An(C,\cO_k)\simeq \Bend_{v_0}(C_{\cO_k\leq0})^+$ from Theorem \ref{thm: Macpherson analytification and the bend functor}.
\end{proof}

Together with Theorem \ref{thm: Giansiracusa tropicalization as bend}, this yields the following description of the Giansiracusa tropicalization in its general form.

\begin{cor}
  Let $k$ be a ring, $v:k\to T$ be a valuation in a totally ordered idempotent semiring $T$. Let $\iota:X\to Y_k^+$ be a closed immersion where  $Y$ is a monoid scheme. Then there is a canonical isomorphism
 \[
  \Trop_{v,\iota}^{GG}(X) \quad \stackrel\sim\longrightarrow \quad \An(Z,\cO_k)\Sotimes_{\An(k,\cO_k)}T
 \]
 where $Z$ is the associated blue scheme and $\cO_k=\{a\in k|v(a)\leq 1\text{ in }T^\pos\}$. \qed
\end{cor}

%%%%%%%%%%%%%%%%%%%%%%%%%%%%%%%%%%%%%%%%%%%%%%%%%%%%%%%%%%%%%%%%%%%%%%%%%%%%%%%%%%%%%%%%%%%%%%%%%%%%%%%%%%%%%%%%%%%%%%%%%%%%%%%%%%%%%%%%%%%%%%%%%%%%%%%%%%%%%%%%%%%%%%%%%%%
%%%%%%%%%%%%%%%%%%%%%%%%%%%%%%%%%%%%%%%%%%%%%%%%%%%%%%%%%%%%%%%%%%%%%%%%%%%%%%%%%%%%%%%%%%%%%%%%%%%%%%%%%%%%%%%%%%%%%%%%%%%%%%%%%%%%%%%%%%%%%%%%%%%%%%%%%%%%%%%%%%%%%%%%%%%

\section{Thuillier analytification}
\label{section: Thuillier analytifiction}

An important variant of the Berkovich analytification was found by Thuillier in \cite{Thuillier07} in case of the trivial valuation $v:k\to \T$ of a field $k$. In the following, we will review the definition of $X^\beth$, and reinterpret this topological space in terms of the scheme theoretic tropicalization of $X$.

Let $\cO_\T=\{a\in\T|a+1=1\}$ be the subsemiring of $\T$ whose underlying set is the real interval $[0,1]$. We endow $\cO_\T$ with the real topology of $[0,1]\subset\R$.

In the affine case $X=\Spec B$, the \emph{Thuillier analytification $X^\beth$} consists of all valuations $w:B\to \T$ in $X^\an$ whose image is contained in $\cO_\T$. It comes with the subspace topology of $X^\an$. In the case of an arbitrary $k$-scheme $X$ with affine presentation $\cU$, we define $X^\beth$ as the colimit of $\cU^\beth$ as a topological space.

Recall from section \ref{section: Berkovich analytification} that a blue model of $X$ is a blue $k$-scheme $Z$ together with an isomorphism $Z^+\to X$.

\begin{thm}\label{thm: Thuillier analytification as rational point set}
 Let $Z$ be a blue model of $X$ such that $\cO_Z(U)$ is a ring for all open subsets $U$ of $Z$. Then the Thuillier space $X^\beth$ is naturally homeomorphic to $\Bend_v(Z)(\cO_\T)$.
\end{thm}

\begin{proof}
 Note that the semiring $\cO_\T$ is a local topological Hausdorff semiring with open unit group. This allows us to apply the same proof as for Theorem \ref{thm: Berkovich analytification as rational point set} to verify the claim of the theorem.
\end{proof}

\begin{rem}
 An alternative realization of $X^\beth$ as a rational point set is the following. Define $B^\mon_{\leq1}=\bpgenquot{B^\mon}{a\leq1|a\in B}$ for an ordered blueprint $B$. Then every valuation $w:B^\mon_{\leq1}\to \T$, which is a morphism $\tilde w:B^\mon_{\leq1}\to \T^\pos$, respects the relation $a\leq 1$, i.e.\ $w(a)\in\cO_\T$. If we define $Z^\mon_{\leq1}$ in terms of an affine presentation, then every valuation $\omega:\Spec\T\to Z^\mon_{\leq1}$ factors through the morphism $\Spec\T\to\Spec\cO_\T$ that corresponds to the inclusion $\cO_\T\subset\T$. We conclude that we have natural bijections
 \[
  X^\beth \quad = \quad \Val_v(Z^\mon_{\leq1},\cO_\T) \quad = \quad \Val_v(Z^\mon_{\leq1},\T) \quad = \quad \Bend_v(Z_{\leq1}^\mon)(\T),
 \]
 which yields, in fact, a homeomorphism of topological spaces.
\end{rem}

%%%%%%%%%%%%%%%%%%%%%%%%%%%%%%%%%%%%%%%%%%%%%%%%%%%%%%%%%%%%%%%%%%%%%%%%%%%%%%%%%%%%%%%%%%%%%%%%%%%%%%%%%%%%%%%%%%%%%%%%%%%%%%%%%%%%%%%%%%%%%%%%%%%%%%%%%%%%%%%%%%%%%%%%%%%
%%%%%%%%%%%%%%%%%%%%%%%%%%%%%%%%%%%%%%%%%%%%%%%%%%%%%%%%%%%%%%%%%%%%%%%%%%%%%%%%%%%%%%%%%%%%%%%%%%%%%%%%%%%%%%%%%%%%%%%%%%%%%%%%%%%%%%%%%%%%%%%%%%%%%%%%%%%%%%%%%%%%%%%%%%%

\section{Ulirsch tropicalization}
\label{section: Ulirsch tropicalization}

In the case of a field $k$ with trivial valuation $v:k\to \T$, and a toroidal embedding $U\subset X$ without self-intersection, Thuillier defines in \cite{Thuillier07} a retraction $X^\beth\to \overline\Sigma_X$ onto an extended cone complex $\overline\Sigma_X$ associated with $U\subset X$. Abramovich, Caporaso and Payne interpret in \cite{Abramovich-Caporaso-Payne12} this retraction map as the tropicalization of $X$. Ulirsch generalizes in \cite{Ulirsch13} Thuillier's tropicalization by associating a log structure with a toroidal embedding. Ulirsch's tropicalization passes through an associated Kato fan that allows him to apply the local tropicalization of Popescu-Pampu and Stepanov in \cite{Popescu-Pampu-Stepanov13} to charts of the log structure. This also recovers the tropicalization of fine and saturated log schemes as studied by Gross and Siebert in \cite{Gross-Siebert13}.

In this section, we will review Ulirsch's tropicalization of log schemes and connect it to the scheme theoretic tropicalization developed in this paper, under some additional assumptions: we restrict ourselves to Zariski log structures and require that the log scheme has an affine open covering that is compatible with the log structure; cf.\ Remark \ref{rem: etale log schemes} and section \ref{subsection: blue scheme associated with a fine and saturated log scheme}. In contrast to \cite{Ulirsch13}, we write all monoids multiplicatively.

%%%%%%%%%%%%%%%%%%%%%%%%%%%%%%%%%%%%%%%%%%%%%%%%%%%%%%%%%%%%%%%%%%%%%%%%%%%%%%%%%%%%%%%%%%%%%%%%%%%%%%%%%%%%%%%%%%%%%%%%%%%%%%%%%%%%%%%%%%%%%%%%%%%%%%%%%%%%%%%%%%%%%%%%%%%

\subsection{Kato fans}
A \emph{monoidal space} is a topological space $X$ together with a sheaf of monoids $\cM_X$. Note that every monoid $M$ is local since the complement of its units $M^\times$ forms the unique maximal $m$-ideal of the monoid. Therefore every stalk $\cM_{X,x}$ is a local monoid with maximal $m$-ideal $\fm_x$. A \emph{(local) morphism of monoidal spaces} is a continuous map $\varphi:X\to Y$ together with a morphism $\varphi^\flat:\varphi^{-1}\cM_Y\to \cM_X$ of sheaves of monoids such that the induced morphisms of stalks $\varphi_x:\cM_{Y,y}\to\cM_{X,x}$ map $\fm_y$ to $\fm_x$ for all $x\in X$ and $y=\varphi(x)$. A morphism $\varphi:X\to Y$ of monoidal spaces is \emph{strict} if $\varphi^\flat:\varphi^{-1}\cM_Y\to \cM_X$ is an isomorphism.

The unit group $M^\times$ acts by multiplication on $M$, and the quotient $M/M^\times$ of this action inherits the structure of a monoid since multiplication is commutative. 

A monoid is \emph{sharp} if $M^\times=\{1\}$. A monoidal space $X$ is \emph{sharp} if $\cM_X(U)^\times=\{1\}$ for all open subsets $U$ of $X$. Given a monoidal space $(X,\cM_X)$, we define its associated sharp monoidal space as $(X,\overline\cM_X)$ where $\overline\cM_X=\cM_X/\cM_X^\times$.

A \emph{multiplicative set in $M$} is a multiplicatively closed subset $S\subset M$ that contains $1$. The \emph{localization of $M$ at $S$} is $S^{-1}M=S\times M/\sim$ where $(s,m)\sim (s',m')$ if and only if there is a $t\in S$ such that $tsm'=ts'm$. We denote the class of $(s,m)$ in $S^{-1}M$ by $\frac ms$. A \emph{prime ideal of $M$} is a subset $\fp\subset M$ such that $\fp M=\fp$ and $S=M-\fp$ is a multiplicative set.

The \emph{affine Kato fan $\Spec^KM$} of a monoid $M$ is the following sharp monoidal space. Its underlying topological space is the set of all prime ideals $\fp$ of $M$ endowed with the topology generated by open subsets of the form $U_h=\{\fp|h\notin\fp\}$. Its structure sheaf $\cM_{\Spec^KM}$ associates with $U_h$ the sharp monoid $S^{-1}M/(S^{-1}M)^\times$ where $S=\{h^i\}_{i\geq0}$. A \emph{Kato fan} is a sharp monoidal space that has an open covering by affine Kato fans.

A monoid $M$ is \emph{fine} if it is finitely generated and embeds into its Grothendieck group $M^\gp$. A fine monoid is \emph{saturated} if $a^n\in M$ with $a\in M^\gp$ and $n\geq1$ implies $a\in M$. A Kato fan is \emph{fine and saturated} if it can be covered by affine Kato fans of fine and saturated monoids.

\begin{rem}
 Some of the notions introduced in this section have already been defined for monoids with zero, which are ordered blueprints. Though most notions are in spirit the same and can be recovered by associating an additional element $0$ to a monoid in the sense of this section, there is an important digression in the notion of the spectrum. While the spectrum of a monoid $A$ with zero associates with an open $U_h$ localizations $S^{-1}A$ where $S=\{h^i\}_{i\geq0}$, the affine Kato fan of a monoid $M$ associates with an open $U_h$ the sharp monoid $S^{-1}M/(S^{-1}M)^\times$.
\end{rem}

%%%%%%%%%%%%%%%%%%%%%%%%%%%%%%%%%%%%%%%%%%%%%%%%%%%%%%%%%%%%%%%%%%%%%%%%%%%%%%%%%%%%%%%%%%%%%%%%%%%%%%%%%%%%%%%%%%%%%%%%%%%%%%%%%%%%%%%%%%%%%%%%%%%%%%%%%%%%%%%%%%%%%%%%%%%

\subsection{Extended cone complexes}
Consider the intervals $S=(0,1]$ and $S_0=[0,1]$, which are both multiplicative monoids endowed with the real topology. Let $M$ be a fine and saturated monoid. The \emph{cone of $M$} is the homomorphism set $\sigma_M=\Hom(M,S)$ endowed with the compact-open topology where we regard $M$ as a discrete monoid. Note that $\sigma_M$ is indeed a cone in the $\R$-vector space $\Hom(M^\gp,\R_{>0})$ where $\R$ acts on $\R_{>0}$ via the exponential map. 

Let $F$ be a fine and saturated Kato fan and $V'\subset V$ open affine Kato subfans where $V'=\Spec^KM'$ and $V=\Spec^KM$ for some fine, saturated and sharp monoids $M'$ and $M$. Then $M'=S^{-1}M/(S^{-1}M)^\times$ for some multiplicative set $S$ in $M$, which implies that $\sigma_{M'}$ is a face of $\sigma_M$. Let $\Delta$ be the diagram of all cones $\sigma_M$ such that $\Spec^KM$ is an affine open in $F$ together with the face maps $\sigma_{M'}\subset \sigma_M$ for which $\Spec^KM'\subset\Spec^KM$. We define the \emph{cone complex $\Sigma_F$ of $F$} as the colimit of $\Delta$ as a topological space. As a point set $\Sigma_F$ is equal to $\Hom(\Spec^K S,F)$.

Similarly, we define the \emph{extended cone of $M$} as $\bar\sigma_M=\Hom(M,S_0)$ endowed with the compact-open topology. The affine Kato subfans of $F$ yield a diagram $\overline\Delta$ of extended cones and face maps. We define the \emph{extended cone complex $\overline\Sigma_F$ of $F$} as the colimit of $\overline\Delta$ as a topological space. As a point set $\overline\Sigma_F$ is equal to $\Hom(\Spec^K S_0,F)$.

%%%%%%%%%%%%%%%%%%%%%%%%%%%%%%%%%%%%%%%%%%%%%%%%%%%%%%%%%%%%%%%%%%%%%%%%%%%%%%%%%%%%%%%%%%%%%%%%%%%%%%%%%%%%%%%%%%%%%%%%%%%%%%%%%%%%%%%%%%%%%%%%%%%%%%%%%%%%%%%%%%%%%%%%%%%

\subsection{Log schemes}
Let $X$ be a $k$-scheme. A \emph{pre-logarithmic structure for $X$} is a sheaf of monoids $\cM_X$ on $X$ together with a morphism $\alpha:\cM_X\to \cO_X$ of sheaves of monoids where the structure sheaf $\cO_X$ is regarded as a sheaf of monoids with respect to multiplication. A \emph{logarithmic structure for $X$} is a pre-logarithmic structure $\cM_X$ such that $\alpha:\cM_X\to \cO_X$ induces an isomorphism $\alpha^{-1}\cO_X^\times\to\cO_X^\times$. One can associate to every pre-logarithmic structure $\cM_X$ the logarithmic structure $\cM_X^a$ that is the push-out of the diagram 
\[
 \xymatrix@R=1pc@C=6pc{\alpha^{-1}\cO_X^\times \ar[r] \ar[d]_{\alpha}& \cM_X \\ \cO_X^\times }
\]
in the category of sheaves in monoids.

A \emph{log scheme} is a scheme $X$ together with a logarithmic structure $\cM_X$. Given a morphism $\varphi:X\to Y$ of schemes, we define the \emph{inverse image $\varphi^\ast\cM_X$} of a logarithmic structure $\cM_Y$ on $Y$ as the logarithmic structure associated with $\varphi^{-1}\cM_Y$. 

A \emph{morphism of log schemes} is a morphism $\varphi:X\to Y$ of schemes together with a morphism $\varphi^\flat:\varphi^\ast\cM_Y\to \cM_X$ of sheaves of monoids such that 
\[
 \xymatrix@R=1pc@C=6pc{\varphi^\ast\cM_Y \ar[r]^{\varphi^\flat} \ar[d]_{\varphi^\ast(\alpha_Y)}& \cM_X\ar[d]^{\alpha} \\ \varphi^{-1}\cO_Y \ar[r]^{\varphi^\sharp} & \cO_X}
\]
commutes.

Let $X$ be a log scheme. Given a monoid $M$, we denote by $M_X$ the constant sheaf with value $M$. A \emph{chart for $X$} is a morphism $\beta:M_X\to\cM_X$ of sheaves of monoids such that $\beta^a:M_X^a\to\cM_X$ is an isomorphism. A log scheme $X$ is called \emph{fine and saturated} if it admits a covering by open subschemes $U_i$ with charts $\beta_i:(M_i)_{U_i}\to\cM_{U_i}$ for fine and saturated monoids $M_i$ where $\cM_{U_i}$ denotes the restriction of $\cM_X$ to $U_i$.

\begin{rem}\label{rem: etale log schemes}
 In this exposition, we restrict ourselves to Zariski log schemes, for which we can make the connection to the scheme theoretic tropicalization precise, and we leave an extension of the theory to \'etale structures to future investigations.
\end{rem}

%%%%%%%%%%%%%%%%%%%%%%%%%%%%%%%%%%%%%%%%%%%%%%%%%%%%%%%%%%%%%%%%%%%%%%%%%%%%%%%%%%%%%%%%%%%%%%%%%%%%%%%%%%%%%%%%%%%%%%%%%%%%%%%%%%%%%%%%%%%%%%%%%%%%%%%%%%%%%%%%%%%%%%%%%%%

\subsection{The associated Kato fan}
Recall that $\overline\cM_X=\cM_X/\cM_X^\times$ is the sharp sheaf of monoids associated with $\cM_X$. A log structure $\alpha:\cM_X\to X$ is said to be \emph{without monodromy} if there exists a Kato fan $F$ and a strict morphism $(X,\overline\cM_X)\to F$. This is, for instance, the case if $\alpha:\cM_X\to\cO_X$ is injective. See \cite[Ex.\ B.1]{Gross-Siebert13} and \cite[Ex.\ 4.12]{Ulirsch13} for examples of log structures with monodromy.

The following is the key observation that allows us to define the tropicalization of a fine and saturated log scheme. This is Proposition 4.7 in \cite{Ulirsch13}, though it is essentially already present in \cite{Kato94}.

\begin{prop}\label{prop: the associated Kato fan}
 Let $X$ be a scheme of finite type over $k$ and $\alpha:\cM_X\to\cO_X$ a fine and saturated log structure without monodromy . Then there is a strict morphism $(X,\overline\cM_X)\to F_X$ of sharp monoidal spaces into a Kato fan $F_X$ that is initial for all strict morphisms from $(X,\overline\cM_X)$ into a Kato fan.
\end{prop}

We call $F_X$ the \emph{associated Kato fan of $X$} and the composition 
\[
 \chi_X: \ (X,\overline\cO_X) \quad \longrightarrow \quad (X,\overline\cM_X) \quad \longrightarrow \quad F_X
\]
its \emph{characteristic morphism} where $\overline\cO_X=\cO_X/\cO_X^\times$. We briefly write $\chi_X:X\to F_X$ for the characteristic morphism.

%%%%%%%%%%%%%%%%%%%%%%%%%%%%%%%%%%%%%%%%%%%%%%%%%%%%%%%%%%%%%%%%%%%%%%%%%%%%%%%%%%%%%%%%%%%%%%%%%%%%%%%%%%%%%%%%%%%%%%%%%%%%%%%%%%%%%%%%%%%%%%%%%%%%%%%%%%%%%%%%%%%%%%%%%%%

\subsection{Tropicalization of a fine and saturated log scheme}
Let $X$ be a fine and saturated log scheme over $k$ with log structure $\alpha:\cM_X\to \cO_X$ and characteristic morphism $\chi_X:X\to F_X$. The \emph{Ulirsch tropicalization $\Trop_\alpha^U(X)$ of $X$} is the extended cone complex $\overline\Sigma_X=\overline\Sigma_{F_X}$ together with a surjective continuous map $\trop_\alpha^U:X^\beth\to\overline\Sigma_X$ that is described for affine open subsets as follows. 

Let $U=\Spec R$ be open in $X$ and $V=\Spec^KM$ open in $F_X$ such that $\chi_X(U)\subset V$ and $M_U\to\cM_U$ is a chart. By Lemma 1.6 in \cite{Kato94}, $X$ and $F_X$ can be covered by such open subsets. This yields a morphism of sheaves of monoids $M_U\to\cM_U\to\overline\cO_U$ and, by taking global sections, a multiplicative map $M\to R/R^\times$. Since every valuation $w:R\to \cO_\T$ maps $R^\times$ to $1$ and since $S_0$ is the underlying monoid of $\cO_\T$, we get a continuous map
\[
 \trop_\alpha^U: \ U^\beth \ = \ \Val_v(R,\cO_\T) \quad \longrightarrow \quad \Hom(M,S_0) \ = \ \overline \Sigma_{\Spec^KM}.
\]

If $\iota:Y\to X$ is a closed immersion of $k$-schemes, then the \emph{Ulirsch tropicalization of $Y$ with respect to $\iota$} is the image $\Trop_{\alpha,\iota}^U(Y)=\trop_\alpha^U(Y)$ of $Y$ in $\Trop_\alpha^U(X)$, together with the restriction $\trop_{\alpha,\iota}^U:Y^\beth\to \Trop^U(Y)$ of $\trop_\alpha^U$ to $Y^\beth$.

%%%%%%%%%%%%%%%%%%%%%%%%%%%%%%%%%%%%%%%%%%%%%%%%%%%%%%%%%%%%%%%%%%%%%%%%%%%%%%%%%%%%%%%%%%%%%%%%%%%%%%%%%%%%%%%%%%%%%%%%%%%%%%%%%%%%%%%%%%%%%%%%%%%%%%%%%%%%%%%%%%%%%%%%%%%

\subsection{The associated blue scheme}\label{subsection: blue scheme associated with a fine and saturated log scheme}
Let $X$ be a fine and saturated log scheme $X$ over $k$ without monodromy and $\chi_X:X\to F_X$ the characteristic map into the associated Kato fan $F_X$. Provided that the inverse images $U=\chi^{-1}(V)$ of affine open Kato subfans $V$ of $F_X$ are affine, we obtain a natural system of affine open subschemes of $X$, which we can endow with charts for the log structure. We will use these open subschemes in the definition of the associated blue scheme. 

The following fact is a strengthening of \cite[Lemma 1.6]{Kato94} and \cite[Prop.\ 4.7]{Ulirsch13} under this additional assumption on $U$.

\begin{lemma}\label{lemma: maximal charts for a fine and saturated log scheme}
 Let $V=\Spec^KM_V$ an affine open of $F_X$. Assume that $U=\chi_X^{-1}(V)$ is affine. Then there is a chart $(M_V)_U\to \cM_U$ such that the composition $M_V\to\cM_U(U)\to  \overline\cM_U(U) = M_V$ is the identity on $M_V$.
\end{lemma}

\begin{proof}
 Since $M_V$ is a fine monoid, it embeds into $M_V^\gp$, which is a finitely generated abelian group. If $a^n=1$ for an element $a\in M_V^\gp$ and $n\geq1$, then we have $a\in M_V$ since $M_V$ is saturated. Since $M_V$ is sharp, we conclude that $a=1$, which shows that $M_V^\gp$ is torsion free and hence a free abelian group of finite rank.
  
 Therefore the multiplicative map $\cM_U(U)\to \overline\cM_U(U) = M_V$ admits a section $s:M_V\to \cM_U(U)$. Since for an open subset $V'=\Spec^KM_{V'}$ of $V$, the restriction map $\rho:M_V\to M_{V'}$ is of the form $M_V\to S^{-1}M_V \to S^{-1}M_V/(S^{-1}M_V)^\times=M_{V'}$ for $S=\{h\in M_V|\rho(h)=1\}$, the section $s$ extends uniquely to a section $s_U:(M_V)_U\to\cM_U$ of sheaves of monoids. 
 
 We are left with showing that $s_U$ is a chart. Clearly, $s_U^a:(M_V)_U^a\to \cM_U$ is a monomorphism of sheaves. It is surjective on stalks since $M_V= \overline\cM_U(U)=\cM_U(U)/\cO_U(U)^\times$ and therefore $M_V\to \overline\cM_{U,x}$ is a surjection for every $x\in U$. Thus $s_U$ is a chart.
\end{proof}

For the rest of this section, we assume that the inverse image $U=\chi^{-1}(V)$ of any affine open $V$ of $F_X$ is affine. 

Let $\iota:Y\to X$ be a closed immersion of $k$-schemes and $\alpha:\cM_X\to \cO_X$ the structure morphisms of sheaves of monoids on $X$. For every affine open $V=\Spec^KM_V$ of $F_X$ and the inverse images $U=\chi_X^{-1}(V)$ and $W=\iota^{-1}(U)$, we obtain the morphism of sheaves of monoids
\[
 \iota^\ast\alpha^\ast\cM_U \quad \longrightarrow \quad \iota^\ast \cO_U \quad \longrightarrow \quad \cO_W.
\]
Taking global sections yields a multiplicative map $\eta_V:\cM(U)\to \cO_W(W)=S_V$. We define the blueprint $B_V=\bpquot{A_V}{\cR_V}$ where
\[\textstyle
 A_V \ = \ \eta_V\bigl(\cM(U)\bigr) \cup \{0\} \qquad \text{and} \qquad \cR_V \ = \ \left\{ \ \sum a_i\=\sum b_j \ \left| \ \sum a_i=\sum b_j\text{ in }S_V \ \right.\right\}.
\]

\begin{lemma}\label{lemma: an inclusion of affine Kato fans yields a finite localization of blueprints}
 An inclusion $V'\subset V$ of open affine Kato subfans of $F_X$ induces a finite localization $B_V\to B_{V'}$.
\end{lemma}

\begin{proof}
 Let $V=\Spec^K M_V$ and $V'=\Spec^KM_{V'}$. Let $U=\chi_X^{-1}(V)$ and $U'=\chi_X^{-1}(V')$. Then the characteristic map $\chi_X$ yields the identifications $\overline\cM_X(U)=M_V$ and $\overline\cM_X(U')=M_{V'}$, and we obtain projections $\pi_V:\cM_X(U)\to M_V$ and $\pi_{V'}:\cM_X(U')\to M_{V'}$.
 
 The inclusion $U'\subset U$ yields the restriction map $\rho:\cM(U)\to\cM(U')$ and the inclusion $V'\subset V$ yields the restriction map $\bar\rho:M_V \to M_{V'}$.  Let $\overline S$ be $\bar\rho^{-1}(1)$ and $S=\pi_U^{-1}(\overline S)$. Since $\bar\rho\circ\pi_U=\pi_{U'}\circ\rho$, the image $\rho(S)$ is contained in the fibre $\pi_{U'}^{-1}(1)$, which is the set $\cM_X(U')^\times$ of invertible elements. Therefore $\rho$ induces a multiplicative map $S^{-1}\cM_X(U)\to\cM_X(U')$. 
 
 We claim that this map is an isomorphism of monoids. By Lemma \ref{lemma: maximal charts for a fine and saturated log scheme}, there are sections $M_V\to \cM_X(U)$ and $M_{V'}\to\cM_X(U')$ to the projections $\pi_V$ and $\pi_{V'}$, respectively, such that $S^{-1}\cM_X(U)=S^{-1}\cO_X(U)^\times M_V$ and $\cM_X(U')=\cO_X(U')^\times M_{V'}$. Thus the claim follows if we can show that $S_\alpha^{-1}\cO_X(U)\to \cO_X(U')$ is an isomorphism where $S_\alpha=\alpha(U)(S)$ is the image of $S$ in $\cO_X(U)$. This can be verified geometrically, i.e.\ it suffices to show that $\Spec S_\alpha^{-1}\cO_X(U)\subset U'$. The former open subset of $U$ consists of all points $x\in U$ such that $S\subset \cM(U)$ is mapped to the units $\cM_{X,x}^\times=\cO_{X,x}^\times$ of the stalk of $\cM_{X}$ in $x$. This means that $S$ is sent to $\{1\}$ in $M_{V'}=\bar S^{-1}M_V/\bar S^\gp$. Therefore $x\in U'$, which shows that $S^{-1}\cM_X(U)\to\cM_X(U')$ is an isomorphism.
 
 We conclude that the natural map $B_V\to B_{V'}$ induces an isomorphism $\eta_V(S)^{-1}A_V\to A_{V'}$ of monoids. It is an isomorphism $\eta_V(S)^{-1}B_V\to B_{V'}$ of blueprints since $R_{V'}=\eta(S)^{-1}R_V$ and consequently the preaddition of $B_{V'}$ is generated by the preaddition of $B_V$. This concludes the proof of the lemma.
\end{proof}

Since $\cM$ is a logarithmic structure for $X$, the monoid $\cM(U)$ contains $k^\times$. Therefore the monoid $\eta_V(\cM(U))\cup\{0\}$ contains $k$, i.e.\ $A_V$ is a blue $k$-algebra, and $\Spec B_V$ is an affine blue $k$-scheme. By Lemma \ref{lemma: an inclusion of affine Kato fans yields a finite localization of blueprints}, the diagram $\cU$ of all morphisms $\Spec B_{V'}\to\Spec B_V$ where $V'\subset V$ are open affine Kato subfans of $F_X$ forms a commutative diagram $\cU$ of affine blue $k$-schemes and open immersions.

We define the \emph{blue $k$-scheme associated with $\alpha:\cM_X\to\cO_X$ and $\iota:Y\to X$} as $Z=\colim \cU$. It comes with a morphism $\beta:Y\to Z$ of blue $k$-schemes and a morphism $\bar\eta:\iota^\ast\cM_X\to \beta^\ast\cO_Z$ of log structures for $Y$ where $\beta^\ast\cO_Z$ is the log structure associated with the pre-log structure $\beta^\sharp:\beta^{-1}\cO_Z\to\cO_Y$. This means that the diagram 
\[
 \xymatrix@C=4pc@R=1pc{\iota^\ast\cM_X \ar[rr]^{\bar\eta} \ar[dr]_{\iota^\ast(\alpha)} && \beta^\ast\cO_Z \ar[dl]^{\beta^\sharp} \\ & \cO_Y}
\]
of morphisms of sheaves of monoids on $Y$ commutes. We define $Z^{+\blue}$ as the colimit of the diagram $(\cU^+)^\blue$ of blueprints.

\begin{thm}\label{thm: Ulirsch tropicalization as rational point set}
 Let $\alpha:\cM_X\to\cO_X$ be a log structure for $X$ with characteristic map $\chi:X\to F_X$. Assume that for every affine open $V$ of $F_X$, the inverse image $U=\chi^{-1}(V)$ is affine. Let $\iota:Y\to X$ a closed immersion of $k$-schemes and $Z$ and $Z^{+\blue}$ be the associated blue $k$-schemes, as defined above. Then the Ulirsch tropicalization $\Trop_{\alpha,\iota}^U(Y)$ is naturally homeomorphic to $\Bend_v(Z)(\cO_\T)$ and the diagram
 \[
  \xymatrix@R=1pc@C=6pc{Y^\beth \ar[d]^\simeq\ar[r]^{\trop_{\alpha,\iota}^U(Y)} & \Trop_{\alpha,\iota}^U(Y)\ar[d]^\simeq \\ \Bend_v(Z^{+\blue})(\cO_\T) \ar[r]^{\Bend_v(\beta)(\cO_\T)} & \Bend_v(Z)(\cO_\T)}
 \]
 of continuous maps commutes.
\end{thm}

\begin{proof}
 Since all functors and maps in question are defined locally, we can assume that $X=\Spec R$ is an affine fine and saturated log scheme with affine Kato fan $F_X=\Spec^KM$. Then the closed immersion $\iota:Y\to X$ corresponds to a surjection $R\to S$ of rings. Let $\eta:\cM_X(X)\to R\to S$ be the multiplicative map that is induced by the log structure $\alpha:\cM_X\to \cO_X$ and define $A=\eta(\cM_X(X))\cup\{0\}$ and the preaddition $\cR=\{\sum a_i\=\sum b_j|\sum a_i=\sum b_j\text{ in }S\}$ on $A$. Then the blue $k$-scheme associated with $\alpha$ and $\iota$ is $Z=\Spec B$ for the blueprint $B=\bpquot{A}{\cR}$.
 
 Recall that the tropicalization of $Y$ is defined as the image of $Y^\beth$ under $\trop_\alpha^U:X^\beth\to\Trop_\alpha^U(X)$ together the restriction of $\trop_\alpha^U$ to $Y^\beth$. Together with the identifications that we used to define the tropicalization map, we obtain the commutative diagram
 \[
  \xymatrix@R=1,5pc@C=2pc{   &                                                                                  &&& \Val_v(B,\cO_\T) \ar@{..>}[d]^\Psi \\
    Y^\beth \ar@{}[r]|(0.4)= & \Val_v(S,\cO_\T) \ar[urrr]^{j^\ast} \ar[rrr]^{\trop_{\alpha,\iota}^U} \arincl[d] &&& \im\bigl( \trop_\alpha^U(Y) \bigr) \ar@{}[r]|(0.55)= \arincl[d] & \Trop^U(Y) \\
    X^\beth \ar@{}[r]|(0.4)= & \Val_v(R,\cO_\T) \ar[rrr]^{\trop_{\alpha}^U}                                     &&& \Hom(M,\cO_\T) \ar@{}[r]|(0.55)=                                & \Trop^U(X)     }
 \]
 where $j^\ast$ is induced by the inclusion $j:B\to S$. In the following, we will define the dotted arrow $\Psi$ in the diagram and show that it is a bijection that commutes with the other maps of the diagram.
 
 Note that $\eta$ restricts to a multiplicative map $\cM_X(X)\to B$, which induces a multiplicative map $\bar\eta:M=\overline\cM_X(X)\to B/B^\times$. Let $w:B\to \cO_\T$ be a valuation that extends the trivial valuation $v:k\to \cO_\T$. We define the multiplicative map $\Psi(w):M\to \cO_\T$ as $a\mapsto w(a')$ where $a'\in B$ is a representative of $\bar\eta(a)\in B/B^\times$. Note that the definition $\Psi(w)$ is independent from the choice of $a'$ since $w$ sends $B^\times$ to $\cO_T^\times=\{1\}$. This defines $\Psi$ as a map from $\Val_v(B,\cO_\T)$ to $\Hom(M,\cO_\T)$. This map is injective since $\bar\eta(M)=B/B^\times-\{0\}$ and every valuation $w:B\to\cO_\T$ maps $0$ to $0$. It commutes with $\trop_\alpha^U$ and $j^\ast$ since the latter maps are defined as the restrictions of valuations $w:R\to\cO_\T$ to $M$ and $B$, respectively.
 
 The next step is to show that $j^\ast$ is surjective. Since $\trop_\alpha^U:\Val_v(R,\cO_\T)\to \Hom(M,\cO_\T)$ is surjective, every valuation $w:B\to \cO_\T$ extends to a valuation $w':R\to\cO_\T$. Since the preaddition of $B$ contains all relations of $S$, $w'$ factors through $S$ and defines a valuation $w'':S\to\cO_\T$. Thus $j^\ast(w'')=w$, which shows that $j^\ast$ is surjective. As a consequence, $\Psi(w)=\trop_{\alpha,\iota}^U(w'')$ is contained in $\Trop_{\alpha,\iota}^U(Y)$ and $\Psi:\Val_v(B,\cO_T)\to \Trop_{\alpha,\iota}^U(Y)$ is a bijection as claimed.
 
 By Theorem \ref{thm: tropicalization for idempotent base}, we have natural identifications $\Val_v(S,\cO_\T)=\Hom_\T(\Bend_v(S),\cO_\T)$ and $\Val_v(B,\cO_\T)=\Hom_\T(\Bend_v(B),\cO_\T)$. Under these identifications, $j^\ast$ corresponds to $\Bend_v(j)^\ast$, which is equal to $\Bend_v(\beta)(\cO_\T)$. This establishes the commutativity of the diagram of the theorem.
 
 It remains to show that the bijection $\Trop_{\alpha,\iota}^U(Y)\to \Bend_v(Z)(\cO_\T)$ is a homeomorphism. This can be shown by the same technique as used in the proof of Theorem \ref{thm: Kajiwara-Payne tropicalization as rational point set}, and we leave out the details.
\end{proof}

\begin{rem}
 In some sense, the construction of the associated blue scheme is reverse to the association of a log scheme with a blue scheme, as considered in \cite[section 7.3]{blueprintedview}. In the following section \ref{subsection: recovering the Kato fan}, we give this a precise meaning in terms of a universal property satisfied by the associated blue scheme.
\end{rem}

\begin{ex}[Toric varieties]
 Let $\Delta$ be a fan in $N_\R$ and $X(\Delta)$ be the associated toric $k$-variety. For a cone $\tau$ in $\Delta$, let $M_\tau=\tau^\vee\cap N_\Z^\vee$ be the associated monoid, cf.\ section \ref{subsection: toric varieties}. 
 
 The associated log structure $\alpha:\cM_X\to\cO_X$ can be described as follows. For $\tau$ in $\Delta$, we define $\cM_{X_\tau}=(M_\tau)_{X_\tau}^a$ where the pre-log structure $(M_\tau)_{X_\tau}\to\cO_{X_\tau}$ comes from the natural inclusion $M_\tau\to\Z[M_\tau]^+=\cO_X(X_\tau)$. Consequently, $\cM_X$ comes with the charts $(M_\tau)_{X_\tau}\to\cM_{X_\tau}$, and $(X,\cM_X)$ is a fine and saturated log scheme.
 
 The blue $k$-scheme $Z$ associated with $X(\Delta)$ is defined locally as $Z_\tau=\Spec k[M_\tau]$, and it comes together with a canonical morphism $X(\Delta)\to Z$. By the very definition of the blue $k$-scheme $Z'$ associated with $(X,\cM_X)$, we see that $Z$ and $Z'$ are naturally isomorphic and that the canonical morphisms $\beta:X(\Delta)\to Z$ and $\beta':X(\Delta)\to Z'$ agree. 
 
 Note that in this case, the closed embedding $\iota:Y\to X(\Delta)$ is the identity and that the morphism $\bar\eta:\iota^{-1}\cM_X\to \beta^{-1}\cO_Z$ of sheaves of monoids on $X$ is an isomorphism.
 
 For closed subvarieties $Y\hookrightarrow X(\Delta)$, we get two different blue models $Y^\toric$ and $Y^\log$ for $Y$. While $Y^\toric$ inherits its structure from the blue scheme structure of $Z$ along the closed embedding $Y\hookrightarrow X(\Delta)$, and therefore is a closed blue subscheme of $Z$, the blue scheme $Y^\log$ is defined by the restriction of the boundary divisor of $X(\Delta)$ to $Y$. These two blue models of $Y$ come with a morphism $Y^\log\to Y^\toric$, but this morphism is not an isomorphism in general.
 
 An example where $Y^\toric$ and $Y^\log$ do not agree is the following. Consider the quadric $Y$ defined by $x^2+y^2+z^2$ in $\P^{2,+}_k=\Proj k[x,y,z]^+$ with the canonical toric structure. Then $Y^\toric$ is the closed blue subscheme $\Proj \bpgenquot{k[x,y,z]}{x^2+y^2+z^2}$ of $\P^2_k$. The restriction $D_Y$ of the boundary divisor $D$ of $\P^{2,+}_k$ to $Y$ consists of the six points of $Y$ that lie on the intersection of $Y$ with $D$. Therefore the Kato fan of the log structure on $Y$ associated with $D_Y$ has six closed points and does not embed into the Kato fan for the canonical log structure of $\P^{2,+}_k$, which has only three closed points. Consequently $Y^\log$ does not embed into $\P^2_k$.
\end{ex}

\begin{ex}[Canonical log structure of a divisor]
 Let $X$ be an integral $k$-scheme of finite type and $H=H_1+\dotsb+H_r$ a Weil divisor where the $H_i$ are pairwise coprime codimension one $k$-subschemes of $X$. Let $\alpha:\cM_X\to\cO_X$ be the canonical log structure of $H$. In the following, we give an explicit description of the associated blue $k$-scheme $Z$. Note that this generalizes the previous example.

 Define $I=\{1,\dotsc,r\}$ and let $U_J$ be the complement of $\bigcup_{i\in J}H_i$ in $X$ for $J\subset I$. Note that $U_J\subset U_{J'}$ for $J'\subset J$, and in particular $U_I\subset U_J$ for all $J$. Define the blueprint $B_J=\bpquot{A_J}{\cR_J}$ as 
 \[\textstyle
  A_J \ = \ \cO_X(U_J)\cap\cO_X(U_I)^\times \quad \text{and} \quad \cR_J \ = \ \left\langle \ \sum a_i\=\sum b_j \ \left| \ \sum a_i=\sum b_j \text{ in } \cO_X(U_J) \ \right.\right\rangle.
 \]
 We obtain induced morphisms $B_{J'}\to B_J$ for $J'\subset J$. Let $\cS$ be the set of subsets of $J\subset I$ such that $B_J\to B_I$ is a finite localization. We denote by $\cV$ the diagram of affine blue $k$-schemes $Z_J=\Spec B_J$ with $J\in\cS$ together with the open immersions $\Spec B_J\to\Spec B_{J'}$ for $J'\subset J$. Then the blue $k$-scheme $Z$ associated with the log structure $\alpha$ is the colimit of $\cV$.
\end{ex}

%%%%%%%%%%%%%%%%%%%%%%%%%%%%%%%%%%%%%%%%%%%%%%%%%%%%%%%%%%%%%%%%%%%%%%%%%%%%%%%%%%%%%%%%%%%%%%%%%%%%%%%%%%%%%%%%%%%%%%%%%%%%%%%%%%%%%%%%%%%%%%%%%%%%%%%%%%%%%%%%%%%%%%%%%%%

\subsection{Recovering the Kato fan}\label{subsection: recovering the Kato fan}
If we assume that $\chi^{-1}(U)$ is affine for all affine opens $U\subset F$ and if we assume that $\alpha:\cM_X\to\cO_X$ is a monomorphism, i.e.\ $\cM_X(U)\to\cO_X(U)$ is injective for all open subsets $U$ of $X$, then we can recover the Kato fan $F_X$ and the extended cone complex $\overline\Sigma_X$ from the scheme theoretic tropicalization of a fine and saturated log scheme $Y=X$ over a field $k$ as follows.

Let $\alpha:\cM_X\to \cO_X$ be the log structure of $X$ and $Z$ the associated blue $k$-scheme. Let $v:k\to \cO_\T$ be the trivial valuation, which factors into the trivial valuation $v_0:k\to\B$ followed by the unique inclusion $i:\B\to\cO_\T$. Therefore, we have $\Bend_v(Z)=\Bend_{v_0}(Z)\otimes_\B\cO_\T$. Conversely, any morphism $p:\cO_\T\to\B$ induces the identification $\Bend_{v_0}(Z)=\Bend_v(Z)\otimes_{\cO_\T}\B$ since $v:k\to\cO_\T$ has image $\{0,1\}$. 

Let $Z_0=\Bend_{v_0}(Z)$, and let $Z_0^\bullet$ be the underlying monoid scheme, as defined in Theorem \ref{thm: endofunctors of ordered blue schemes}. Note that $Z_0^\bullet$ is integral. Therefore we can endow $Z_0^\bullet$ with the sheaf $\overline\cO_{Z_0^\bullet}$ of strict monoids $\overline\cO_{Z_0^\bullet}(V) =(\cO_{Z_0}(V)^\bullet-\{0\})/\cO_{Z_0^\bullet}(V)^\times$ where $V\subset Z_0$ is open.

\begin{thm}\label{thm: recovering the Kato fan}
 The monoidal space $(Z_0^\bullet,\overline\cO_{Z_0^\bullet})$ is naturally isomorphic to the Kato fan $F_X$ of $X$. Consequently, $\Bend_{v}(Z)(\cO_\T)=F_X(\cO_\T)$ comes with the structure of an extended cone complex.
\end{thm}

\begin{proof}
 It suffices to verify this theorem in the affine case. Thus let $X=\Spec R$, $Z=\Spec B$ and $F_X=\Spec^K M$.
 
 Then $\Bend_{v_0}(B)=\bpquot{(B^\bullet\otimes_{k^\bullet}\B)}{\bend_{v_0}(B)}$. If the bend relation contains $a\=b$, then this relation is a sequence of generators of $\bend_{v_0}(B)$. In particular, $\bend_{v_0}(B)$ must contain a generator of the form $\sum c_j\=b$, which must be of the form $b+c\=b$ and come from $b\leq c$. Since $B$ is algebraic, we have $b=c$ in $B$. This shows that the underlying monoid of $\Bend_{v_0}(B)$ is $B^\bullet\otimes_{k^\bullet}\B$.
 
 Since $v_0$ is surjective with fibres $v_0^{-1}(0)=\{0\}$ and $v_0^{-1}(1)=k^\times$, we have $B^\bullet\otimes_{k^\bullet}\B^\bullet\simeq B^\bullet/k^\times$. It is easily verified that this implies that the association $\fp\mapsto\pi^{-1}(\fp)$ defines a homeomorphism between the prime spectra of $\Bend_{v_0}(B)^\bullet= B^\bullet\otimes_{k^\bullet}\B^\bullet$ and $B^\bullet$.
 
 The association $\fp\mapsto\fp\cup\{0\}$ defines a homeomorphism between the affine Kato fan of $\cM_X(X)$ and the prime spectrum $B^\bullet=\cM_X(X)\cup\{0\}$. Similar to the surjection $\pi$, the surjection $\cM_X(X)\to M=\cM_X(X)/\cM_X(X)^\times$ induces a homeomorphism between the respective affine Kato fans.

 The composition of these homeomorphisms yields a homeomorphism $(Z_0^\bullet,\overline\cO_{Z_0^\bullet})\to F_X$. We have 
 \[
  M \ = \ \cM_X(X)/\cM_X(X)^\times \ = \ (B^\bullet-\{0\})/B^\times \ = \ (\Bend_{v_0}(B)^\bullet-\{0\})/\Bend_{v_0}(B)^\times,
 \]
 which shows that the monoids of global sections coincide. Since for both monoidal spaces $(Z_0^\bullet,\overline\cO_{Z_0^\bullet})$ and $F_X$, local sections are defined in terms of localizations modulo units of $M$ and $\Bend_{v_0}(B)^\bullet$, respectively, we conclude that $(Z_0^\bullet,\overline\cO_{Z_0^\bullet})$ is naturally isomorphic to $F_X$ as a monoidal space.
 
 The second claim of the theorem follows since the structure of $\overline\Sigma_X=F_X(\cO_\T)$ as an extended cone complex is induced by the topology of $F_X=(Z_0^\bullet,\overline\cO_{Z_0^\bullet})$.
\end{proof}

%%%%%%%%%%%%%%%%%%%%%%%%%%%%%%%%%%%%%%%%%%%%%%%%%%%%%%%%%%%%%%%%%%%%%%%%%%%%%%%%%%%%%%%%%%%%%%%%%%%%%%%%%%%%%%%%%%%%%%%%%%%%%%%%%%%%%%%%%%%%%%%%%%%%%%%%%%%%%%%%%%%%%%%%%%%

\subsection{What is new?}
\label{subsection: What is new for Ulirsch?}

The interpretation of the Ulirsch tropicalization as a scheme theoretic tropicalization enhances the topological spaces $\Trop_{\alpha,\iota}^U(Y)$ with a scheme structure, which was bound to subvarieties of toric varieties in terms of the Giansiracusa tropicalization so far. In particular, the scheme theoretic tropicalization endows the Ulirsch tropicalization intrinsically with a topology, which allows us to detach the Ulirsch tropicalization from its ambient extended cone complex. 

A posteriori, we can recover the extended cone complex via the natural identification of the Kato fan with the prime spectrum of the underlying monoid scheme of the scheme theoretic tropicalization.

Besides these structural improvements of the Ulirsch tropicalization, we observe that the associated blue scheme $Z$ of a fine and saturated log scheme $X$ can be tropicalized along any valuation $v:k\to \T$. This yields an immediate answer to some questions posed in section 9.1 of the overview paper \cite{Abramovich-Chen-Marcus-Ulirsch-Wise15} by Abramovich, Chen, Marcus, Ulirsch and Wise.

%%%%%%%%%%%%%%%%%%%%%%%%%%%%%%%%%%%%%%%%%%%%%%%%%%%%%%%%%%%%%%%%%%%%%%%%%%%%%%%%%%%%%%%%%%%%%%%%%%%%%%%%%%%%%%%%%%%%%%%%%%%%%%%%%%%%%%%%%%%%%%%%%%%%%%%%%%%%%%%%%%%%%%%%%%%
%%%%%%%%%%%%%%%%%%%%%%%%%%%%%%%%%%%%%%%%%%%%%%%%%%%%%%%%%%%%%%%%%%%%%%%%%%%%%%%%%%%%%%%%%%%%%%%%%%%%%%%%%%%%%%%%%%%%%%%%%%%%%%%%%%%%%%%%%%%%%%%%%%%%%%%%%%%%%%%%%%%%%%%%%%%

\bibliographystyle{plain}
%\bibliography{tropical}

%\begin{comment}

%\end{comment}

%\listoftodos

\end{document}